\documentclass[12pt,reqno]{amsart}
\usepackage{graphicx,subfigure,cases}
\usepackage{hyperref}
\newcommand{\vertiii}[1]{{\left\vert\kern-0.25ex\left\vert\kern-0.25ex\left\vert #1
    \right\vert\kern-0.25ex\right\vert\kern-0.25ex\right\vert}}
\hypersetup{colorlinks=true, citecolor=blue, linkcolor=red}

\numberwithin{equation}{section} \numberwithin{figure}{section}
\numberwithin{table}{section} 
{ \theoremstyle{plain}
\newtheorem{theorem}{Theorem}[section]
\newtheorem{proposition}[theorem]{Proposition}
\newtheorem{lemma}[theorem]{Lemma}
\newtheorem{cor}[theorem]{Corollary}
\newtheorem{rem}[theorem]{Remark}

}

\def\beq{\begin{equation}}
\def\eeq{\end{equation}}

\def\R{\mathbb{R}}

\def\C{\mathbb{C}}

\def\ep{\varepsilon}
\def\eps{\varepsilon}
\def\la{\lambda}

\begin{document}

\title[]{Thomas-Fermi approximation for coexisting two component Bose-Einstein condensates and nonexistence of vortices for small rotation}

\author{Amandine Aftalion}
\address{CNRS UMR 8100, Laboratoire de Math\'ematiques de Versailles, Universit\'{e} de Versailles Saint-Quentin, 45
avenue des Etats-Unis, 78035 Versailles C\'edex, France.}
\email{amandine.aftalion@uvsq.fr}

\author{Benedetta Noris}
\address{INdAM-COFUND Marie Curie Fellow, Laboratoire de Math\'ematiques de Versailles, Universit\'{e} de Versailles Saint-Quentin, 45 avenue des Etats-Unis, 78035 Versailles C\'edex, France.}
\email{benedettanoris@gmail.com}

\author{Christos Sourdis}
\address{Department of Applied Mathematics and Department of Mathematics, University of
Crete, GR--714 09 Heraklion, Crete, Greece.}
\email{csourdis@tem.uoc.gr}

\date{\today}

\begin{abstract}
We study minimizers of a Gross--Pitaevskii energy
describing  a two-component Bose-Einstein condensate
 confined in a radially symmetric harmonic trap and set into
rotation. We consider the case of coexistence
of the components in the Thomas-Fermi regime, where a small
parameter $\ep$ conveys a singular perturbation.   The
minimizer of the energy without rotation is determined as the
positive solution of a system of coupled PDE's for which we show uniqueness. The
limiting problem for $\ep =0$ has degenerate and irregular behavior at  specific
radii, where the gradient blows up.
 By means of a perturbation argument, we obtain precise estimates
for the convergence of the minimizer to this limiting profile, as $\ep$ tends to 0. For low rotation, based
on these estimates, we can show that the ground states remain real valued and do not have vortices, even
 in the region of small density.
\end{abstract}

\maketitle

\section{Introduction}
\subsection{The problem}

In this paper, we study the behavior of the minimizers of the
following energy functional describing a two component Bose-Einstein condensate
\begin{equation}\label{eqEnergyOmega}
\begin{split}
E^\Omega_\ep(u_1,u_2)=\sum_{j=1}^2 \int_{\R^2} \left\{ \frac{|\nabla u_j|^2}{2} + \frac{|x|^2}{2\ep^2}|u_j|^2 +\frac{g_j}{4\ep^2}|u_j|^4 -\Omega x^{\perp}\cdot(i u_j,\nabla u_j) \right\} \,dx \\
+\frac{g}{2\ep^2} \int_{\R^2} |u_1|^2|u_2|^2 \,dx
\end{split}
\end{equation}
 in the space
\begin{equation}\label{H_space_definition}
\mathcal{H}=\left\{(u_1,u_2):\ u_j\in H^1(\R^2,\C), \
\int_{\R^2}|x|^2|u_j|^2\,dx<\infty, \ \|u_j\|_{L^2(\R^2)}=1,\
j=1,2\right\}.
\end{equation}
The parameters $g_1,g_2,g,\ep$ and $\Omega$ are positive: $\Omega$ is the
angular velocity corresponding to the rotation of the condensate, $x^\perp=(-x_2,x_1)$ and $\cdot$ is the
scalar product for vectors, whereas $(\ ,\ )$ is the complex scalar product, so
that  we have
\[
x^{\perp}\cdot(iu,\nabla u)=x^{\perp}\cdot\frac{iu\nabla\bar u-i\bar
u\nabla u}{2}= -x_2\frac{iu\partial_{x_1}\bar u-i\bar
u\partial_{x_1}u}{2}+x_1\frac{iu\partial_{x_2}\bar u-i\bar
u\partial_{x_2}u}{2}.
\]
 Here, $g_j$ is the self interaction of each component (intracomponent coupling) while $g$ measures the effect of interaction between the two components (intercomponent coupling).
 We are interested in studying the existence and behavior of the
minimizers in the limit when $\ep$ is
small, describing strong interactions, also called the Thomas-Fermi limit. We  assume the condition
\begin{equation}\label{eq:condition_on_g_thomas_fermi}
g^2<g_1g_2,
\end{equation}
which  implies that the two components $u_1$ and $u_2$ of the
minimizers can coexist, as opposed to the segregation case
$g^2>g_1g_2$. Additionally, we can assume without loss of generality
that
\begin{equation}\label{eq:condition_R_1<R_2}
0<g_1\leq g_2.
\end{equation}

We start with the analysis of the minimizers of the energy functional
$E_\varepsilon^0$ without rotation, namely with $\Omega=0$.
 Up to multiplication by a complex number
 of modulus 1, the minimizers
$(\eta_{1,\ep},\eta_{2,\ep})$ of $E_\varepsilon^0$ in $\mathcal{H}$
 are positive solutions of the following system of coupled
Gross--Pitaevskii equations:
\begin{subequations}\label{eq:main_eta_1_eta_2}
\begin{align} \label{eq:main_eta_1_eta_2a}
&-\ep^2 \Delta \eta_{1,\ep}+|x|^2\eta_{1,\ep}+g_1|\eta_{1,\ep}|^2\eta_{1,\ep}+g\eta_{1,\ep}|\eta_{2,\ep}|^2=\la_{1,\ep}\eta_{1,\ep} \quad  \text{ in } \R^2, \\
 \label{eq:main_eta_1_eta_2b}
&-\ep^2 \Delta
\eta_{2,\ep}+|x|^2\eta_{2,\ep}+g_2|\eta_{2,\ep}|^2\eta_{2,\ep}+g|\eta_{1,\ep}|^2\eta_{2,\ep}=\la_{2,\ep}\eta_{2,\ep}
\quad  \text{ in } \R^2,
 \\
& \     \eta_{i,\varepsilon}(x)\to 0 \ \textrm{as}\ \ |x|\to
  \infty,\ i=1,2,& \label{limeqeta}
\end{align}
\end{subequations}
where $\la_{1,\ep}$, $\la_{2,\ep}$ are the Lagrange multipliers due
to the constraints. We will also refer to the positive minimizers as
ground state solutions.
 Formally
setting $\varepsilon=0$ in (\ref{eq:main_eta_1_eta_2}) gives rise to
the nonlinear algebraic system
\begin{equation}\label{eqlinarAlgebr}
\left\{ \begin{array}{ll}
 |x|^2\eta_1+g_1\eta_1^3+g\eta_1\eta_2^2=\la_{1,0}\eta_1 \quad & \text{ in } \R^2, \\
&\\
|x|^2\eta_2+g_2\eta_2^3+g\eta_1^2\eta_2=\la_{2,0}\eta_2 \quad &
\text{ in } \R^2,
\end{array}\right.
\end{equation}
where $\eta_i\geq 0$  satisfy
$\|\eta_i\|_{L^2(\mathbb{R}^2)}=1$, $i=1,2$. In the region where
neither $\eta_i$ is identically zero, this yields the system
\begin{equation}\label{redetai}
\left\{ \begin{array}{ll}
 g_1\eta_1^2+g\eta_2^2=\la_{1,0}-|x|^2, \\
&\\
g\eta_1^2+g_2\eta_2^2=\la_{2,0}-|x|^2.
\end{array}\right.
\end{equation} This leads to  the condition (\ref{eq:condition_on_g_thomas_fermi})
 and the fact that the supports of $\eta_i$ are
   compact sets: more precisely,  the supports of $\eta_i$ are
    either 2 disks or a disk and an annulus.
 The limiting geometry is two disks when
 \begin{equation}\label{eq:condition_on_g_two_disks}
0<g<\frac{g_1+\sqrt{g_1^2+8g_1g_2}}{4}.
\end{equation}
% which comes from the inequality $2g^2-g_1g-g_1g_2<0$.
 This condition
is more restrictive than \eqref{eq:condition_on_g_thomas_fermi}, in
the sense that \eqref{eq:condition_R_1<R_2} and
\eqref{eq:condition_on_g_two_disks} together imply
\eqref{eq:condition_on_g_thomas_fermi}.
   If, on the contrary, we assume that \begin{equation}\label{eqannulus}g>\frac{g_1+\sqrt{g_1^2+8g_1g_2}}{4},\end{equation} then
the limiting configuration consists of a disk and an annulus; in
this case, the assumption $g_1\leq g_2$ implies $g>g_1$. It is helpful to introduce the following quantities:
\begin{equation}\label{eqGammas}
\Gamma_1=1-\frac{g}{g_1}, \qquad \Gamma_2=1-\frac{g}{g_2}, \qquad
\Gamma=1-\frac{g^2}{g_1g_2}.
\end{equation}
Under the assumptions
\eqref{eq:condition_on_g_thomas_fermi}-\eqref{eq:condition_R_1<R_2}
and \eqref{eq:condition_on_g_two_disks}, we  prove that
\begin{equation}\label{eqIntroConverg}\eta_{i,\varepsilon}^2 \to a_i \ \ \textrm{uniformly\ in}\ \mathbb{R}^2 \ \textrm{as} \
\ep\to0,\ i=1,2,
\end{equation}
where $\sqrt{a_1},\ \sqrt{a_2}$ are the solutions to
(\ref{eqlinarAlgebr}) with $L^2$ constraint 1, so that
\begin{equation}\label{eq:a_i_def}
a_1(x)=\left\{ \begin{array}{ll}
a_{1,0}(x), \quad & |x|\leq R_{1,0}, \\
& \\
0, \quad & |x|\geq R_{1,0},
\end{array}\right.
\ \ \ a_2(x)=\left\{ \begin{array}{ll}
a_{2,0}(x), \quad & |x|\leq R_{1,0}, \\
& \\
a_{2,0}(x)+\frac{g}{g_2}a_{1,0}(x), \quad & R_{1,0} \leq |x|\leq R_{2,0}, \\
& \\
0, \quad & |x|\geq R_{2,0},
\end{array}\right.
\end{equation}
with $R_{1,0}\leq R_{2,0}$ determined explicitly in terms of
$g,g_1,g_2$ (see (\ref{eq:R_i_def})), and
\begin{equation}\label{eq:a_i0_def}
a_{1,0}(x)=\frac{\Gamma_2}{g_1\Gamma}(R_{1,0}^2-|x|^2), \qquad
a_{2,0}(x)=\frac{R_{2,0}^2-R_{1,0}^2}{g_2}+\frac{\Gamma_1}{g_2\Gamma}
(R_{1,0}^2-|x|^2),
\end{equation}
\begin{equation}\label{eq:a_20out}
a_{2,0}(x)+\frac{g}{g_2}a_{1,0}(x)=\frac1{g_2} (R_{2,0}^2-|x|^2).
\end{equation}
We note that $R_{1,0}<R_{2,0}$ if $g_1<g_2$ and $a_1\equiv a_2$ if $g_1=g_2$.
 Moreover, we  show that $\la_{i,\ep}\to \la_{i,0}$, $i=1,2$,
where
\begin{equation}\label{eq:lambda_1_def}
\la_{1,0}=\frac{g}{g_2}R_{2,0}^2+\Gamma_2R_{1,0}^2, \qquad \qquad
\la_{2,0}=R_{2,0}^2.
\end{equation}

Because of
\eqref{eq:condition_on_g_thomas_fermi}-\eqref{eq:condition_R_1<R_2},
we always have that $\Gamma$ and $\Gamma_2$ are positive. On the other hand,
$\Gamma_1$ can have either sign:
  if $g<g_1$,  the singular limits $a_i$ consist of two decreasing
functions, and in the case $g>g_1$, $a_2$
 is increasing near the origin and up to $R_{1,0}$ (though it remains strictly
positive under assumption (\ref{eq:condition_on_g_two_disks})) and then decreasing, while $a_1$ is decreasing. If
$g=g_1$, we have that $a_2$ is constant on the ball of radius
$R_{1,0}$. We remark that the first derivatives of $\sqrt{a_1}$ and
$\sqrt{a_2}$ have an \emph{infinite} jump discontinuity across the
circles $|x|=R_{1,0}$ and $|x|=R_{2,0}$ respectively, while the
first derivative of $\sqrt{a_2}$ has a \emph{finite} jump
discontinuity across $|x|=R_{1,0}$ (if $g_1<g_2$).
In particular, neither function belongs to the Sobolev
space $H^1(\mathbb{R}^2)$. Actually, their maximal regularity is
that of the H\"{o}lder space $C^\frac{1}{2}(\mathbb{R}^2)$.

In the case of \eqref{eq:condition_on_g_thomas_fermi}-\eqref{eq:condition_R_1<R_2}
and \eqref{eqannulus}, that is when $a_1$ is supported in a disk and $a_2$ in an annulus, we  also define the corresponding functions
 $a_i$ and prove (\ref{eqIntroConverg}).

%\begin{figure}[htbp]
%\begin{center}
%\includegraphics[width=4.0in]{befign.pdf}
%\caption{The graphs of $\sqrt{a_1}$ (in black) and $\sqrt{a_2}$ (in grey) for $g=2$, $g_1=1$, $g_2=8$.} \label{figai}
%\end{center}
%\end{figure}

Based on the estimates for the convergence in
(\ref{eqIntroConverg}), we will show that for a large range of
velocities  $\Omega$, the minimizers of $E_\varepsilon^\Omega$ in
$\mathcal{H}$ coincide with the minimizers  of $E_\varepsilon^0$,
provided that $\varepsilon>0$ is sufficiently small.

The aim of this paper is threefold: \begin{enumerate}
\item prove the uniqueness of the positive solution $(\eta_{1,\varepsilon},\eta_{2,\varepsilon})$ of
 (\ref{eq:main_eta_1_eta_2}) (given any $\lambda_{1,\varepsilon}$ and $\lambda_{2,\varepsilon}$), and of the minimizer of $E_\varepsilon^0$ in $\mathcal{H}$
 (modulo a constant complex phase),
\item get precise estimates on the convergence, as  $\varepsilon\to 0$, of  $(\eta_{1,\varepsilon},\eta_{2,\varepsilon})$, the
positive minimizer of $E_\varepsilon^0$ in $\mathcal{H}$, to the
singular limit $(\sqrt{a_1},\sqrt{a_2})$ defined in
(\ref{eq:a_i_def}),
\item prove that for $\Omega$ below a critical velocity, the
minimizers of $E_\varepsilon^\Omega$ in $\mathcal{H}$ have no
vortices in $\mathbb{R}^2$, provided that $\varepsilon>0$ is
sufficiently small.
\end{enumerate}

Point 1 relies on the division of two possible positive solutions componentwise, and proving that each  quotient is equal to a constant
of modulus 1.

 Point 2 is the extension to the system of the results of \cite{KaraliSourdisGround}  for a single equation. The idea
 is  to apply a perturbation argument to construct a positive  solution to (\ref{eqsystem}),
 ``near'' $(\sqrt{a_1},\sqrt{a_2})$. Then, the uniqueness
 result in Point 1
allows us to conclude that the constructed solution
%that is provided by the
%perturbation argument
is indeed  the
  ground state. Therefore, we are able to
  obtain precise
asymptotic estimates for the behavior of
 the ground state as $\varepsilon\to 0$.
We emphasize that, even though the system (\ref{eq:main_eta_1_eta_2}) is
 coupled,  we are going to reduce it, at leading order, to two independent Gross-Pitaevskii
 equations. The proof of Point 2, when both condensates are disks, with different techniques, is the topic of a paper in preparation
  by C. Gallo \cite{Gallo}.

  Point 3 relies on fine estimates for the Jacobian from \cite{jerard}. It consists in extending the proof of \cite{AftalionJerrardLetelierJFA} for a single equation to the system, which works well
  once the difficult results of points 1 and 2 have been established.

%%%%%%%%%%%%%%%%%%%%%%%%%%%%%%%%%%%%%%%%%%%%%%%%%%%%%%%%%%%%%%%%%%%%%%%%%%%%%%%%%%%%%%%%%%%%%

\subsection{Motivation and known results}

 Two component condensates
  can describe a single isotope in two different
hyperfine spin states, two different isotopes of the same atom or
isotopes of two different atoms.  We refer to \cite{AftalionMason}
for more details  on the modeling and the experimental references.

According to the respective values of $g_1$, $g_2$, $g$ and
$\Omega$, the minimizers exhibit very different properties in terms
 of shape of the bulk, defects and coexistence of the components or spatial separation, as $\varepsilon\to 0$. In a recent  paper,  Aftalion and Mason \cite{AftalionMason} have produced phase diagrams to
  classify the types of minimizers according to the parameters of the
  problem. Below, we summarize their findings.
   \begin{itemize}
   \item Coexisting condensates with vortex lattices: each condensate is a disk, and, for sufficiently large rotation, displays a vortex lattice.
    The specificity is that each vortex in component 1 creates a peak in component 2 and vice versa.
     It is this
     interaction between peaks and vortices which governs the shape of the vortex lattice.
     For some parameter regimes, the square lattice gets stabilized because
     it has less energy than the triangular lattice \cite{AftalionMasonWei}.
     \item phase separation with radial symmetry: component 1 is a disk while component 2 is an annulus.
    New defects emerge such as giant skyrmions and the presence of peaks inside the annulus, corresponding to vortices in the disk.  \item
    phase separation and complete breaking of symmetry with either droplets  or vortex sheets. \end{itemize}

    It turns out that the sign of the parameter
  $\Gamma$ defined in (\ref{eqGammas}) plays an important  role: if $\Gamma>0$, the two components coexist
   while if $\Gamma<0$, they separate or segregate (case of droplets, vortex sheets). In the case of no rotation,
 the segregation behavior in two component condensates has been  studied by many authors: regularity of the wave function \cite{NoTaTeVe}, regularity of the interface \cite{CaffLin2}, asymptotic behavior near the interface \cite{BeLinWeiZhao,BeTer,dwz}, $\Gamma$-convergence to a Modica-Mortola type energy \cite{AAJRL2013} in the case of a trapped condensate.
 On the other hand, the case of coexistence is the topic of emerging works in terms of vortices: \cite{AftalionMasonWei} for
 a trapped condensate and \cite{ABM} for a homogeneous condensate. The results that are the core of this paper are totally
  new and will allow a much better description of vortices.
 Indeed, in order to understand the behaviour of vortices in a trapped condensate, one has first
    to understand the effect of the trap at leading order on the profile. Therefore, one requires very precise estimates on the
   ground state at $\Omega=0$ for small $\varepsilon$. This is the analogue of what has been obtained for the single
   component case
    that we  now recall.
 Many papers \cite{Aft,AftalionJerrardLetelierJFA,GalloPelinovskyAA,IgnatMillotJFA,KaraliSourdisGround} have studied the one
component analogue of the energy functional (\ref{eqEnergyOmega}),
namely the functional
\begin{equation}\label{eqenergy1comp}
J_\varepsilon^\Omega(u)=\int_{\R^2} \left\{ \frac{|\nabla u|^2}{2} +
\frac{|x|^2}{2\ep^2}|u|^2 +\frac{\gamma}{4\ep^2}|u|^4
-\Omega x^{\perp}\cdot(i u,\nabla u) \right\} \,dx \\
\end{equation}
under the constraint $\int_{\R^2}|u|^2\,dx=1$, where $\gamma$ is
some positive constant.

In the following theorem we have collected various results from
\cite{AftalionJerrardLetelierJFA,GalloPelinovskyAA,KaraliSourdisGround}
(see also Appendix \ref{secAppenGS} herein)  concerning the
minimizers of the energy $J_\varepsilon^0$ without rotation.

\begin{theorem}\label{thmGSNoRotationScalar}
For every $\varepsilon>0$, there exists a unique positive minimizer
$\eta_\varepsilon$ of $J_\varepsilon^0$ with $L^2$ constraint 1, and
any minimizer has the form $e^{i\alpha}\eta_\varepsilon$ for
some $\alpha\in \mathbb{R}$. The minimizer $\eta_\varepsilon$ is
radial and there is a unique pair $(\eta_{\eps}, \lambda_\eps)$ which is a solution of
\begin{equation}\label{eqIntroGSscalarEq}
-\varepsilon^2 \Delta \eta +|x|^2\eta
+\gamma\eta^3=\lambda_\varepsilon \eta,\ x\in \mathbb{R}^2,\
\eta(x)\to 0\ \textrm{as}\ |x|\to \infty,
\end{equation}
with $\eta$ positive. Let
\[
a_0(x)=\left\{\begin{array}{ll}
                \frac{\lambda_0-|x|^2}{\gamma}, & |x|\leq R_0, \\
                  &   \\
                 0, & |x|\geq R_0,
              \end{array}
 \right.
 %\ \ \ a_{0,\varepsilon}(x)=\left\{\begin{array}{ll}
 %               \frac{\lambda_\varepsilon-|x|^2}{\gamma}, & |x|\leq R_{0,\varepsilon}, \\
 %                 &   \\
 %                0, & |x|\geq R_{0,\varepsilon},
 %             \end{array}
 %\right.
\]
where $\lambda_0>0$ is uniquely determined from the condition
$\int_{\mathbb{R}^2}^{}a_0(x)dx=1$ and $R_0=\sqrt{\lambda_0}$.

There exist constants $c,C,\delta>0$, with $\delta<\frac{R_0}{4}$,
such that the following properties hold:
\[
|\lambda_\varepsilon-\lambda_0|\leq C|\log
\varepsilon|\varepsilon^2,
\]
%\[
%\left|\eta_\varepsilon(r)-\varepsilon^\frac{1}{3}\beta_{\varepsilon}
%V\left(\beta_{\varepsilon}
%\frac{r-R_{0,\varepsilon}}{\varepsilon^\frac{2}{3}}
%\right)\right|\leq C\left\{\begin{array}{ll}
%                 \varepsilon+|r-R_{0,\varepsilon}|^\frac{3}{2} & \textrm{if}\ -\delta \leq r-R_{0,\varepsilon}\leq 0, \\
%                   &   \\
%                 \varepsilon \exp\left\{-c\frac{|r-r_{i,\varepsilon}|}{\varepsilon^\frac{2}{3}} \right\} &
%                 \textrm{if} \ 0\leq r-R_{0,\varepsilon}\leq
%                 \delta,
%               \end{array}
%\right.
%\]
%where $V$ is the Hastings-McLeod solution \cite{hastingsbook} to the
%Painlev\'{e}-II equation \cite{fokas}, namely the unique solution to
%the boundary value problem
%\begin{equation}\label{eqpainleve}
%v''=v(v^2+s),\ s\in \mathbb{R};\ v(s)-\sqrt{-s}\to 0\ \textrm{as}\
%s\to -\infty,\ v(s)\to 0 \ \textrm{as}\ s\to \infty,
%\end{equation}
%and
%$\beta_{\varepsilon}=\left(-a_0'(R_{0,\varepsilon})\right)^\frac{1}{3}$;
%\[
%\left|\eta_\varepsilon(r)-\sqrt{a_{0,\varepsilon}(r)} \right|\leq
%C\varepsilon^2 |r-R_{0,\varepsilon}|^{-\frac{5}{2}}\ \ \textrm{if}\
%C\varepsilon^\frac{2}{3}\leq R_{0,\varepsilon}-r\leq \delta;
%\]
%\[
%\left|\eta_\varepsilon(r)-\sqrt{a_{0,\varepsilon}(r)} \right|\leq
%C\varepsilon^2 \ \ \textrm{if}\ 0 \leq r\leq R_{0,\varepsilon}-
%\delta;
%\]
%\[
%\eta_\varepsilon(r)\leq
%C\varepsilon^\frac{1}{3}\exp\left\{-c\frac{r-r_{0,\varepsilon}}{\varepsilon^\frac{2}{3}}\right\},\
%\ r\geq R_{0,\varepsilon}.
%\]
\[
\|\eta_{\varepsilon}-\sqrt{a_0}\|_{L^\infty(\mathbb{R}^2)}\leq
C\varepsilon^\frac{1}{3},
\]
\[
\left|\eta_{\varepsilon}(r)-\sqrt{a_0(r)}\right|\leq
C\varepsilon^\frac{1}{3}\sqrt{a_0(r)},\ \ 0\leq r=|x|\leq
R_{0}-\varepsilon^\frac{1}{3},
\]
\[
\|\eta_{\varepsilon}-\sqrt{a_0}\|_{L^\infty(|x|\leq
R_{0}-\delta)}\leq C|\log \varepsilon|\varepsilon^2,
\]
and
\[
\eta_\varepsilon(r)\leq
C\varepsilon^\frac{1}{3}\exp\left\{-c\frac{r-R_{0}}{\varepsilon^\frac{2}{3}}\right\},\
\ r\geq R_{0},
\]
for sufficiently small $\varepsilon>0$.
\end{theorem}

In fact, given $\lambda>0$, the assertions of the above theorem hold
for the unique positive solution of  equation
(\ref{eqIntroGSscalarEq}) with $\lambda$ in place of
$\lambda_\varepsilon$.
 Using these estimates and a Jacobian estimate from \cite{jerard},
the following theorem is proven in
\cite{AftalionJerrardLetelierJFA}:
\begin{theorem}
Assume that $u_\varepsilon$, $\eta_\eps$ minimize respectively  $J_\varepsilon^\Omega$, $J_\varepsilon^0$ under the constraint of $L^2$ norm 1. There exist $\tilde{\varepsilon}_0, \tilde{\omega}_0,
\tilde{\omega}_1>0$ such that if
$0<\varepsilon<\tilde{\varepsilon}_0$ and $\Omega\leq
\tilde{\omega}_0 |\log \varepsilon|-\tilde{\omega}_1 \log|\log
\varepsilon|$ then $u_\varepsilon=e^{i\alpha}\eta_\varepsilon$ in
$\mathbb{R}^2$ for some constant $\alpha$.
\end{theorem} This leads, in particular, to the uniqueness of the ground state for small $\Omega$.

%%%%%%%%%%%%%%%%%%%%%%%%%%%%%%%%%%%%%%%%%%%%%%%%%%%%%%%%%%%%%%%%%%%%%%%%%%%%%%%%%%%%%%%%%%%%%%%%%%%%%%%%%%%%%%%%%%%%%%%%

\subsection{Main results}

 Our first result concerns uniqueness and radial symmetry for the problem without rotation.
\begin{theorem}\label{thm:uniqueness} Assume that (\ref{eq:condition_on_g_thomas_fermi}) holds.
\begin{enumerate}
                         \item  Let us fix some $\lambda_{i,\varepsilon}>0$, $i=1,2$. Then, the positive solution of (\ref{eq:main_eta_1_eta_2})
                          is unique, if it exists.
                         \item The positive minimizer
$(\eta_{1,\ep},\eta_{2,\ep})$ of $E_\ep^0$ in $\mathcal{H}$ is
unique  and radially symmetric. Every other
minimizer has the form
$(e^{i\alpha_1}\eta_{1,\ep},e^{i\alpha_2}\eta_{2,\ep})$, where
$\alpha_1,\alpha_2$ are constants. If $g_1=g_2$ then
$\eta_{1,\varepsilon}\equiv\eta_{2,\varepsilon}$.
                       \end{enumerate}
                       \end{theorem}

In the case $g_1=g_2$, the system reduces to a single equation. Therefore,
 for the next result, we can
assume that
\begin{equation}\label{eq14+}
g_1<g_2.
\end{equation}

Our main result provides estimates on the convergence of the positive
minimizer $(\eta_{1,\ep},\eta_{2,\ep})$ to its limiting profile
$\left(\sqrt{a_1},\sqrt{a_2} \right)$ as $\varepsilon\to 0$, first in the case of two disks.

\begin{theorem}\label{thmMain} Assume that (\ref{eq:condition_on_g_thomas_fermi}),
 (\ref{eq:condition_on_g_two_disks}) and (\ref{eq14+}) hold. Recall that the $a_i$ are defined by
 (\ref{eq:a_i_def})-(\ref{eq:a_i0_def}) and $\lambda_{i,0}$ by (\ref{eq:lambda_1_def}). Let $(\eta_{1,\ep},\eta_{2,\ep})$ be the positive minimizer of $E_\ep^0$ in $\mathcal{H}$.
  Then,  $(\eta_{1,\ep},\eta_{2,\ep})$ is a solution of (\ref{eq:main_eta_1_eta_2}), for some positive
   Lagrange multipliers $\lambda_{1,\varepsilon}$, $\lambda_{2,\varepsilon}$, and
 there exist constants $c,C,\delta>0$, with
$\delta<\min\left\{\frac{R_{1,0}}{4},\frac{R_{2,0}-R_{1,0}}{4}
\right\}$, such that the following estimates hold:
\begin{equation}\label{eq14Lagr}
|\lambda_{i,\varepsilon}-\lambda_{i,0}|\leq C|\log
\varepsilon|\varepsilon^2,
\end{equation}
\begin{equation}\label{eq:main_theorem_relation2}
\|\eta_{i,\varepsilon}-\sqrt{a_i}\|_{L^\infty(\mathbb{R}^2)}\leq
C\varepsilon^\frac{1}{3},
\end{equation}
\begin{equation}\label{eq14complex}
\left|\eta_{i,\varepsilon}(r)-\sqrt{a_i(r)}\right|\leq
C\varepsilon^\frac{1}{3}\sqrt{a_i(r)},\ \ |x|\leq
R_{i,0}-\varepsilon^\frac{1}{3},
\end{equation}
\begin{equation}\label{eq14bootstrap}
\sum_{i=1}^{2}\|\eta_{i,\varepsilon}-\sqrt{a_i}\|_{L^\infty(|x|\leq
R_{1,0}-\delta)}+\|\eta_{2,\varepsilon}-\sqrt{a_{2}}\|_{L^\infty(R_{1,0}+\delta\leq
|x|\leq R_{2,0}-\delta)}\leq C|\log \varepsilon|\varepsilon^2,
\end{equation}
and
\begin{equation}\label{eq14decay}
\eta_{i,\varepsilon}(r)\leq
C\varepsilon^\frac{1}{3}\exp\left\{-c\frac{r-R_{i,0}}{\varepsilon^\frac{2}{3}}\right\},\
\ r\geq R_{i,0},\ \ i=1,2,
\end{equation}
for sufficiently small $\varepsilon>0$.
\end{theorem}This theorem is the natural extension of what is known for a single condensate described in Theorem \ref{thmGSNoRotationScalar}. The  fine behavior of the minimizer near $R_{1,0}$ and $R_{2,0}$, as
$\varepsilon\to 0$, is established in Theorem \ref{thmLong} in
Section \ref{secPerturb}. It is based on a perturbation argument which proves very powerful
 in this system case, where we have not managed to extend the sub and super solutions techniques of \cite{alama,AftalionJerrardLetelierJFA}.

Based on the above,  we can show the absence of vortices for the
minimizers of $E_\varepsilon^\Omega$ with  small rotation.

\begin{theorem}\label{thm:nonexistence_vortices}
Assume that (\ref{eq:condition_on_g_thomas_fermi}), (\ref{eq:condition_R_1<R_2}) and
 (\ref{eq:condition_on_g_two_disks})  hold. Let $(u_{1,\ep},u_{2,\ep})$ be a minimizer of $E_\ep^\Omega$ in
$\mathcal{H}$. There exist $\ep_0,\omega_0,\omega_1>0$ such that if
$0<\ep<\ep_0$ and $\Omega\leq \omega_0|\log\ep|-\omega_1
\log|\log\ep|$ then $u_{i,\ep}=e^{i\alpha_i}\eta_{i,\ep}$ in $\R^2$
for some constants $\alpha_i$, $i=1,2$.
\end{theorem}

In the case of a disk and an annlus, that is when (\ref{eqannulus}) holds instead of (\ref{eq:condition_on_g_two_disks}), we can prove the equivalent of Theorem \ref{thmMain}. Generalizing Theorem \ref{thm:nonexistence_vortices} is harder, because
 vortices may exist in the central hole of component 2 and this has to be tackled by other techniques than the ones
  in this paper.
\begin{theorem}\label{thmMainann} Assume that (\ref{eq:condition_on_g_thomas_fermi}),
  (\ref{eqannulus})   and (\ref{eq14+}) hold. We define the functions $a_i$  by
 \beq \label{a1ann}
a_1(r)=\left\{\begin{array}{ll}
         a_{1,0}(r)+\frac{g}{g_1}a_{2,0}(r), & 0\leq r \leq R_{2,0}^-, \\
           &   \\
         a_{1,0}(r), & R_{2,0}^-\leq r \leq R_{1,0},\\
           &   \\
         0, & r\geq R_{1,0},
       \end{array}
\right. \hbox{ where } a_{i,0}(r)=\frac{1}{g_i
\Gamma}\left(\lambda_{i,0}-\frac{g}{g_{i+1}}\lambda_{{i+1},0}-\Gamma_{i+1} r^2
\right),\eeq \beq\label{a2ann}
a_2(r)=\left\{\begin{array}{ll}
                0, & 0\leq r\leq R_{2,0}^-, \\
                  &   \\
                a_{2,0}(r),  & R_{2,0}^-\leq r \leq R_{1,0}, \\
                  &   \\
                a_{2,0}(r)+\frac{g}{g_2}a_{1,0}(r), & R_{1,0}\leq r \leq R_{2,0}^+, \\
                  &   \\
                 0 & r\geq R_{2,0}^+.
              \end{array}
 \right.
\eeq and \beq\label{lambdaiann}\lambda_{2,0}=(R_{2,0}^+)^2,\ \lambda_{1,0}-\frac{g}{g_2}\lambda_{2,0}=\Gamma_2
R_{1,0}^2\hbox{ and } \lambda_{2,0}-\frac{g}{g_1}\lambda_{1,0}= \Gamma_1
(R_{2,0}^-)^2\eeq where $\lambda_{1,0}$, $\lambda_{2,0}$
will be given by (\ref{Riann}).
  Let $(\eta_{1,\ep},\eta_{2,\ep})$ be the positive minimizer of $E_\ep^0$ in $\mathcal{H}$.
  Then there exist constants $c,C,\delta>0$,  such that, for sufficiently small $\varepsilon>0$, (\ref{eq14Lagr}), (\ref{eq:main_theorem_relation2})  hold,
    (\ref{eq14complex}) holds for $i=1$ and is replaced, for $i=2$ by
\begin{equation}\label{eq14complexann}
\left|\eta_{2,\varepsilon}(r)-\sqrt{a_2(r)}\right|\leq
C\varepsilon^\frac{1}{3}\sqrt{a_2(r)},\ \hbox{ for } R_{2,0}^-+\varepsilon^\frac{1}{3}<|x|\leq
R_{2,0}^+-\varepsilon^\frac{1}{3},
\end{equation} (\ref{eq14decay}) holds with $R_{2,0}^+$ instead of $R_{2,0}$, and on
fixed compact sets away from $|x|=R_{1,0}$ and $|x|=R_{2,0}^\pm$,
$(\eta_{1,\varepsilon},\eta_{2,\varepsilon})$ is close to
$(\sqrt{a_1},\sqrt{a_2})$ with an error of order
$\mathcal{O}(|\log \varepsilon|)\varepsilon^2$.
\end{theorem} We point out that if $g_2<g_1$, then an analogous theorem holds exchanging the subscript 1 and 2
 in the formulae.

\subsection{Methods of proof and outline of the paper}

 Theorem \ref{thm:uniqueness} is proved by assuming that there are two solutions,
 studying their ratio componentwise and writing the system satisfied
  by the ratio, as inspired by \cite{brezis-oswald,LaMi}. For the first part
   of the theorem, we need decay properties at infinity of the solutions of
 (\ref{eq:main_eta_1_eta_2}), that we prove in a similar way to
 a Liouville theorem in \cite{berestyckiCaffarelli}.  We point out that the system is non-cooperative, and
        the usual moving plane method
      does not seem to apply easily in this case to derive radial symmetry of positive solutions. Nevertheless, our result implies that since positive solutions of
     (\ref{eq:main_eta_1_eta_2}) are unique, they are thus radial.

 For the second part of the theorem,
  we use the decay of finite energy solutions and extra estimates for radial functions.
  A key relation is the following splitting of energy: if $(\eta_{1},\eta_{2})$
   is a ground state among radial functions, then for any $(u_1,u_2)$,
    $E_\ep^0( u_1,
u_2)=E_\ep^0(\eta_1,\eta_2)+F_\ep^0(v_1,v_2)$, where $v_i=
u_i/\eta_i$ and
\begin{equation}\label{ensplit}
%\begin{split}
F_\ep^0(v_1,v_2)=\sum_{i=1}^2 \int_{\R^2} \left\{ \frac{\eta_i^2}{2}|\nabla v_i|^2+\frac{g_i}{4\ep^2}\eta_i^4(|v_i|^2-1)^2 \right\} \,dx
+ \frac{g}{2\ep^2} \int_{\R^2}\eta_1^2\eta_2^2(1-|v_1|^2)(1-|v_2|^2)
\, dx.
%\end{split}
\end{equation} The condition $\Gamma >0$, that is $g>\sqrt{g_1 g_2}$, implies that $F_\ep^0 (v_1,v_2)\geq 0$.
 If we assume that $(u_1,u_2)$ is a ground state of $E_\ep^0$, then using the sign of $F_\ep^0$, we find that $(u_1,u_2)$ is equal, up to a multiplication by a complex number of modulus 1, to $(\eta_1,\eta_2)$. Thus, any ground state is radially
 symmetric and equal, up to a multiplication by a complex number of modulus 1, to $(\eta_1,\eta_2)$.

   \hfill

      Theorem \ref{thmMain} contains a fine asymptotic behaviour of
      the ground state $(\eta_{1,\varepsilon},\eta_{2,\varepsilon})$
       as $\ep$ tends to zero. The difficulty is especially in the regions
       near $R_{1,0}$, $R_{2,0}$ where the approximate inverted parabola matches an
       exponentially small function in a region of size $\ep^{2/3}$.
  The general procedure is to first construct a
sufficiently good approximate solution to the problem (\ref{eq:main_eta_1_eta_2}), for small
$\varepsilon>0$, with coefficients
$\lambda_{1,\varepsilon}$ and $\lambda_{2,\varepsilon}$ being
 {equal to the unique Lagrange multipliers} that are provided by
Theorem \ref{thm:uniqueness}.  Then, using the invertibility properties of the linearized operator about this
approximate solution, we perturb it to a genuine
solution.
 The first
uniqueness result  of Theorem \ref{thm:uniqueness} implies that this constructed
solution coincides with the positive minimizer of $E_0^\varepsilon$
in $\mathcal{H}$. The method is a generalization to the system
 case of the tools developed in \cite{KaraliSourdisGround} for the single equation.

 In order to construct this approximate solution,
we rewrite the system
(\ref{eq:main_eta_1_eta_2}) as
\begin{subequations}\label{eqsystem}
%\begin{cases}
 \begin{align}
 & -\varepsilon^2\Delta  \eta_1+g_1\eta_1\left(\eta_1^2-a_{1,\varepsilon}(x) \right)+g\eta_1\left(\eta_2^2-a_{2,\varepsilon}(x) \right)=0, \label{eqsystem1}\\ \label{eqsystem2}
&-\varepsilon^2\Delta
\eta_2+g_2\eta_2\left(\eta_2^2-a_{2,\varepsilon}(x)
\right)+g\eta_2\left(\eta_1^2-a_{1,\varepsilon}(x) \right)=0,\\ &\eta_i(x)\to 0\ \ \textrm{as}\ \ |x|\to \infty,\ \ i=1,2.\label{eqsystemBdry}
\end{align}
%\end{cases}
\end{subequations}
 where $a_{1,\ep}$, $a_{2,\ep}$ are the $\ep$ equivalent to (\ref{eq:a_i0_def}), that is
\begin{equation}\label{eqa12}
a_{1,\varepsilon}(x)=\frac{1}{g_1
\Gamma}\left(\lambda_{1,\varepsilon}-\frac{g}{g_2}\lambda_{2,\varepsilon}-\Gamma_2|x|^2
\right),\ \ a_{2,\varepsilon}(x)=\frac{1}{g_2
\Gamma}\left(\lambda_{2,\varepsilon}-\frac{g}{g_1}\lambda_{1,\varepsilon}-\Gamma_1|x|^2
\right),
\end{equation}
and
\begin{equation}\label{eqR12}
R_{1,\varepsilon}^2=\frac{1}{\Gamma_2}\left(\lambda_{1,\varepsilon}-\frac{g}{g_2}\lambda_{2,\varepsilon}\right),\
\ R_{2,\varepsilon}^2=\lambda_{2,\varepsilon}.
\end{equation} At leading order,  in the regions where neither $\eta_i$ is close to zero, we expect that the $\ep^2 \Delta \eta_i$ terms are negligible
 so that at leading order,
 \begin{subequations}\label{eqsystemlo}
%\begin{cases}
 \begin{align}
 & g_1\left(\eta_1^2-a_{1,\varepsilon}(x) \right)+g\left(\eta_2^2-a_{2,\varepsilon}(x) \right)=0, \label{eqsystemlo1}\\ \label{eqsystemlo2}
&g_2\left(\eta_2^2-a_{2,\varepsilon}(x)
\right)+g\left(\eta_1^2-a_{1,\varepsilon}(x) \right)=0.
\end{align}
%\end{cases}
\end{subequations} Near $R_{1,0}$,  the term
$\varepsilon^2 \Delta \eta_1$ cannot be neglected so that we use (\ref{eqsystemlo2}) to express $\eta_2^2-a_{2,\varepsilon}$ and insert it
 into (\ref{eqsystem1}) to find a scalar equation for $\eta_1$. In the region of coexistence, that is in the disk of
 radius $R_{1,0}$,
   $\eta_2$ is obtained from $\eta_1$ by (\ref{eqsystemlo2}), while outside this disk, $\eta_1$ is small and can be neglected in  (\ref{eqsystem2}).
 This reduces the system (\ref{eqsystem}) to two independent approximate scalar problems:
\begin{subequations}\label{eqgroundstate}\begin{align}\label{eqgroundstate1}
&-\varepsilon^2\Delta
\eta_1+\left(g_1-\frac{g^2}{g_2}\right)\eta_1\left(\eta_1^2-a_{1,\varepsilon}(x)
\right)=0,\ \ x\in \mathbb{R}^2;\ \eta_1(x)\to 0\ \textrm{as}\
|x|\to \infty,
\\\label{eqgroundstate2}
&-\varepsilon^2\Delta
\eta_2+g_2\eta_2\left(\eta_2^2-a_{2,\varepsilon}(x)-\frac{g}{g_2}a_{1,\varepsilon}(x)
\right)=0,\ \ x\in \mathbb{R}^2;\ \eta_2(x)\to 0\ \textrm{as}\
|x|\to \infty,\end{align}
\end{subequations}
whose unique positive solutions are called
$\hat{\eta}_{1,\varepsilon}$ and $\hat{\eta}_{2,\varepsilon}$.  The properties of $\hat{\eta}_{1,\varepsilon}$ and $\hat{\eta}_{2,\varepsilon}$  can be deduced from an analogue of Theorem
\ref{thmGSscalar} (see Proposition \ref{proGS} below).  We point out that
they are linearly nondegenerate, which implies that the spectrum of
the associated linearized operators to (\ref{eqgroundstate1}) and
(\ref{eqgroundstate2})  consists only of
positive eigenvalues.
 The main
features of $a_{1,\varepsilon}$ and $a_{2,\varepsilon}$ that are
used for studying $\hat{\eta}_{1,\varepsilon}$ and
$\hat{\eta}_{2,\varepsilon}$ are that $a_{1,\varepsilon}$ and
$a_{2,\varepsilon}+\frac{g}{g_2}a_{1,\varepsilon}$ change sign
 once, from positive to negative as $|x|$ crosses
$R_{1,\varepsilon}$ and $R_{2,\varepsilon}$ respectively and that
\begin{equation}\label{eqnondegeneracy}
a_{1,\varepsilon}'(R_{1,\varepsilon})\to -c<0\ \ \textrm{and}\ \
\left(a_{2,\varepsilon}+\frac{g}{g_2}a_{1,\varepsilon}\right)'(R_{2,\varepsilon})\to
-c'<0\ \ \textrm{as}\ \varepsilon\to 0.
\end{equation}In
particular,  $\hat{\eta}_{1,\varepsilon}^2$ and
$\hat{\eta}_{2,\varepsilon}^2$ converge to $\left(a_{1,0}\right)^+$
and $\left(a_{2,0}+\frac{g}{g_2}a_{1,0}\right)^+$,
uniformly on $\mathbb{R}^2$,  as $\varepsilon\to 0$.

Equation
(\ref{eqgroundstate1}) provides an effective approximation for  (\ref{eqsystem1}) up to a neighborhood of
$R_{2,0}$, where the term $\varepsilon^2 \Delta \eta_{2}$ is
expected to be of equal or higher order than $\varepsilon^2\Delta
\eta_1$ as $\varepsilon\to 0$.  The approximate solutions to (\ref{eqsystem}) that are
constructed in this way match in $C^k$,  in  fixed intervals contained
between $R_{1,0}$ and $R_{2,0}$.  Therefore, we can pick any point
 in $(R_{1,0},R_{2,0}$), for instance the middle point and
we can smoothly glue the solutions together, via a standard interpolation
argument  to create a global approximate solution to the
problem (\ref{eqsystem}). We remark that,
in order to estimate the remainder that this approximation leaves in
(\ref{eqsystem}), we have to prove some new estimates for the
behavior of the derivatives of the ground states of
(\ref{eqgroundstate1}) and (\ref{eqgroundstate2}) near $R_{1,0}$ and
$R_{2,0}$ respectively, as $\varepsilon\to 0$.

 The next step is to
study the associated linearized operator to the system
(\ref{eqsystem}) about this approximation (see
(\ref{eqlinearization--}) below). To this end, as before, our
approach is to reduce the corresponding coupled linearized system to
the two independent scalar linearized problems that are associated
to (\ref{eqgroundstate1}) and (\ref{eqgroundstate2}). As we have
already remarked, the spectrum of the latter scalar  operators
consists only of strictly positive eigenvalues. This property allows
us to apply a domain decomposition argument to handle the ``two
center'' difficulty of the problem. We are able to show that the
associated linearized operator to the system (\ref{eqsystem}) about
the approximate solution is invertible, for small $\varepsilon>0$,
and estimate its inverse in various suitable Sobolev $L^2$-norms. We
point out that this is in contrast to the scalar case, where it is
more convenient to estimate the inverse in uniform norms (see
\cite{GalloPelinovskyAA,KaraliSourdisGround}). Armed with these
estimates, we can apply a contraction mapping argument to prove the
existence of a true solution to (\ref{eqsystem}), near the
approximate one with respect to a suitable $\varepsilon$-dependent
Sobolev norm, for small $\varepsilon>0$. Finally, we can show that the solution is
 positive and obtain the
uniform estimates of Theorem \ref{thmMain} by building on the
Sobolev estimates and making use of the equation. In particular, we
also make use of some carefully chosen weighted uniform norms, see
(\ref{eqnormStar}), which are partly motivated by \cite{pacard-riviere}.

We note that our approach can be extended to cover the more
complicated case of the disk and annulus configuration, that is when
(\ref{eqannulus}) is assumed.

In order to prove Theorem \ref{thm:nonexistence_vortices}, we need to study some auxiliary functions $\xi_{i,\eps}$,
involving the primitive of $s\eta_{i,\eps}^2 (s)$, where $(\eta_{1,\eps},\eta_{2,\eps}$ is the positive minimizer of $E_\varepsilon^0$. For this purpose, we use the
estimates that link (\ref{eqsystem}) to (\ref{eqgroundstate}), obtained in the proof of Theorem
\ref{thmMain}, to derive the estimates for  $\xi_{i,\eps}$ as perturbations of those
 for the scalar equation.
 Then, we use a division trick which  splits
the energy as the sum of the energy of the minimizer without
rotation plus a reduced energy (see Lemma
\ref{lemma:splitting_energy}). The reduced energy bears
similarities to a weighted coupled Ginzburg-Landau  energy
 and, after integrating by parts, includes a Jacobian
(see Lemma \ref{lemma:tilde_F_negative}). Assumption (\ref{eq:condition_on_g_thomas_fermi})  allows us to
treat this coupled energy as two uncoupled ones of analogous form,
and conclude by following the arguments of
\cite{AftalionJerrardLetelierJFA} which are based on the control of
the auxiliary functions and  a Jacobian estimate due
to \cite{jerard}.

% When restricting our attention to
%radial solutions, system (\ref{eq:main_eta_1_eta_2}) can be recast
%as a $5$-dimensional slow-fast system of ordinary differential
%equations, with four fast variables $(\eta_1,\varepsilon \eta_1',\eta_2,\varepsilon\eta_2')$ and one slow variable $r$, see
%\cite{jones}.  The
%difficulty lies on the fact that the slow manifold admits
%pitchfork bifurcations as the slow variable crosses $R_{1,0}$ and $R_{2,0}$.

\hfill

The organization of the paper is as follows: in Section \ref{secUniq}, we prove Theorem \ref{thm:uniqueness}.
 In Section \ref{sec:energy_estimates}, we obtain the first rough estimates
 for the asymptotic behavior of
$(\eta_{1,\varepsilon}^2,\eta_{2,\varepsilon}^2)$ as $\varepsilon\to 0$.
 In Section \ref{secPerturb}, we apply the perturbation argument to prove Theorem \ref{thmMain}.
  In Section \ref{secann}, we prove Theorem \ref{thmMainann}.
In Section \ref{secAuxiliary}, we study some auxiliary functions
 that will be useful for the estimates with rotation.
In Section \ref{secRotation}, we prove Theorem
\ref{thm:nonexistence_vortices}.
 We close the paper with two appendixes. In Appendix
\ref{secAppenGS}, we summarize some known results about the scalar
ground states of Theorem \ref{thmGSNoRotationScalar}, in a more
general setting, and derive estimates for their derivatives. In
Appendix \ref{ApAlg}, we have postponed the proof of a technical
estimate from Section \ref{secPerturb} that is related to our use of
the weighted norms.

\subsection{Notation}By $c/C$ we denote small/large positive
generic constant, whose
values may decrease/increase from line to line.  By $\mathcal{O}(\cdot)$ and
$o(\cdot)$, we denote the standard Landau symbols. We write $r=|x|$ to denote the Euclidean distance
of a point $x$ from the origin. By $B_R$ we denote the
Euclidean ball of radius $R$ and center $0$.

%\tableofcontents

%%%%%%%%%%%%%%%%%%%%%%%%%%%%%%%%%%%%%%%%%%%%%%%%%%%%%%%%%%%%%%%%%%%%%%%%%%%%%%%%%%%%%%%%
%%%%%%%%%%%%%%%%%%%%%%%%%%%%%%%%%%%%%%%%%%%%%%%%%%%%%%%%%%%%%%%%%%%%%%%%%%%%%%%%%%%%%%%%%%%%%
%%%%%%%%%%%%%%%%%%%%%%%%%%%%%%%%%%%%%%%%%%%%%%%%%%%%%%%%%%%%%%%%%%%%%%%%%%%%%%%%%%%%%%%%%

\section{Uniqueness issues}\label{secUniq}

In this section, we prove Theorem \ref{thm:uniqueness}. Since
the result holds for every $\ep>0$, we  often omit the
subscript $\ep$.

\subsection{Uniqueness of positive solutions of (\ref{eq:main_eta_1_eta_2})}
Given positive $\lambda_{1,\varepsilon}$,
$\lambda_{2,\varepsilon}$, we want to prove
 the uniqueness of the positive solutions of
(\ref{eq:main_eta_1_eta_2}).  We
use some ideas from \cite{brezis-oswald} which deals with a class of scalar equations in bounded domains.
 In order to extend it to the entire space,
 we have to establish
some control on the decay  of positive,
possibly  non-radial, solutions.

\begin{lemma}\label{lemma:L_infty_bounds}
Let $(u_1,u_2)$ be a positive solution of (\ref{eq:main_eta_1_eta_2a})-(\ref{eq:main_eta_1_eta_2b}), then
\[
u_i^2\leq \lambda_{i,\ep}/g_i, \qquad \|\nabla
u_i\|_{L^\infty(\R^2)}\leq
C\frac{\sqrt{\lambda_{i,\ep}}(\lambda_{i,\ep}+\lambda_{j,\ep}+1)}{\ep},\ i=1,2, \ j\neq i.
\]
\end{lemma} The  proof is adapted from
\cite{AftalionJerrardLetelierJFA} and \cite{IgnatMillotJFA}.
\begin{proof}
Let $ w_i=(\sqrt{g_i} u_i-\sqrt{\lambda_{i,\ep}})/\ep$, then Kato's inequality yields
\[
\Delta w_i^+ \geq \chi_{\{ w_i\geq0\}} \Delta w_i \geq \chi_{\{
w_i\geq0\}} \frac{\sqrt{g_i}}{\ep^3} u_i(g_i u_i^2-\la_{i,\ep}),
\]
where $\chi$ is the characteristic function of a set. Then we obtain
\[
\Delta w_i^+ \geq \chi_{\{ w_i\geq0\}} \frac{\ep
w_i+\sqrt{\la_{i,\ep}}}{\ep^2}  w_i (\ep w_i+2\sqrt{\la_{i,\ep}})
\geq ( w_i^+)^3.
\]
A non-existence result by Brezis \cite{brezisLiouville} implies that
$ w_i^+\equiv0$, so that the first bound is proved. In fact, since
$u_i$ is bounded, it follows by a standard barrier argument that
(\ref{limeqeta}) is also satisfied.

Now fix $x\in\R^2$, $L>0$ and for $y\in B_{2L}(x)$, let $z_i(y)=
u_i(\ep(y-x))$. Then
\[
-\Delta
z_i=-z_i(\ep^2|y-x|^2+g_iz_i^2+gz_j^{2}-\la_{i,\ep})=:h_{i,\ep}(y),\
\ \ (i\neq j).
\]
We have proved above that there exists $C>0$ independent of $\ep$ and of
$x$ such that $\|h_{i,\ep}\|_{L^\infty(B_{2L}(x))}\leq
C\sqrt{\lambda_{i,\ep}}(\lambda_{i,\ep}+\lambda_{j,\ep}+1)$.
Standard regularity theory for elliptic equations implies $\|\nabla
z_i\|_{L^\infty(B_{L}(x))}\leq
C\sqrt{\lambda_{i,\ep}}(\lambda_{i,\ep}+\lambda_{j,\ep}+1)$, and in
turn the second part of the statement.
\end{proof} This implies in particular uniform bounds for the
 solutions of (\ref{eq:main_eta_1_eta_2}).
 In the following lemma, we prove that positive solutions of
(\ref{eq:main_eta_1_eta_2}) decay super-exponentially fast as
$|x|\to \infty$.
\begin{lemma}\label{lemma:decay_for_large_x}
Let $(u_1,u_2)$ be a positive solution of
\eqref{eq:main_eta_1_eta_2}. For every  $k>0$, let
$r_i=\sqrt{(1+k)\la_{i,\ep}}$ and
\[
W_i(s)= \max_{\partial B_{r}} u_i \cdot \exp\left( -\frac{1}{2\ep}
\sqrt{\frac{k}{1+k}}(s^2-r^2) \right) \qquad \text{for }\ s\geq
r\geq r_i, \ \ i=1,2.
\]
Then we have $u_i(x)\leq W_i(|x|)$ for $|x|\geq r\geq r_i$, $i=1,2$. Moreover,
\begin{equation}\label{equniqGenExp}
|u_i(x)|+\left|\nabla u_i(x) \right|\leq C_\varepsilon
e^{-c_\varepsilon |x|^2},\ \ x\in \mathbb{R}^2,\ i=1,2.
\end{equation}
\end{lemma}

\begin{proof}
Given $k>0$, let $r\geq r_i=\sqrt{(1+k)\la_{i,\ep}}$. For $|x|\geq
r$ we have $\la_{i,\ep} \leq \frac{|x|^2}{1+k}$, and hence
$|x|^2-\la_{i,\ep}\geq \frac{k}{1+k}|x|^2$, so that $u_i$ satisfies
\[
-\Delta u_i+\frac{ u_i}{\ep^2}\frac{k}{1+k}|x|^2 \leq -\Delta
u_i+\frac{ u_i}{\ep^2}(|x|^2-\la_{i,\ep})=-\frac{ u_i}{\ep^2}(g_i
u_i^2+gu_j^2)\leq 0
\]
for every $|x|\geq r$ (here $j\neq i$). On the other hand, it is
easy to check that
\[
-\Delta W_i+\frac{W_i}{\ep^2}\frac{k}{1+k}|x|^2 > 0 \quad
\text{for } |x|\geq r, \qquad W_i(r)\geq u_i(x) \text{ for } |x|=r.
\]
Suppose by contradiction that $W_i- u_i$ is negative somewhere in
the exterior of $\bar{B}_r$. Since both functions go to zero at
infinity, there exists $\bar x$, with $|\bar x| >r$, where
 $W_i- u_i$ reaches its minimum:
$W_i(|\bar x|)- u_i(\bar x)<0$ and $\Delta W_i(|\bar x|)-\Delta
u_i(\bar x)\geq 0$. By subtracting the two differential inequalities
satisfied by $ u_i$ and $W_i$, and then evaluating at $\bar x$, we
obtain a contradiction.

Lemma \ref{lemma:L_infty_bounds} implies a uniform bound for $\max u_i$.
Using  (\ref{eq:main_eta_1_eta_2}), we obtain that
\[
\left|\Delta u_i(y) \right|\leq C_\varepsilon e^{-c_\varepsilon
|x|^2},\ \ y\in B_1(x),\ \ x\in \mathbb{R}^2.
\]
 By
standard interior  elliptic  estimates, we deduce (\ref{equniqGenExp}).
\end{proof}

\begin{proposition}\label{proUniqGeneral}
Assume (\ref{eq:condition_on_g_thomas_fermi}). Given
$\lambda_{i,\varepsilon}>0$, then problem
(\ref{eq:main_eta_1_eta_2}) has at most one positive solution.
\end{proposition}
\begin{proof}
Let $(\eta_1,\eta_2)$ and $(u_1,u_2)$ be two positive solutions of
(\ref{eq:main_eta_1_eta_2}) with the same $\lambda_{1,\varepsilon}$
and $\lambda_{2,\varepsilon}$. Then, the function $\psi_i=
u_i/\eta_i$ solves the following
equation with $j\neq i$:
\begin{equation}\label{equniqEqlin}
\begin{split}
-\nabla\cdot(\eta_i^2\nabla \psi_i)  &= u_i\Delta\eta_i-\eta_i\Delta u_i  \\
&= \frac{\eta_i^2\psi_i}{\ep^2}\left[
g_i\eta_i^2(1-\psi_i^2)+g\eta_j^2(1-\psi_j^2)\right].
\end{split}
\end{equation}We want to show
 that $\psi_i$ is identically equal to $1$. To
this end, we multiply equation (\ref{equniqEqlin}) by
$(\psi_i^2-1)/\psi_i$ in a ball of radius $R$ to obtain
\begin{align*}
\int_{B_R} \left\{ \eta_i^2|\nabla\psi_i|^2\left(1+\frac{1}{\psi_i^2}\right)+\frac{\eta_i^2}{\ep^2}\left[ g_i\eta_i^2(\psi_i^2-1)^2 +g\eta_j^2(\psi_i^2-1)(\psi_j^2-1)\right]  \right\} \,dx \\
=\int_{\partial B_R} \left\{ \left(  u_i-\frac{\eta_i^2}{ u_i}
\right)\nabla u_i-\left(\frac{
u_i^2}{\eta_i}-\eta_i\right)\nabla\eta_i \right\}\cdot\nu\, d\sigma,
\end{align*}
where $\nu$ denotes the outer unit normal vector to $\partial B_R$.
We sum the previous identities for $i=1,2$ and then we use the
following inequality
\begin{equation}\label{eq:coercivity}
|2g\eta_1^2\eta_2^2(\psi_1^2-1)(\psi_2^2-1)|\leq
(g_1-\gamma)\eta_1^4(\psi_1^2-1)^2+(g_2-\gamma)\eta_2^4(\psi_2^2-1)^2,
\end{equation}
where $0<\gamma<\min\{g_1,g_2\}$ is such that
\begin{equation}\label{eq:g_mathfrak_definition}
g\leq\sqrt{g_1-\gamma}\sqrt{g_2-\gamma},
\end{equation}
which exists by \eqref{eq:condition_on_g_thomas_fermi}. We get
\begin{equation}\label{equniqlastRel}
\begin{split}
\sum_{i=1}^2 \int_{B_R} \left\{ \eta_i^2|\nabla\psi_i|^2\left( 1+\frac{1}{\psi_i^2}\right)+\frac{\gamma}{\ep^2} \eta_i^4(\psi_i^2-1)^2 \right\}\,dx \\
\leq \sum_{i=1}^2 \int_{\partial B_R} \left\{ \left(
 u_i-\frac{\eta_i^2}{ u_i}
\right)\nabla u_i-\left(\frac{
u_i^2}{\eta_i}-\eta_i\right)\nabla\eta_i \right\}\cdot\nu\, d\sigma.
\end{split}
\end{equation}
To conclude that $\psi_i\equiv 1$, it is enough to show
that there exist $R_k\to \infty$ such that the right-hand side of
(\ref{equniqlastRel}), with $R=R_k$, tends to zero as $k\to \infty$.
This task will take up the rest of the proof.

 Let $\chi$ be a smooth cutoff function in
$\mathbb{R}^2$ which is identically equal to $1$ in the unit ball
and identically equal to $0$ outside of the ball of radius $2$. For
all $R\geq 1$, we define $\chi_R=\chi(\cdot/R)$. In the remaining part of this proof, we
 denote by $C_\varepsilon$ a positive generic constant  which is
independent of $R\geq 1$. We multiply (\ref{equniqEqlin}) by $r^\alpha \chi_R^2\psi_i$,
where $\alpha>2$ $(r=|x|)$, and integrate the resulting identity by
parts over $\mathbb{R}^2$ to find that
\[
\left| \int_{B_{2R}}^{}\chi_R^2 \left(r^\alpha \nabla \psi_i+\alpha
r^{\alpha-1}\frac{x}{r}\psi_i \right)\eta_i^2 \nabla
\psi_idx+\int_{B_{2R}}^{}r^\alpha \psi_i 2 \chi_R \nabla \chi_R
\eta_i^2 \nabla \psi_idx
 \right|\leq C_\varepsilon,
\]
where we have also made use of (\ref{equniqGenExp}) and of the definition
of $\psi_i$.  Our motivation for including $\chi_R^2$ comes from
\cite[Thm. 1.8]{berestyckiCaffarelli}.
 Thanks to the
elementary inequalities
\[\left|
\chi_R^2 r^{\alpha-1}\frac{x}{r}\psi_i \eta_i^2 \nabla \psi_i\right|
\leq d \chi_R^2r^\alpha \eta_i^2 |\nabla
\psi_i|^2+\frac{1}{2d}\chi_R^2r^{\alpha-2}\psi_i^2\eta_i^2\ \ \
\forall\ d>0,
\]
and
\[
\left|2r^\alpha \psi_i  \chi_R \nabla \chi_R \eta_i^2 \nabla \psi_i
\right|\leq d\chi_R^2r^\alpha\eta_i^2|\nabla
\psi_i|^2+\frac{1}{d}r^\alpha \psi_i^2 |\nabla\chi_R|^2\eta_i^2\ \ \
\forall\ d>0,
\]
choosing a sufficiently small $d>0$ (independent of $R$), via
(\ref{equniqGenExp}),  we infer that
\[
\int_{\mathbb{R}^2}^{}r^\alpha \chi_R^2 \eta_i^2 |\nabla \psi_i|^2dx
\leq C_\varepsilon.
\]
By Lebesgue's monotone convergence theorem, letting $R\to \infty$ in
the above relation, we obtain that
\[
\int_{\mathbb{R}^2}^{}r^\alpha \eta_i^2 |\nabla \psi_i|^2dx \leq
C_\varepsilon.
\]
Replacing $\psi_i$ by its value and using (\ref{equniqGenExp}),
we find that
\[
\int_{\mathbb{R}^2}^{}r^\alpha  u_i^2\frac{|\nabla
\eta_i|^2}{\eta_i^2}dx \leq C_\varepsilon.
\]
Reversing the roles of $u_i$ and $\eta_i$, and summing, we reach
\[
\int_{\mathbb{R}^2}^{}r^\alpha\sum_{i=1}^{2} \left(
u_i^2\frac{|\nabla \eta_i|^2}{\eta_i^2}+\eta_i^2\frac{|\nabla
u_i|^2}{u_i^2}\right)dx \leq C_\varepsilon.
\]
Therefore, by the co-area formula, there
exists a sequence $R_k\to \infty$ such that
\[
R_k^\alpha \int_{\partial B_{R_k}}^{}\sum_{i=1}^{2} \left(
u_i^2\frac{|\nabla \eta_i|^2}{\eta_i^2}+\eta_i^2\frac{|\nabla
u_i|^2}{u_i^2}\right)d\sigma \to 0\ \ \textrm{as}\ k\to \infty.
\]
To conclude, we note that the above relation, the Cauchy-Schwarz
inequality, and (\ref{equniqGenExp}), imply that the right-hand side
of (\ref{equniqlastRel}) at $R=R_k$ tends to zero as $k\to \infty$,
as desired. We remark that, in the case of radial symmetry, one can
argue directly, analogously to \cite{AftalionJerrardLetelierJFA}, by
making use of Lemma \ref{lemma:eta'/eta_bound} below.
\end{proof}

\begin{rem}\label{remark:g_1=g_2}
If $g_1=g_2$ and
$\la_{1,\ep}=\la_{2,\ep}$, then $\eta_1=\eta_2$. Indeed, $(\eta_1,\eta_2)$ and $(\eta_2,\eta_1)$
are both positive solutions to (\ref{eq:main_eta_1_eta_2}) and we can apply
Proposition \ref{proUniqGeneral}.
\end{rem}
\begin{rem}\label{remark:unique} The uniqueness result of Proposition \ref{proUniqGeneral} yields radial symmetry of $u_1$ and $u_2$.
\end{rem}

We observe that the proof of Proposition \ref{proUniqGeneral}
applies to provide also the following local uniqueness result, since in this case, the boundary terms vanish.

\begin{proposition}\label{lemma:local_uniqueness}
Assume \eqref{eq:condition_on_g_thomas_fermi}. Given
$\lambda_{i,\varepsilon}>0$, $i=1,2$, and a bounded domain $B\subset
\R^2$ with Lipschitz continuous boundary, if $(\eta_1,\eta_2)$ and
$( u_1, u_2)$ are positive solutions to the elliptic system in
\eqref{eq:main_eta_1_eta_2} on $\bar{B}$ such that $ u_i=\eta_i$ on
$\partial B$ for $i=1,2$, then $ u_i\equiv\eta_i$ in $B$.
\end{proposition}

\subsection{Uniqueness and radial symmetry of the ground state}

We now turn to the uniqueness and radial symmetry of the positive
minimizer of the energy without rotation
\begin{align}\label{eq:energy_E_0_ep}
E^0_\ep(u_1,u_2)=\sum_{i=1}^{2} \int_{\R^2}\left\{ \frac{|\nabla
u_i|^2}{2} + \frac{|x|^2}{2\ep^2}| u_i|^2 +\frac{g_i}{4\ep^2}|
u_i|^4 \right\} \,dx +\frac{g}{2\ep^2}\int_{\R^2}| u_1|^2| u_2|^2
\,dx
\end{align}
in the space $\mathcal{H}$ which is defined in
(\ref{H_space_definition}).
%To this aim, we shall also provide some preliminary estimates on the minimizers.

If $( u_1, u_2)$ minimizes $E_\ep^0$ in $\mathcal{H}$, then the
diamagnetic inequality
implies $E^0_\ep(| u_1|,| u_2|)\leq E^0_\ep( u_1, u_2)$, so that $
u_i$ differs from $| u_i|$ by a constant complex phase. In fact, by
the strong maximum principle, we can assume that $ u_i$ are positive
functions.
%that is belong to the space
%\begin{equation}\label{H+_space_definition}
%\mathcal{H}^+=\left\{( u_1, u_2)\in \mathcal{H}:\  u_i\in
%H^1(\R^2,\R), \  u_i >0 \right\}.
%\end{equation}
By elliptic regularity, positive minimizers of $E_\varepsilon^0$ in
$\mathcal{H}$ lead to smooth solutions of
(\ref{eq:main_eta_1_eta_2a})-(\ref{eq:main_eta_1_eta_2b}) for some positive Lagrange multipliers
$\la_{1,\ep},\la_{2,\ep}$. Nevertheless, we have to prove that (\ref{limeqeta}) holds too.
 A priori, we only know that it holds for radial functions by the Strauss lemma \cite{straussCMP}.
 In the subsequent lemma, we provide a lower bound for the decay rate
of positive solutions to (\ref{eq:main_eta_1_eta_2}) as $|x|\to
\infty$. The following proof  is adapted from
\cite{AftalionJerrardLetelierJFA} and \cite{IgnatMillotJFA}.

\begin{lemma}\label{lemma:bound_from_below}
Let $(u_1,u_2)$ be a positive solution of
\eqref{eq:main_eta_1_eta_2}. Let
\[
w_i(s)= \min_{\partial B_r} u_i \cdot \exp\left(
-\frac{\alpha_i}{2}(s^2-r^2) \right) \qquad \text{for }\ s\geq
r>\sqrt{\la_{i,\ep}},
\]
where, for $i,j=1,2$ and $j\neq i$, we have defined
\[
\alpha_i=\frac{1}{\la_{i,\ep}}+\sqrt{\frac{1}{\la_{i,\ep}^2}+\frac{1}{\ep^2}\left(1+\frac{g\la_{j,\ep}}{g_j\la_{i,\ep}}\right)}.
\]
Then $ u_i(x)\geq w_i(|x|)$ for $|x| \geq r>\sqrt{\la_{i,\ep}}$.
\end{lemma}
\begin{proof}
We know from Lemma \ref{lemma:L_infty_bounds} that $
u_i^2\leq\la_{i,\ep}/g_i$ for $i=1,2$. Thus,  $ u_i$
satisfies
\[
-\ep^2\Delta u_i+ u_i\left(|x|^2+g
\frac{\la_{j,\ep}}{g_j}\right)\geq 0,\ \ x\in \mathbb{R}^2,\ \
(j\neq i).
\]
On the other hand, our choice of $\alpha_i$ implies that
\[
-\ep^2\Delta w_i+ w_i\left(|x|^2+g \frac{\la_{j,\ep}}{g_j}\right) <
w_i \left(2\alpha_i
\ep^2+\frac{g}{g_j}\la_{j,\ep}+(1-\ep^2\alpha_i^2)\la_{i,\ep}\right)
= 0,\ \ |x|>r,
\]
where we have used that $1-\ep^2\alpha_i^2<0$. The maximum principle can
now be applied, as in Lemma \ref{lemma:decay_for_large_x}, to yield
the desired lower bound.
\end{proof}

In the case of radial solutions, we can show the following lemma analogous to the one in \cite{AftalionJerrardLetelierJFA}.

\begin{lemma}\label{lemma:eta'/eta_bound}
Let $(u_1,u_2)$ be a positive radial solution of
\eqref{eq:main_eta_1_eta_2}. There exists $C_\ep>0$ independent of
$r$ such that, for $i=1,2$, we have
\[
| u_i'(r)|\leq C_\ep r  u_i(r) \qquad \text{for }\  r >2
\sqrt{\la_{i,\ep}}.
\]
\end{lemma}
\begin{proof}
Since $u_{i}$ is radial, and $r>2 \sqrt{\la_{i,\ep}}$, an
application of Lemma \ref{lemma:decay_for_large_x} (with $k=3$)
yields that
\[
u_{i}(r)=W_{i}(r)\quad\text{ and } \quad u_{i}(s)\leq W_{i}(s)\
\text{ for }\ s>r.
\]
It follows that $ u_i'(r)\leq W_i'(r)\leq 0$. Similarly, an
application of Lemma \ref{lemma:bound_from_below} yields that $
u_i'(r)\geq w_i'(r)=-\alpha_i r  u_i(r)$, completing the proof.
\end{proof}

In order to proceed, we need the following splitting of the energy:

\begin{proposition}\label{lemma:splitting_energy_no_rotation} Assume \eqref{eq:condition_on_g_thomas_fermi}.
Let $(\eta_1,\eta_2)$ be a minimizer among radial functions in $\mathcal{H}$
 with $\eta_i>0$. Let $( u_1, u_2)\in
\mathcal{H}$. Then the splitting of energy (\ref{ensplit}) holds.
\end{proposition}
\begin{proof}
We test the equation for $\eta_i$ by $\eta_{i}(|v_i|^2-1)$ in a ball
of radius $R$ and then integrate by parts. As a result, we get the
term
\[
\begin{split}
\int_{B_R}\left\{|\nabla\eta_i|^2(|v_i|^2-1)+2\eta_i\nabla\eta_i\cdot(v_i,\nabla v_i)\right\}\,dx = \int_{B_R} -\Delta \eta_i \eta_i(|v_i|^2-1)\,dx \\
+\int_{\partial B_R}\left(\frac{| u_i|^2}{\eta_i}-\eta_i\right)
\nabla\eta_i\cdot\nu\, d\sigma,
\end{split}
\]
 Lemmas
\ref{lemma:L_infty_bounds} and \ref{lemma:eta'/eta_bound} apply to
$\eta_{i}$ to provide
\[
\left|\int_{\partial B_R}\left(\frac{ |u_i|^2}{\eta_i}-\eta_i\right)
\nabla\eta_i\cdot\nu\, d\sigma \right| \leq C_\ep \int_{\partial
B_R} R  |u_{i}|^{2} \, d\sigma + C_\ep \int_{\partial B_R} R
\eta_{i}^2 \, d\sigma,
\]
Note that the conditions $(\eta_1,\eta_2)$, $( u_{1}, u_{2})\in
\mathcal{H}$ imply the existence of a sequence $R_{k}\to\infty$ such
that the integrals above vanish along $R_{k}$ via the co-area
formula. Therefore, we have
\[
\begin{split}
\int_{\R^2}(|\nabla\eta_i|^2(|v_i|^2-1)+2\eta_i\nabla\eta_i\cdot(\nabla
v_i,v_i))\,dx = -\frac{1}{\ep^2} \int_{\R^2} \eta_i^2 (|v_i|^2-1) (|x|^2
+g_i\eta_i^2 +g\eta_j^2 )\,dx,
\end{split}
\]
where the Lagrange multiplier term has disappeared because
$\int_{\R^2}\eta_i^2(|v_i|^2-1)\,dx=0$. We replace the last equality
into the definition of $E_\ep^0( u_1, u_2)$ to find
\[
\begin{split}
& E_\ep^0( u_1, u_2)=E_\ep^0(\eta_1v_1,\eta_2v_2)\\
& =\sum_{i=1}^2 \int_{\R^2} \left\{ \frac{|\nabla\eta_i|^2}{2}|v_i|^2+ \eta_i \nabla\eta_i\cdot(\nabla v_i, v_i) +\eta_i^2\frac{|\nabla v_i|^2}{2}+\frac{|x|^2}{2\ep^2}\eta_i^2|v_i|^2+\frac{g_i}{4\ep^2}\eta_i^4 |v_i|^4 \right\}\,dx \\
&+\frac{g}{2\ep^2} \int_{\R^2} \eta_1^2\eta_2^2|v_1|^2 |v_2|^2\,dx \\
& =\sum_{i=1}^2 \int_{\R^2} \left\{ \frac{|\nabla\eta_i|^2}{2}+\eta_i^2\frac{|\nabla v_i|^2}{2}+\frac{|x|^2}{2\ep^2}\eta_i^2+\frac{g_i}{4\ep^2}\eta_i^4 + \frac{g_i}{4\ep^2}\eta_i^4(|v_i|^2-1)^2 \right\} \,dx \\
& + \frac{g}{2\ep^2}
\int_{\R^2}\eta_1^2\eta_2^2(|v_1|^2|v_2|^2-|v_1|^2-|v_2|^2+2)\,dx.
\end{split}
\]
By collecting the term $E_\ep^0(\eta_1,\eta_2)$ in the previous
expression, the result follows.
\end{proof}

\subsection{Proof of Theorem \ref{thm:uniqueness}}\begin{proof}Given $\lambda_{i,\varepsilon}>0$, $i=1,2$, the first assertion of
the theorem is proven in Proposition \ref{proUniqGeneral}.

Now, let $(\eta_1,\eta_2)$ be a minimizer of $E_\ep^0$ in $\mathcal H$ among radial
  functions, and let $( u_1, u_2)$ be a
minimizer of $E^0_\ep$ in $\mathcal{H}$. Since $(\eta_1,\eta_2)$ is
an admissible test function, we have $E_\ep^0( u_1, u_2)\leq
E_\ep^0(\eta_1,\eta_2)$. Consequently, the quantity
$F_\ep^0(v_1,v_2)$,  defined in (\ref{ensplit}), satisfies
\[
F_\ep^0(v_1,v_2)=E_\ep^0( u_1, u_2)- E_\ep^0(\eta_1,\eta_2)\leq0.
\]
%We recall that \eqref{eq:condition_on_g_thomas_fermi} implies the
%existence of $0<\gamma<g_1$ such that
%\begin{equation*}
%g\leq\sqrt{g_1-\gamma}\sqrt{g_2-\gamma}.
%\end{equation*}
On the other hand, recalling (\ref{eq:g_mathfrak_definition}), we
find that
\[
F_\ep^0(v_1,v_2)\geq\sum_{i=1}^2 \int_{\R^2} \left(
\frac{\eta_i^2}{2}|\nabla
v_i|^2+\frac{\gamma}{4\ep^2}\eta_i^4(|v_i|^2-1)^2 \right) \,dx\geq
0.
\]
This implies that $F_\ep^0(v_1,v_2)=0$ and that $|v_i|\equiv 1$ for
$i=1,2$, which implies the second assertion of the theorem. If $g_1=g_2$ then $\eta_1\equiv \eta_2$, as $(\eta_1,\eta_2)$ and
$(\eta_2,\eta_1)$ are both minimizers.
\end{proof}

% Sobolev
%spaces
%\[H^m(a,b,rdr)\equiv
%    \left\{\begin{array}{c}v:(a,b)\to \mathbb{R}\ :\ v \
%\textrm{has\ weak\  derivatives\ up\ to\ order}\ m \ \textrm{and}
%     \\
%    \int_{a}^{b}\left|v^{(j)}(r)  \right|^2r dr<\infty,\ \ j=0,1,\cdot,m
%  \end{array}
%\right\},
%\]
%for $0\leq a<b\leq \infty$, with the convection that
%$H^0(a,b,rdr)\equiv L^2(a,b,rdr)$. These spaces are Hilbert spaces
%when equipped with the obvious norm (see \cite{deFigueiredoRadial}
%for more information).

%%%%%%%%%%%%%%%%%%%%%%%%%%%%%%%%%%%%%%%%%%%%%%%%%%%%%%%%%%%%%%%%%%%%%%%%%%%%%%%%%%%%%%%%%%%%%%%%%%%%
%%%%%%%%%%%%%%%%%%%%%%%%%%%%%%%%%%%%%%%%%%%%%%%%%%%%%%%%%%%%%%%%%%%%%%%%%%%%%%%%%%%%%%%%%%%%%%%%%%%%

%\]

\section{Preliminary estimates for the energy minimizer without rotation}\label{sec:energy_estimates}

In this section we prove that, under assumptions
(\ref{eq:condition_on_g_thomas_fermi}),
 (\ref{eq:condition_on_g_two_disks}) and (\ref{eq14+}), the positive minimizer $(\eta_1,\eta_2)$ provided by Theorem \ref{thm:uniqueness} satisfies
\[
\eta_i^2\to a_i \quad\text{in } L^2(\R^2) \quad\text{ and }\quad
\la_{i,\ep}\to\la_{i,0}\ \ \textrm{as}\ \ \varepsilon\to 0.
\]
This result is achieved through the estimate of the energy of the
minimizer.

\subsection{Limiting profiles}\label{subsecLimitingProf}

%We will show that there exists a unique minimizer of $E_\ep^0$ in
%$\mathcal{H^+}$ and that it is radial. We will denote it by
%$(\eta_1,\eta_2)$ (we omit the subscript $\ep$ in this section). Our aim is to study
%the convergence of $(\eta_1,\eta_2)$ as $\ep\to0$, under the
%assumption that the supports of the limiting functions are two disks
%of radii $R_{1,0}$ and $R_{2,0}$ respectively.

We recall briefly how to calculate the limiting configuration
\eqref{eq:a_i_def}. We first assume the case of two disks \[
D_i=\{x\in\R^2: \ |x|<R_{i,0} \}.
\] where $R_{1,0}<R_{2,0}$ to be determined later.
 If
$x\in D_1$, formally let $\ep=0$ in \eqref{eq:main_eta_1_eta_2} and
solve the resulting algebraic system in $\eta_1^2,\eta_2^2$. This
provides, for $x\in D_1$,
\begin{equation}\label{eq:a_1_alternative_form}
a_{1,0}(x)=\frac{1}{g_1\Gamma}\left(\la_{1,0}-\frac{g}{g_2}\la_{2,0}-\Gamma_2|x|^2\right),
\end{equation}
\begin{equation}\label{eq:a_2_int_alternative_form}
a_{2,0}(x)=\frac{1}{g_2\Gamma}\left(\la_{2,0}-\frac{g}{g_1}\la_{1,0}-\Gamma_1|x|^2\right),
\end{equation}
and also the value of $R_{1,0}$, which is the radius at which
$a_{1,0}$ vanishes:
\begin{equation}\label{eq:R_1_def}
R_{1,0}^2=\frac{1}{\Gamma_2} \left( \la_{1,0}-\frac{g}{g_2}\la_{2,0}
\right).
\end{equation}
If $x\in D_2\setminus D_1$, then $\eta_1=0$ and formally with
$\ep=0$ in \eqref{eq:main_eta_1_eta_2}, we solve the resulting
equation for $\eta_2^2$, to obtain the following limiting behavior
for $\eta_2^2$:
\begin{equation}\label{eq:a_2_ext_alternative_form}
\frac{\la_{2,0}-|x|^2}{g_2},\qquad \text{with}\qquad
R_{2,0}^2=\la_{2,0}.
\end{equation}
Notice that $(\la_{2,0}-|x|^2)/g_2=a_{2,0}+\frac{g}{g_2} a_{1,0}$,
in agreement with our definition of $a_2$ in \eqref{eq:a_i_def}.
Finally, by imposing the normalization conditions
$\|a_1\|_{L^2(\R^2)}=\|a_2\|_{L^2(\R^2)}=1$, we obtain
\[
\la_{2,0}^2=\frac{2(g_2+g)}{\pi}, \qquad
\la_{1,0}-\frac{g}{g_2}\la_{2,0}=\sqrt{\frac{2g_1\Gamma\Gamma_2}{\pi}},
\]
and hence
\begin{equation}\label{eq:R_i_def}
R_{1,0}^4=\frac{2g_1\Gamma}{\pi\Gamma_2}, \qquad
R_{2,0}^4=\frac{2(g_2+g)}{\pi}.
\end{equation}

Notice that, in our setting, the condition (\ref{eq14+}) is
equivalent to $R_{1,0}<R_{2,0}$, as can be deduced from
\eqref{eq:R_i_def}. Next observe that the monotonicity of $a_{2,0}$
depends on the sign of $\Gamma_1$. If $\Gamma_1>0$, then $a_{2,0}$
is decreasing and
\begin{equation}\label{eq:bound_below_a_1}
a_{2,0}(x)\geq a_{2,0}(R_{1,0})=(R_{2,0}^2-R_{1,0}^2)/g_2>0, \quad
x\in D_1.
\end{equation}
If $\Gamma_1=0$ then $a_{2,0}$ is constant, whereas it is increasing
when $\Gamma_1<0$. In this last case, we have, for $x\in \R^2$,
\begin{equation}\label{eq:bound_below_a_2}
a_{2,0}(x)\geq a_{2,0}(0)=\frac{1}{g_2\Gamma}\left(
\la_{2,0}-\frac{g}{g_1}\la_{1,0} \right),
\end{equation}
which is a positive constant thanks to
\eqref{eq:condition_on_g_two_disks}. Condition (\ref{eq:condition_on_g_two_disks}) is thus
 equivalent to having two disks.

In the case of a disk plus annulus, we assume that $a_1$ is supported in a disk $D_1$ of radius $R_{1,0}$
 and $a_2$ on an annulus $$D_2=\{x\in\R^2: \ R_{2,0}^-<|x|<R_{2,0}^+ \}$$ with
  $R_{2,0}^-< R_{1,0}<R_{2,0}^+$. Other rearrangements of $R_{1,0}$, $R_{2,0}^-$, $R_{2,0}^+$ can be excluded, see \cite{AftalionMasonWei}.
  In the coexistence region, that is $R_{2,0}^-< |x|< R_{1,0}$, $(\sqrt{a_{1,0}}, \sqrt{a_{2,0}})$ given by (\ref{eq:a_1_alternative_form})-(\ref{eq:a_2_int_alternative_form}) is the solution of (\ref{redetai}).
    The fact that $a_1$ vanishes at $R_{1,0}$ and $a_2$ at $R_{2,0}^-$ and $R_{2,0}^+$ yields (\ref{lambdaiann}). If
    $$r\leq R_{2,0}^-\leq R_{2,0}^+,\ a_2=0\hbox{ and }a_1=\frac{\lambda_{1,0}-r^2}{g_1}$$
    $$R_{1,0}\leq r,\ a_1=0\hbox{ and }a_2=\frac{\lambda_{2,0}-r^2}{g_2}$$
    which are consistent with (\ref{a1ann})-(\ref{a2ann}).
     The computations of the $L^2$ norms provide
\beq\label{Riann}\la_{1,0}=\sqrt{2\frac{g_1(1+\frac{g_2^2}{g^2}  (1-\Gamma_2)^2)}\pi}\hbox{ and } \la_{2,0}-\la_{1,0}=\sqrt{-\Gamma_1\Gamma_2}\sqrt{2\frac{g_1 g_2^2 (1-\Gamma_2)}{\pi g^2}}.\eeq

\subsection{Energy estimates}

In order to obtain some energy estimates, we first rewrite the
energy functional in a different form.

\begin{lemma}\label{lemma:E_tilde_def} Assume (\ref{eq:condition_on_g_thomas_fermi}),
 (\ref{eq:condition_on_g_two_disks}) and (\ref{eq14+}).
Let $( u_1, u_2)\in \mathcal{H}$, then $E^0_\ep( u_1, u_2)=\tilde
E^0_\ep( u_1, u_2)+K$, where
\begin{align*}
\tilde E^0_\ep( u_1, u_2)&=\sum_{i=1}^2\int_{\R^2} \left\{
\frac{|\nabla u_i|^2}{2}  +\frac{g_i}{4\ep^2} ( |u_i|^2-a_i)^2
\right\}\,dx
+\frac{g}{2\ep^2}\int_{\R^2} ( |u_1|^2-a_1)( |u_2|^2-a_2)\,dx \\
&+ \frac{g_1\Gamma}{2\ep^2}\int_{\R^2\setminus D_1}
 |u_1|^2a_{1,0}^-\,dx + \frac{1}{2\ep^2}\int_{\R^2\setminus D_2}
(g |u_1|^2+g_2 |u_2|^2)\left( a_{2,0}+\frac{g}{g_2}a_{1,0} \right)^-
\,dx
\end{align*}
and $K$ is the following constant (depending on $\ep$)
\[
K=\frac{\lambda_{1,0}+\lambda_{2,0}}{2\ep^2} -\frac{1}{4\ep^2}
\int_{\R^2} ( g_1a_1^2+g_2a_2^2+2ga_1a_2 )\,dx.
\]
\end{lemma}
\begin{proof} We note that \[
|x|^2+g_1a_1+ga_2=\left\{\begin{array}{ll}
\lambda_{1,0}, \quad & x\in D_1, \\
\Gamma_2 |x|^2+\frac{g}{g_2}\lambda_{2,0}, \quad & x\in D_2\setminus D_1, \\
|x|^2, & x\in \R^2\setminus D_2,
\end{array}\right.
\]
and
\[
|x|^2+g_2a_2+ga_1=\left\{\begin{array}{ll}
\lambda_{2,0}, \quad & x\in D_2, \\
|x|^2, & x\in \R^2\setminus D_2.
\end{array}\right.
\]
Therefore, we have:
%\[
%g_i |u_i|^4+2|x|^2 |u_i|^2=g_i( |u_i|^2-a_i)^2+2
%|u_i|^2(|x|^2+g_ia_i)-g_ia_i^2,
%\]
%\[
%|u_1|^2 |u_2|^2=( |u_1|^2-a_1)( |u_2|^2-a_2)+ |u_1|^2
%a_2 + |u_2|^2a_1 -a_1a_2.
%\]
%Hence
\begin{equation*}
\begin{split}
g_1 |u_1|^4+2|x|^2 |u_1|^2 & +g_2 |u_2|^4+2|x|^2 |u_2|^2 +2g |u_1|^2 |u_2|^2= \\
& = g_1( |u_1|^2-a_1)^2 + g_2( |u_2|^2-a_2)^2 + 2g ( |u_1|^2-a_1)( |u_2|^2-a_2) \\
& + 2 |u_1|^2(|x|^2+g_1a_1+ga_2) +2 |u_2|^2(|x|^2+g_2a_2+ga_1) \\
& -g_1a_1^2 -g_2a_2^2 -2g a_1a_2.
\end{split}
\end{equation*}
Inserting the above in the definition of $E_\ep^0$, and rearranging
the terms, gives the statement.
\end{proof}

The following proposition provides some estimates  for the minimizer which will be used in the sequel for estimating
the associated Lagrange multipliers.

\begin{proposition}\label{prop:energy_estimates} Assume (\ref{eq:condition_on_g_thomas_fermi}),
 (\ref{eq:condition_on_g_two_disks}) and (\ref{eq14+}).
Let $(\eta_{1},\eta_{2})$ be the positive minimizer of $E^0_\ep$ in
$\mathcal{H}$ that is provided by Theorem \ref{thm:uniqueness}. If
$\varepsilon>0$ is sufficiently small, for $i=1,2$, we have \beq\label{nablaetai}
\int_{\R^2} |\nabla \eta_i|^2\,dx \leq C|\log\ep|,
\eeq
\beq\label{etaiai}
\int_{\R^2} (\eta_i^2-a_i)^2\,dx \leq C \ep^2|\log\ep|,
%\qquad \int_{\R^2} (\eta_1^2-a_1)(\eta_2^2-a_2)\,dx \leq C \ep^2|\log\ep|.
\eeq
\beq\label{estenergydisk}
\int_{\R^2\setminus D_1}\eta_1^2a_{1,0}^-\,dx+\int_{\R^2\setminus
D_2}(g\eta_1^2+g_2\eta_2^2)\left( a_{2,0}+\frac{g}{g_2}a_{1,0}
\right)^-\,dx\leq C\ep^2|\log\ep|.
\eeq
In particular,  $\eta_i^2\to a_i$ in $L^2(\R^2)$ as
$\ep\to0$.
\end{proposition}
\begin{proof}
First, we claim that, for small $\varepsilon>0$, we have
\begin{equation}\label{eq:energy_estimates_claim}
\tilde E^0_\ep(\eta_1,\eta_2)\leq C |\log\ep|,
\end{equation}
with $\tilde E^0_\ep$ defined in Lemma \ref{lemma:E_tilde_def}. This
is proved as in \cite{alama} and \cite{AftalionJerrardLetelierJFA},
therefore we only give a sketch here. Consider the competitor
functions
\[
\tilde\eta_i=\frac{h_\ep(a_i)}{\|h_\ep(a_i)\|_{L^2(\R^2)}}, \qquad
\text{where} \qquad h_\ep(s)=\left\{\begin{array}{ll}
s/\ep \quad &\text{ if } \ 0 \leq s\leq \ep^2, \\
\sqrt{s} \quad &\text{ if } \ s>\varepsilon^2.
\end{array}\right.
\]
It is proved in the aforementioned papers that
\[
1-C\ep^2 \leq \int_{\R^2} h_\ep(a_i)^2\,dx \leq1
\]
\[
\int_{\R^2} |\nabla h_\ep(a_i)|^2\,dx \leq C|\log\ep|
\]
\[
\int_{\R^2} (h_\ep(a_i)^2-a_i)^2\,dx \leq C \ep^2.
\]
Here we are implicitly using assumption
\eqref{eq:condition_on_g_two_disks} which ensures that $a_i$ are
positive (recall
\eqref{eq:bound_below_a_1},\eqref{eq:bound_below_a_2}). In addition
notice that
\[
2\int_{\R^2} (h_\ep(a_1)^2-a_1)(h_\ep(a_2)^2-a_2)\,dx \leq
\sum_{i=1}^2 \int_{\R^2} (h_\ep(a_i)^2-a_i)^2\,dx
\]
and that
\[
\int_{\R^2\setminus D_1} \tilde\eta_1^2a_{1,0}^-\,dx=
\int_{\R^2\setminus D_2} (g\tilde\eta_1^2 +g_2\tilde\eta_2^2 )\left(
a_{2,0}+\frac{g}{g_2}a_{1,0} \right)^-  \,dx=0.
\]
Therefore, we have obtained $\tilde
E^0_\ep(\tilde\eta_1,\tilde\eta_2)\leq C |\log\ep|$. Finally, let
$(\eta_1,\eta_2)$ be the positive minimizer, the decomposition
proved in the previous lemma provides
\[
\tilde E^0_\ep(\eta_1,\eta_2)+K\leq \tilde
E^0_\ep(\tilde\eta_1,\tilde\eta_2)+K,
\]
so that \eqref{eq:energy_estimates_claim} is proved.

On the other hand, relation \eqref{eq:g_mathfrak_definition} implies
\[
%\begin{split}
2g \left| \int_{\R^2} (\eta_1^2-a_1)
%\right. & \left.
(\eta_2^2-a_2)\,dx \right|
%  \\
%&
\leq \int_{\R^2} \left\{
(g_1-\gamma)(\eta_1^2-a_1)^2+(g_2-\gamma)(\eta_2^2-a_2)^2\right\}\,dx.
%\end{split}
\]
The result follows by combining this inequality with
\eqref{eq:energy_estimates_claim}.
\end{proof}

In the following proposition, we derive a preliminary estimate
for the Lagrange multipliers. Even though this estimate is far from
optimal, its form will play an important role when we improve it in
Proposition \ref{proLagrangeImproved}.

\begin{proposition}\label{prop:convergence_lagrange_multipliers} Assume (\ref{eq:condition_on_g_thomas_fermi}),
 (\ref{eq:condition_on_g_two_disks}) and (\ref{eq14+}).
Let $(\eta_{1},\eta_{2})$ be the positive minimizer of $E_\ep^0$ in $\mathcal H$. Let $\la_{i,\ep}$ be the associated Lagrange multipliers in
\eqref{eq:main_eta_1_eta_2}. There exists $C>0$ independent of $\ep$
such that, for $i=1,2$,
\begin{equation}\label{eq:convergence_lagrange_multipliers}
|\la_{i,\ep}-\la_{i,0}|\leq C \ep |\log\ep|^{1/2}
\end{equation}
where $\la_{i,0}$ are defined in \eqref{eq:lambda_1_def}.
Given (\ref{eqR12}), this implies
\begin{equation}\label{eqRi-R0}
|R_{i,\varepsilon}-R_{i,0}|\leq C\varepsilon |\log
\varepsilon|^\frac{1}{2},\ \ i=1,2.
\end{equation}
\end{proposition}
\begin{proof}
We test the equation for $\eta_1$ in \eqref{eq:main_eta_1_eta_2} by
$\eta_1$ itself, integrate by parts (since $\eta_1 \in
H^1(\mathbb{R}^2)$), and then  subtract $\la_{1,0}$ from both sides
to obtain
\[
\la_{1,\ep}-\la_{1,0}=\int_{\R^2}\left\{
\ep^2|\nabla\eta_1|^2+\eta_1^2(|x|^2+g_1\eta_1^2+g\eta_2^2-\la_{1,0})
\right\} \,dx.
\]
%Notice that the boundary term vanishes in the integration by parts
%since, due to the Lemma \ref{lemma:L_infty_bounds}, we have
%\[
%\int_{\partial
%B_{R}}\eta_{1}\nabla\eta_{1}\cdot\nu\,d\sigma\leq
%C\frac{1+\lambda_{i,\ep}}{\ep}\int_{\partial
%B_{R}}\eta_{1}\,d\sigma \to0 \quad \text{ as } R\to\infty.
%\]
With calculations similar to the one used in the proof of Lemma
\ref{lemma:E_tilde_def}, we rewrite the right hand side of the
previous expression in the following form:
\[
\begin{split}
\int_{\R^2}\left\{ \ep^2 |\nabla\eta_1|^2+ g_1(\eta_1^2-a_1)^2
+g(\eta_1^2-a_1)(\eta_2^2-a_2)+g_1a_1(\eta_1^2-a_1) \right.\\ \left.
+ga_1(\eta_2^2-a_2) \right\} \,dx +g_1\Gamma \int_{\R^2\setminus
D_1} a_{1,0}^-\eta_1^2\,dx+ g \int_{\R^2\setminus D_2} \left(
a_{2,0}+\frac{g}{g_2}a_{1,0} \right)^-\eta_1^2\,dx.
\end{split}
\]
We notice that
\[
\begin{split}
\left|\int_{\R^2} a_1\left\{ g_1(\eta_1^2-a_1)+g(\eta_2^2-a_2)
\right\}\,dx\right| \leq \|a_1\|_{L^2(\R^2)} \left(
g_1\|\eta_1^2-a_1\|_{L^2(\R^2)} \right.\\ \left.
+g\|\eta_2^2-a_2\|_{L^2(\R^2)} \right).
\end{split}
\]
Hence by applying Proposition \ref{prop:energy_estimates} we obtain
the convergence of $\la_{1,\ep}$. The convergence of $\la_{2,\ep}$
can be proved similarly.
\end{proof}

\begin{rem}\label{rem3.4}The equivalent of Proposition \ref{prop:energy_estimates} and \ref{prop:convergence_lagrange_multipliers} hold when (\ref{eqannulus}) is assumed instead of
 (\ref{eq:condition_on_g_two_disks}). The only difference is that
  (\ref{estenergydisk}) has to be replaced by \[
\begin{split}
\int_{\{|x|\geq R_{1,0}\}} \eta_{1,\ep}^2a_{1,0}^-\,dx
+\int_{\{|x|\leq R_{2,0}^- \}} \eta_{2,\ep}^2 a_{2,0}^- \,dx
+\int_{\{|x|\geq R_{2,0}^+\}}(g\eta_{1,\ep}^2+g_2\eta_{2,\ep}^2)\left( a_{2,0}+\frac{g}{g_2}a_{1,0} \right)^-\,dx \\
\leq C\ep^2|\log\ep|.
\end{split}
\]\end{rem}

%%%%%%%%%%%%%%%%%%%%%%%%%%%%%%%%%%%%%%%%%%%%%%%%%%%%%%%%%%%%%%%%%%%%%%%%%%%%%%%%%%%%%%%%%%%%%%%%%%%%
%%%%%%%%%%%%%%%%%%%%%%%%%%%%%%%%%%%%%%%%%%%%%%%%%%%%%%%%%%%%%%%%%%%%%%%%%%%%%%%%%%%%%%%%%%%%%%%%%%%%

\section{Refined estimates for the energy minimizer without rotation}\label{secPerturb}

In this section we capture the fine behavior of the minimizer
$(\eta_{1,\varepsilon},\eta_{2,\varepsilon})$, as $\varepsilon\to
0$, by means of  a perturbation argument. Since this type of
approach is in principle not applicable to problems with integral
constraints, we argue indirectly as follows. First, given
$(\lambda_{1,\varepsilon},\lambda_{2,\varepsilon})$ as in the
previous section, for small $\varepsilon>0$, we construct a
positive radial solution of (\ref{eq:main_eta_1_eta_2}) ``near''
$(a_1,a_2)$ by a perturbation argument. Then, the uniqueness result
in Theorem \ref{thm:uniqueness} will imply that this solution
coincides with the unique positive minimizer of $E_\ep^0$ in
$\mathcal{H}$.

%\subsection{Reformulation of the energy minimization problem as a singular perturbed elliptic system without $L^2$-norm constraints}\label{subsecProblem}

\subsection{The main result concerning the minimizer without rotation}
\begin{theorem}\label{thmLong} Assume  that (\ref{eq:condition_on_g_thomas_fermi}),
 (\ref{eq:condition_on_g_two_disks}) and (\ref{eq14+}) hold. Let $(\eta_{1},\eta_{2})$ be the positive minimizer of $E_\ep^0$ in $\mathcal H$. Let $\la_{i,\ep}$ be the associated Lagrange multipliers in
\eqref{eq:main_eta_1_eta_2}.
 There exist constants $c,C,D>0$ and $\delta\in
\left(0,\frac{1}{4}\min\{R_{1,0}, R_{2,0}-R_{1,0} \}\right)$ such
that the following estimates hold:
\underline{\emph{Estimates for the Lagrange multipliers}}
\begin{equation}\label{eqthmLagrimproved}
|\lambda_{i,\varepsilon}-\lambda_{i,0}|\leq C|\log
\varepsilon|\varepsilon^2,\  \textrm{which\ implies\ that}\
|R_{i,\varepsilon}-R_{i,0}|\leq C |\log \varepsilon|\varepsilon^2,\
\ i=1,2;
\end{equation}

\underline{\emph{Outer estimates}}
 \begin{equation}\label{eqthmUnifaway1}
 \|\eta_{1,\varepsilon}-\sqrt{a_{1,\varepsilon}}\|_{L^\infty(|x|\leq
 R_{1,\varepsilon}-\delta)}+\|\eta_{2,\varepsilon}-\sqrt{a_{2,\varepsilon}}\|_{L^\infty(|x|\leq
 R_{1,\varepsilon}-\delta)}\leq C\varepsilon^2,
 \end{equation}
 and
 \begin{equation}\label{eqthmUnifaway2}
 \Big{\|}\eta_{2,\varepsilon}-\sqrt{a_{2,\varepsilon}+\frac{g}{g_2}a_{1,\varepsilon}}\Big{\|}_{L^\infty(R_{1,\varepsilon}+\delta\leq|x|\leq
 R_{2,\varepsilon}-\delta)}\leq C\varepsilon^2,
 \end{equation}
uniformly as $\varepsilon\to 0$;

 \underline{\emph{Algebraic decay
estimates}}
\begin{equation}\label{eqthmAlg1}
\eta_{1,\varepsilon}(r)-\sqrt{a_{1,\varepsilon}(r)}=\mathcal{O}\left(\varepsilon^2|r-R_{1,\varepsilon}|^{-\frac{5}{2}}\right),
\ \
\eta_{2,\varepsilon}(r)-\sqrt{a_{2,\varepsilon}(r)}=\mathcal{O}\left(\varepsilon^2|r-R_{1,\varepsilon}|^{-2}\right),
\end{equation}
if $r\in
[R_{1,\varepsilon}-\delta,R_{1,\varepsilon}-D\varepsilon^\frac{2}{3}]$,
and
\begin{equation}\label{eqthmAlg2}
\eta_{2,\varepsilon}(r)-\sqrt{a_{2,\varepsilon}(r)+\frac{g}{g_2}a_{1,\varepsilon}(r)}=\mathcal{O}\left(\varepsilon^2|r-R_{2,\varepsilon}|^{-\frac{5}{2}}\right)
\ \ \textrm{if}\  \ \ r\in
[R_{2,\varepsilon}-\delta,R_{2,\varepsilon}-D\varepsilon^\frac{2}{3}],
\end{equation}
 uniformly as $\varepsilon\to 0$;

\underline{\emph{Exponential decay estimates}}
\begin{equation}\label{eqthmdecay1}
\eta_{i,\varepsilon}(r)\leq
C\varepsilon^\frac{1}{3}\exp\left\{c\frac{R_{i,\varepsilon}-r}{\varepsilon^\frac{2}{3}}\right\},\
\ r\geq R_{i,\varepsilon},\ \ i=1,2,
\end{equation}
and
\begin{equation}\label{eqthmdecay2}
\eta_{2,\varepsilon}(r)=\sqrt{a_{2,\varepsilon}(r)+\frac{g}{g_2}a_{1,\varepsilon}(r)}+\mathcal{O}(\varepsilon^\frac{2}{3})
\exp\left\{c\frac{R_{1,\varepsilon}-r}{\varepsilon^\frac{2}{3}}\right\}
+\mathcal{O}(\varepsilon^2)\ \ \textrm{if}\ \ r\in
[R_{1,\varepsilon},R_{1,\varepsilon}+\delta],
\end{equation}
uniformly as $\varepsilon\to 0$;

 \underline{\emph{Inner estimates}}
\begin{equation}\label{eqthmInner1}
 \eta_{1,\varepsilon}(r)=\varepsilon^\frac{1}{3}(g_1\Gamma)^{-\frac{1}{6}}\beta_{1,\varepsilon}
 V\left((g_1\Gamma)^\frac{1}{3}\beta_{1,\varepsilon}\frac{r-R_{1,\varepsilon}}{\varepsilon^\frac{2}{3}}
 \right)+\left\{\begin{array}{ll}
           \mathcal{O}\left(\varepsilon+|r-R_{1,\varepsilon}|^\frac{3}{2}\right) & \textrm{if}\ \ r\in[R_{1,\varepsilon}-\delta,R_{1,\varepsilon}],   \\
             &     \\
           \mathcal{O}(\varepsilon)\exp\left\{c\frac{R_{1,\varepsilon}-r}{\varepsilon^\frac{2}{3}} \right\} & \textrm{if}\ \ r\in[R_{1,\varepsilon},
           R_{1,\varepsilon}+\delta],
         \end{array}
\right.
\end{equation}
 \begin{equation}\label{eqthmInner2}\begin{array}{rcl}
\eta_{2,\varepsilon}(r) & = & \sqrt{
a_{2,\varepsilon}(r)+\frac{g}{g_2}a_{1,\varepsilon}(r)-\frac{g}{g_2}(g_1\Gamma)^{-\frac{1}{3}}\varepsilon^\frac{2}{3}\beta_{1,\varepsilon}^2V^2\left((g_1\Gamma)^\frac{1}{3}\beta_{1,\varepsilon}\frac{r-R_{1,\varepsilon}}{\varepsilon^\frac{2}{3}}
 \right)}
 \\
       &   &   \\
      &  &+\left\{\begin{array}{ll}
             \mathcal{O}\left(\varepsilon^\frac{4}{3}+|r-R_{1,\varepsilon}|^2+\varepsilon
      |r-R_{1,\varepsilon}|^\frac{1}{2}+\varepsilon^\frac{1}{3}
       |r-R_{1,\varepsilon}|^\frac{3}{2}\right) & \textrm{if}\ \ r\in[R_{1,\varepsilon}-\delta,R_{1,\varepsilon}], \\
               &   \\
             \mathcal{O}(\varepsilon^\frac{4}{3}) \exp\left\{c\frac{R_{1,\varepsilon}-r}{\varepsilon^\frac{2}{3}}
             \right\}+\mathcal{O}(\varepsilon^2)
             & \textrm{if}\ \ r\in[R_{1,\varepsilon},R_{1,\varepsilon}+\delta],
           \end{array}\right.
\end{array}
 \end{equation}
and
 %\[
 %\eta_{2,\varepsilon}(r)=\left(a_{2,\varepsilon}(r)+\frac{g}{g_2}a_{1,\varepsilon}(r)-\frac{g}{g_2}\eta_{1,\varepsilon}^2(r)
 %\right)^\frac{1}{2} \right| \leq
 %C\varepsilon^\frac{7}{6}+\varepsilon
 %|r-R_{1,\varepsilon}|^\frac{1}{2}\ \textrm{if}\
 %|r-R_{1,\varepsilon}|\leq \delta,
 %\]
 \begin{equation}\label{eqthmInner3}
 \eta_{2,\varepsilon}(r)=\varepsilon^\frac{1}{3}g_2^{-\frac{1}{6}}\beta_{2,\varepsilon}
 V\left(g_2^\frac{1}{3}\beta_{2,\varepsilon}\frac{r-R_{2,\varepsilon}}{\varepsilon^\frac{2}{3}}
 \right)
 +\left\{\begin{array}{ll}
           \mathcal{O}\left(\varepsilon+|r-R_{2,\varepsilon}|^\frac{3}{2}\right) & \textrm{if}\ \ r\in[R_{2,\varepsilon}-\delta,R_{2,\varepsilon}],   \\
             &     \\
           \mathcal{O}(\varepsilon)\exp\left\{c\frac{R_{2,\varepsilon}-r}{\varepsilon^\frac{2}{3}} \right\} & \textrm{if}\ \ r\in[R_{2,\varepsilon},
           R_{2,\varepsilon}+\delta],
         \end{array}
\right.
 \end{equation}
 uniformly as $\varepsilon\to 0$,
where $V$ is the Hastings-McLeod solution \cite{hastingsbook} to the
Painlev\'{e}-II equation \cite{fokas}, namely the unique solution of
the boundary value problem
\begin{equation}\label{eqpainleve}
v''=v(v^2+s),\ s\in \mathbb{R};\ v(s)-\sqrt{-s}\to 0\ \textrm{as}\
s\to -\infty,\ v(s)\to 0 \ \textrm{as}\ s\to \infty,
\end{equation}
and
\[
\beta_{1,\varepsilon}=\left(-a_{1,\varepsilon}'(R_{1,\varepsilon})
\right)^\frac{1}{3},\ \
\beta_{2,\varepsilon}=\left(-a_{2,\varepsilon}'(R_{2,\varepsilon})-\frac{g}{g_2}a_{1,\varepsilon}'(R_{2,\varepsilon})
\right)^\frac{1}{3}.
\]

\end{theorem}
The proof of this theorem will be completed in Subsection
\ref{secProofs}. Note that in the case $g_1=g_2$,  then an analogous of this theorem holds.
It is simpler since $\eta_{1,\eps}=\eta_{2,\eps}$ so that $R_{1,\eps}=R_{2,\eps}$. The result is just a consequence of Theorem \ref{thmGSscalar} in
 the appendix.

\subsection{An approximate
solution}\label{subApprox} In this subsection we  construct a
sufficiently good approximate solution
$(\check{\eta}_{1,\varepsilon},\check{\eta}_{2,\varepsilon})$ to the
problem (\ref{eqsystem}) such that
$\check{\eta}_{i,\varepsilon}>0$ and
$\check{\eta}_{i,\varepsilon}\to \sqrt{a_i}$, uniformly on
$\mathbb{R}^2$, as $\varepsilon\to 0$, $i=1,2$. The building blocks
of our construction will be the unique positive solutions
$\hat{\eta}_{1,\varepsilon}$ and $\hat{\eta}_{2,\varepsilon}$ of the
reduced problems (\ref{eqgroundstate1}) and (\ref{eqgroundstate2})
respectively.

\subsubsection{The reduced problems}

The asymptotic behavior  of $\hat{\eta}_{i,\varepsilon}$, $i=1,2$,
as $\varepsilon\to 0$, can be deduced, after a proper constant
re-scaling, from the following proposition which is a special case
of the more general result that we prove in  Appendix
\ref{secAppenGS}.
\begin{proposition}\label{proGS}
Suppose that $\lambda_\varepsilon$ satisfy $\lambda_\varepsilon \to
\lambda_0$ as $\varepsilon\to 0$, for some $\lambda_0>0$, and let
\[
 A_{\varepsilon}(x)=
                \lambda_\varepsilon-\mu|x|^2, \ \ x\in \mathbb{R}^2,
\]
with $\mu>0$ independent of $\varepsilon$.

There exists a unique positive solution $u_\varepsilon$ of the
problem
\[
\varepsilon^2 \Delta u =u\left(u^2-A_\varepsilon(x) \right),\ \ x\in
\mathbb{R}^2;\ \ u(x)\to 0 \ \textrm{as}\ |x|\to \infty.
\]
This solution is radially symmetric and, for small $\varepsilon>0$,
satisfies the following properties:
\begin{equation}\label{eqGSuUnif}
\|u_\varepsilon-\sqrt{A_\varepsilon^+}\|_{L^\infty(\mathbb{R}^2)}\leq
C\varepsilon^\frac{1}{3},\ \
\|u_\varepsilon-\sqrt{A_\varepsilon}\|_{C^2\left(|x|\leq
r_\varepsilon-\delta\right)}\leq C\varepsilon^2, \end{equation}
where
$r_\varepsilon=\left(\mu^{-1}\lambda_\varepsilon\right)^\frac{1}{2}$,
for some $\delta \in \left(0,\frac{1}{4}r_0 \right)$,
$r_0=(\mu^{-1}\la_0)^{1/2}$, and
\begin{equation}\label{eq:proGS2}
u_\varepsilon(r)\leq C
\varepsilon^\frac{1}{3}\exp\left\{c\varepsilon^{-\frac{2}{3}}(r_{\varepsilon}-r)\right\},\
\ r\geq r_\varepsilon.
\end{equation}
%Moreover, we have that
%\begin{equation}\label{eqGScomplex}
%\left|u_{\varepsilon}(r)-\sqrt{A_\varepsilon(r)}\right|\leq
%C\varepsilon\sqrt{A_\varepsilon(r)},\ \ 0\leq r\leq
%r_{\varepsilon}-\varepsilon^\frac{1}{3}.
%\end{equation}
 In fact, the potential of the
associated linearized operator satisfies the  lower bound
\begin{equation}\label{eq:proGS3}
3u_\varepsilon^2(r)-A_\varepsilon(r)\geq \left\{\begin{array}{ll}
                                                              c |r-r_{\varepsilon}|+c\varepsilon^\frac{2}{3} &
                                                              \textrm{if}\ \ |r-r_{\varepsilon}|\leq
                                                               \delta, \\
                                                                 &   \\
                                                               c &
                                                               \textrm{otherwise}.
                                                             \end{array}
\right.
\end{equation}

 More precisely,  we
have
\begin{equation}\label{eq:proGS4}
u_\varepsilon(r)=\varepsilon^\frac{1}{3}\beta_{\varepsilon}
V\left(\beta_{\varepsilon}
\frac{r-r_{\varepsilon}}{\varepsilon^\frac{2}{3}}
\right)+\left\{\begin{array}{ll}
                 \mathcal{O}\left(\varepsilon+|r-r_{\varepsilon}|^\frac{3}{2}\right) & \textrm{if}\ \ r_\varepsilon-\delta\leq r\leq r_\varepsilon, \\
                   &   \\
                 \mathcal{O}(\varepsilon) \exp\left\{-c\frac{|r-r_{\varepsilon}|}{\varepsilon^\frac{2}{3}} \right\} &
                 \textrm{if}\  \  r_\varepsilon\leq r \leq r_\varepsilon+\delta,
               \end{array}
\right.
\end{equation}
 where $V$ is the Hastings-McLeod solution, as described in
 (\ref{eqpainleve}), and
\[\beta_{\varepsilon}=\left(-A_{\varepsilon}'(r_{\varepsilon})
\right)^\frac{1}{3}.\] Furthermore, we have
\begin{equation}\label{eqeta1groundinner2}
u_\varepsilon'(r)=\varepsilon^{-\frac{1}{3}}\beta_{\varepsilon}^2
V'\left(\beta_{\varepsilon}
\frac{r-r_{\varepsilon}}{\varepsilon^\frac{2}{3}}
\right)+\mathcal{O}\left(\varepsilon^\frac{1}{3}+|r-r_{\varepsilon}|^\frac{1}{2}\right)\
\ \textrm{if}\ \ |r-r_\varepsilon |\leq \delta,
\end{equation}
uniformly, as $\varepsilon\to 0$. Moreover, there exists $D>0$ such
that the following estimates hold for $r\in [r_\varepsilon-\delta,
r_\varepsilon-D\varepsilon^\frac{2}{3}]$:
\begin{equation}\label{eqGSuAlg}
\begin{array}{c}
 u_\varepsilon
(r)-\sqrt{A_\varepsilon(r)}=\varepsilon^2\mathcal{O}(|r-r_{\varepsilon}|^{-\frac{5}{2}}),
\ \ \
u'_\varepsilon-\left(\sqrt{A_\varepsilon}\right)'=\varepsilon^2\mathcal{O}(|r-r_{\varepsilon}|^{-\frac{7}{2}}), \\
    \\
 \Delta
u_\varepsilon-\Delta\left(\sqrt{A_\varepsilon}\right)=\varepsilon^2\mathcal{O}(|r-r_{\varepsilon}|^{-\frac{9}{2}}),
\end{array}
\end{equation}
uniformly, as $\varepsilon\to 0$.
\end{proposition}

%\begin{rem}
%In fact, the first estimate in (\ref{eqGSuAlg}) follows from the
%stronger estimate:
%\begin{equation}\label{eqstronger}
%u_\varepsilon(r)=\left\{A_{\varepsilon}(r)+\beta_\varepsilon^2\left[
%\beta_\varepsilon
%(r-r_\varepsilon)+\varepsilon^\frac{2}{3}V^2\left(\beta_\varepsilon\frac{r-r_\varepsilon}{\varepsilon^\frac{2}{3}}
%\right) \right]
%\right\}^\frac{1}{2}+\mathcal{O}\left(\varepsilon^2|r-r_\varepsilon|^{-\frac{3}{2}}
%\right),
%\end{equation}
%uniformly on
%$[r_\varepsilon-\delta,r_\varepsilon-D\varepsilon^\frac{2}{3}]$, as
%$\varepsilon\to 0$ (see \cite[Prop. 3.6]{KaraliSourdisGround}).
%\end{rem}

%\begin{rem}
%In this article, we will use the above proposition with
%$|\lambda_\varepsilon-\lambda_0|\leq C|\log
%\varepsilon|^\frac{1}{2}\varepsilon$.
%\end{rem}

\subsubsection{Gluing approximate solutions}
 Consider a one-dimensional smooth cutoff function $\zeta$ such that
\begin{equation}\label{eqzeta}\zeta(t)=1\ \ \textrm{if}\ \ t\leq R_\varepsilon-\delta;\ \ \zeta(t)=0\ \ \textrm{if}\ \ t\geq
R_\varepsilon,\end{equation} where, for convenience, we have denoted
\begin{equation}\label{eqReps}R_\varepsilon=\frac{R_{1,\varepsilon}+R_{2,\varepsilon}}{2},\end{equation}
and $\delta>0$ is a small number that is independent of small
$\varepsilon>0$. Note that $\zeta$ can be chosen independent
of  $\varepsilon>0$ as well. We recall that $\hat{\eta}_{1,\varepsilon}$, $\hat{\eta}_{2,\varepsilon}$
 are the solutions of (\ref{eqgroundstate}).
 In view of (\ref{eq:proGS2}), let
\begin{equation}\label{eqeta1check}
\check{\eta}_{1,\varepsilon}(x)=\zeta(|x|)\hat{\eta}_{1,\varepsilon}(x),\
\ x\in \mathbb{R}^2.
\end{equation}
 Then, let
\begin{equation}\label{eqeta2cut}
\tilde{\eta}_{2,\varepsilon}(x)=\left(a_{2,\varepsilon}(x)+\frac{g}{g_2}a_{1,\varepsilon}(x)-\frac{g}{g_2}\check{\eta}_{1,\varepsilon}^2
\right)^\frac{1}{2},\ \ |x|\leq R_\varepsilon+\delta.
\end{equation}
The motivation for this comes from neglecting the term
$\varepsilon^2 \Delta \eta_2$ in
(\ref{eqsystem2}), since it is expected to be of higher order,
compared to the other terms, in the region $|x|\leq
R_{2,\varepsilon}-\delta$.

From (\ref{eqGSuUnif}), it follows that
\[
\hat{\eta}_{2,\varepsilon}(x)-\left(a_{2,\varepsilon}(x)+\frac{g}{g_2}a_{1,\varepsilon}(x)
\right)^\frac{1}{2}=\mathcal{O}(\varepsilon^2),\ \textrm{in}\
C^2\left(R_\varepsilon\leq |x|\leq R_\varepsilon+\delta\right),\
\textrm{as}\ \varepsilon\to 0.
\]
 In other words, recalling
(\ref{eqzeta}) and (\ref{eqeta2cut}), we have that
\begin{equation}\label{eqeta2gluing}
\hat{\eta}_{2,\varepsilon}(x)-\tilde{\eta}_{2,\varepsilon}(x)=\mathcal{O}(\varepsilon^2),\
\textrm{in}\ C^2\left(R_\varepsilon\leq |x|\leq
R_\varepsilon+\delta\right),\ \textrm{as}\ \varepsilon\to 0.
\end{equation}
Thus, we can smoothly interpolate between
$\tilde{\eta}_{2,\varepsilon}$ and $\hat{\eta}_{2,\varepsilon}$ to
obtain a new approximation $\check{\eta}_{2,\varepsilon}$ such that
\begin{equation}\label{eqeta2check}
\check{\eta}_{2,\varepsilon}(x)=\left\{\begin{array}{ll}
                                     \tilde{\eta}_{2,\varepsilon}(x), & |x|\leq R_\varepsilon, \\
                                       &  \\
                                       \tilde{\eta}_{2,\varepsilon}(x)+\mathcal{O}_{C^2}(\varepsilon^2), &  R_\varepsilon \leq|x|\leq R_\varepsilon+\delta, \\
                                       &  \\
                                      \hat{\eta}_{2,\varepsilon}(x),
                                      &|x|\geq R_\varepsilon+\delta.
                                    \end{array}
 \right.
\end{equation}

To conclude, we define our approximate solution of the system
(\ref{eqsystem}), for small $\varepsilon>0$, to be the pair
$(\check{\eta}_{1,\varepsilon}, \check{\eta}_{2,\varepsilon})$, as
described by (\ref{eqeta1check}) and (\ref{eqeta2check}). We point
out that this approximation satisfies the desired limiting behavior
\begin{equation}\label{eqsingularlimit}
(\check{\eta}_{1,\varepsilon}, \check{\eta}_{2,\varepsilon})\to
(\sqrt{a_1},\sqrt{a_2}),\ \textrm{uniformly in}\ \mathbb{R}^2,\
\textrm{as}\ \varepsilon \to 0,
\end{equation}
where $a_1$ and $a_2$ are as in (\ref{eq:a_i_def}). Moreover,
estimates that quantify this convergence can be derived easily from
the corresponding ones that are available for the ground states
$\hat{\eta}_{1,\varepsilon}$ and $\hat{\eta}_{2,\varepsilon}$ from
Proposition \ref{proGS}.

\subsection{Estimates for the error on the approximate solution}\label{subsecError}

The remainder that is left when substituting the approximate
solution $(\check{\eta}_{1,\varepsilon},
\check{\eta}_{2,\varepsilon})$ to the system (\ref{eqsystem}) is
\begin{equation}\label{eqError}
\mathcal{E}(\check{\eta}_{1,\varepsilon},
\check{\eta}_{2,\varepsilon}) \equiv \left(\begin{array}{c}
                                       E_1 \\
                                         \\
                                       E_2
                                     \end{array}\right)
\equiv\left(\begin{array}{c}
  -\varepsilon^2\Delta  \check{\eta}_{1,\varepsilon}+g_1\check{\eta}_{1,\varepsilon}
  \left(\check{\eta}_{1,\varepsilon}^2-a_{1,\varepsilon} \right)+g\check{\eta}_{1,\varepsilon}
  \left(\check{\eta}_{2,\varepsilon}^2-a_{2,\varepsilon} \right) \\
   \\
-\varepsilon^2\Delta
\check{\eta}_{2,\varepsilon}+g_2\check{\eta}_{2,\varepsilon}\left(\check{\eta}_{2,\varepsilon}^2-a_{2,\varepsilon}
\right)+g\check{\eta}_{2,\varepsilon}\left(\check{\eta}_{1,\varepsilon}^2-a_{1,\varepsilon}
\right)
\end{array}\right).
\end{equation}

The next proposition provides estimates for the $L^2$-norms of
$E_i$, $i=1,2$, which  follow from some delicate pointwise
estimates that will be established in the process.

\begin{proposition}\label{proerror}
The following estimates hold for small $\varepsilon>0$:
\begin{equation}\label{eqerrorestims}
\|E_1\|_{L^2(\mathbb{R}^2)}\leq Ce^{-c\varepsilon^{-\frac{2}{3}}},\
\|E_2\|_{L^2(|x|<R_\varepsilon)}\leq C\varepsilon^\frac{5}{3}\
\textrm{and}\ \|E_2\|_{L^2(|x|>R_\varepsilon)}\leq C\varepsilon^2.
\end{equation}
\end{proposition}
\begin{proof}
It follows from the construction of $\check{\eta}_{1,\varepsilon}$
and $\check{\eta}_{2,\varepsilon}$, via (\ref{eq:proGS2}), that
\begin{equation}\label{eqE1}
E_1=0\ \textrm{if}\ |x|\leq R_\varepsilon-\delta\ \textrm{or}\
|x|\geq R_\varepsilon; \ \ |E_1|\leq
Ce^{-c\varepsilon^{-\frac{2}{3}}}\ \textrm{if}\ R_\varepsilon-\delta
\leq|x| \leq R_\varepsilon.
\end{equation}
 On the other side, we have
\begin{equation}\label{eqE2}
E_2=-\varepsilon^2 \Delta \tilde{\eta}_{2,\varepsilon}\ \
\textrm{if}\ |x|\leq R_\varepsilon-\delta;\end{equation}
\[E_2=-\varepsilon^2 \Delta
\tilde{\eta}_{2,\varepsilon}+\mathcal{O}(e^{-c\varepsilon^{-\frac{2}{3}}})\
\textrm{uniformly\ if}\ R_\varepsilon-\delta\leq |x|\leq
R_\varepsilon,\ \textrm{as}\ \varepsilon\to 0;\] \[ E_2=0\
\textrm{if}\ |x|\geq R_\varepsilon+\delta;\ \ |E_2|\leq
C\varepsilon^2 \ \textrm{if}\ R_\varepsilon\leq |x|\leq
R_\varepsilon+\delta\ (\textrm{recall}\ (\ref{eqeta2gluing})).
\]

In view of the previous observations, we only have to show the
second relation in (\ref{eqerrorestims}). In fact, since $\Delta
\tilde{\eta}_{2,\varepsilon}$  remains uniformly bounded if
$|r-R_{1,\varepsilon}|\geq \delta$ as $\varepsilon\to 0$, it
suffices to show that
\begin{equation}\label{eqerrorgoal}
\|\Delta
\tilde{\eta}_{2,\varepsilon}\|_{L^2\left(|r-R_{1,\varepsilon}|<
\delta\right)}\leq C\varepsilon^{-\frac{1}{3}}\ \ \textrm{for\
small}\ \varepsilon>0.
\end{equation}
%(keep in mind that, as we have already pointed out, the radial
%derivative of $\eta_{2,0}$ has a finite jump discontinuity across
%$|x|=R_{1,0}$, and recall (\ref{eqsingularlimit})).

It follows readily from (\ref{eqeta2cut}) and (\ref{eqeta2check})
that
\begin{equation}\label{eqerrorBasic}
|\Delta \tilde{\eta}_{2,\varepsilon}|\leq C
\hat{\eta}_{1,\varepsilon}^2|\nabla
\hat{\eta}_{1,\varepsilon}|^2+C\left|\hat{\eta}_{1,\varepsilon}\Delta
\hat{\eta}_{1,\varepsilon}+|\nabla \hat{\eta}_{1,\varepsilon}|^2
\right|+C\ \ \textrm{if}\ |r-R_{1,\varepsilon}|<\delta.
\end{equation}
Next, we  estimate the terms in the right-hand side by making
use of Proposition \ref{proGS}. To this end, we need to derive a
relation for $\Delta \hat{\eta}_{1,\varepsilon}$ in terms of the
Hastings-McLeod solution near $R_{1,\varepsilon}$. Making use of
(\ref{eqgroundstate1}), (\ref{eq:proGS4}) and the natural bound
$|V(s)|\leq C(|s|^\frac{1}{2}+1),\ s\in \mathbb{R}$, setting
$\tilde{\varepsilon}=(g_1 \Gamma)^{-\frac{1}{2}}\varepsilon$, after
a tedious calculation, we arrive at
\begin{equation}\label{eqerrorLong}
\begin{array}{lll}
  \hat{\eta}_{1,\varepsilon} \Delta
\hat{\eta}_{1,\varepsilon} & = &
\tilde{\varepsilon}^{-\frac{2}{3}}\beta_\varepsilon^4
V^2\left(\beta_\varepsilon
\frac{r-R_{1,\varepsilon}}{\tilde{\varepsilon}^\frac{2}{3}}
\right)\left[V^2\left(\beta_\varepsilon
\frac{r-R_{1,\varepsilon}}{\tilde{\varepsilon}^\frac{2}{3}}
\right)+\beta_\varepsilon
\frac{r-R_{1,\varepsilon}}{\tilde{\varepsilon}^\frac{2}{3}} \right]+ \\
    &   &   \\
    &   &
    +\mathcal{O}\left(\varepsilon^{-2}|r-R_{1,\varepsilon}|^3
    +\varepsilon^{-\frac{4}{3}}|r-R_{1,\varepsilon}|^2+\varepsilon^{-1}|r-R_{1,\varepsilon}|^\frac{3}{2}\right)+
     \\
    &   &   \\
    &   & +\mathcal{O}\left(
\varepsilon^{-\frac{1}{3}}|r-R_{1,\varepsilon}|^\frac{1}{2}
+\varepsilon^{-\frac{2}{3}}|r-R_{1,\varepsilon}|+1+\varepsilon^{-\frac{5}{3}}|r-R_{1,\varepsilon}|^\frac{5}{2}
\right),
\end{array}
\end{equation}
uniformly if $|r-R_{1,\varepsilon}|\leq \delta$, as $\varepsilon \to
0$, where
$\beta_\varepsilon^3=-a_{1,\varepsilon}'(R_{1,\varepsilon})$.
Similarly, but with considerably less effort, it follows from
(\ref{eqeta1groundinner2}), and the bound $\left|V'(s)\right|\leq
C\left(|s|+1 \right)^{-\frac{1}{2}}$, that
\begin{equation}\label{eqerrorLong2}
|\nabla
\hat{\eta}_{1,\varepsilon}(r)|^2=\tilde{\varepsilon}^{-\frac{2}{3}}\beta_\varepsilon^4
\left[V'\left(\beta_\varepsilon
\frac{r-R_{1,\varepsilon}}{\tilde{\varepsilon}^\frac{2}{3}}
\right)\right]^2+\mathcal{O}(1), \ \ \textrm{uniformly\ if}\
|r-R_{1,\varepsilon}|\leq \delta,\ \textrm{as}\ \varepsilon\to 0.
\end{equation}
%It is clear that the term $\varepsilon^{-2}|r-R_{1,\varepsilon}|^3$
%in (\ref{eqerrorLong}) (for example) is not small enough away from
%$R_{1,\varepsilon}$, as required by the estimate (\ref{eqerrorgoal})
%that we want to prove (recall also (\ref{eqerrorBasic})). This
%observation implies that we have to complement these ``inner''
%estimates with corresponding ``outer'' ones outside of a
%neighborhood of $R_{1,\varepsilon}$ (which will be determined in the
%process). This is what we will do next.

 It follows readily from the estimates in
(\ref{eqGSuAlg}) that
\begin{equation}\label{eqerorAlg-}
\hat{\eta}_{1,\varepsilon}^2 |\nabla
\hat{\eta}_{1,\varepsilon}|^2\leq C,
\end{equation}
and
\begin{equation}\label{eqerrorAlg}
\hat{\eta}_{1,\varepsilon} \Delta \hat{\eta}_{1,\varepsilon}+|\nabla
\hat{\eta}_{1,\varepsilon}|^2=\mathcal{O}\left(1+\varepsilon^2
|r-R_{1,\varepsilon}|^{-4} \right),
\end{equation}
uniformly in $-\delta\leq r-R_{1,\varepsilon}\leq
-C\varepsilon^\frac{2}{3}$, as $\varepsilon \to 0$. Keep in mind
that our eventual goal is to show (\ref{eqerrorgoal}). In view of
(\ref{eqerrorBasic}), the above relations imply the partial estimate
\begin{equation}\label{eqerrorgoal1}
\|\Delta \tilde{\eta}_{2,\varepsilon} \|_{L^\infty(-\delta \leq
r-R_{1,\varepsilon}\leq -\varepsilon^\frac{7}{12})}\leq
C\varepsilon^{-\frac{1}{3}}\ \ \textrm{for\ small}\ \varepsilon>0.
\end{equation}
On the other side, by the exponential decay of
$\hat{\eta}_{1,\varepsilon}$ for $r>R_{1,\varepsilon}$, we certainly
have that \begin{equation}\label{eqerrorgoal2}\|\Delta
\tilde{\eta}_{2,\varepsilon} \|_{L^\infty(
\varepsilon^\frac{7}{12}\leq r-R_{1,\varepsilon}\leq \delta)}\leq
C\varepsilon^{-\frac{1}{3}}\ \ \textrm{for\ small}\ \varepsilon>0.
\end{equation}

In the remaining  interval
$(R_{1,\varepsilon}-\varepsilon^\frac{7}{12},R_{1,\varepsilon}+\varepsilon^\frac{7}{12})$
 we use the inner estimates (\ref{eq:proGS4}),
(\ref{eqerrorLong}), and (\ref{eqerrorLong2}), which in particular
imply that
\begin{equation}\label{eqerror-}
\hat{\eta}_{1,\varepsilon}^2 |\nabla
\hat{\eta}_{1,\varepsilon}|^2\leq C,
\end{equation}
and
\begin{equation}\label{eqerror-+}\begin{array}{rcl}
                                   \hat{\eta}_{1,\varepsilon} \Delta \hat{\eta}_{1,\varepsilon}+|\nabla
\hat{\eta}_{1,\varepsilon}|^2 & = &
\tilde{\varepsilon}^{-\frac{2}{3}}\beta_\varepsilon^4
\left\{V\left(\beta_\varepsilon
\frac{r-R_{1,\varepsilon}}{\tilde{\varepsilon}^\frac{2}{3}}
\right)V''\left(\beta_\varepsilon
\frac{r-R_{1,\varepsilon}}{\tilde{\varepsilon}^\frac{2}{3}}
\right)+\left[V'\left(\beta_\varepsilon
\frac{r-R_{1,\varepsilon}}{\tilde{\varepsilon}^\frac{2}{3}}
\right)\right]^2\right\}+ \\
                                     &   &   \\
                                     &   & +\mathcal{O}\left(\varepsilon^{-\frac{1}{4}}\right),
                                 \end{array}
\end{equation}
uniformly if $|r-R_{1,\varepsilon}|\leq \varepsilon^\frac{7}{12}$,
as $\varepsilon\to 0$ (for obtaining the last relation, we have also
used (\ref{eqpainleve})). In order to proceed, we need the
following easy estimate:
\begin{equation}\label{eqpainlevecancel}
V(s)V''(s)+\left[V'(s) \right]^2=\mathcal{O}(|s|^{-4})\ \
\textrm{as}\ |s|\to \infty,
\end{equation}
which follows from the asymptotic behavior
$V(s)=(-s)^\frac{1}{2}+\mathcal{O}\left(|s|^{-\frac{5}{2}}\right)$
as $s\to -\infty$, and from the super-exponential decay of $V$ and
its derivatives  as $s\to \infty$. Now, by (\ref{eqerror-}),
(\ref{eqerror-+}), and (\ref{eqpainlevecancel}), via
(\ref{eqerrorBasic}), we deduce that
\begin{equation}\label{eqerrorgoal3}
\|\Delta \tilde{\eta}_{2,\varepsilon} \|_{L^2(|
r-R_{1,\varepsilon}|\leq \varepsilon^\frac{7}{12})}\leq
C\varepsilon^{-\frac{1}{3}}\ \ \textrm{for\ small}\ \varepsilon>0.
\end{equation}
Finally, the desired estimate (\ref{eqerrorgoal}) follows directly
from (\ref{eqerrorgoal1}), (\ref{eqerrorgoal2}) and
(\ref{eqerrorgoal3}).\end{proof}

\subsection{Linear analysis}\label{subsectionLinear}

 In this part of the paper we are going to
study the linearization of (\ref{eqsystem}) about the approximate
solution
$(\check{\eta}_{1,\varepsilon},\check{\eta}_{2,\varepsilon})$,
namely the linear operator
\begin{equation}\label{eqlinearization--}
\mathcal{L}(\varphi,\psi)\equiv \left(\begin{array}{l}
                                        -\varepsilon^2 \Delta \varphi +\left[g_1(3 \check{\eta}_{1,\varepsilon}^2
                                        -a_{1,\varepsilon})+g(\check{\eta}_{2,\varepsilon}^2-a_{2,\varepsilon})
                                        \right]
                                        \varphi+2g\check{\eta}_{1,\varepsilon}\check{\eta}_{2,\varepsilon}\psi  \\
                                          \\
-\varepsilon^2 \Delta \psi +\left[g_2(3
\check{\eta}_{2,\varepsilon}^2
                                        -a_{2,\varepsilon})+g(\check{\eta}_{1,\varepsilon}^2-a_{1,\varepsilon})
                                        \right]
                                        \psi+2g\check{\eta}_{1,\varepsilon}\check{\eta}_{2,\varepsilon}\varphi\end{array}\right),
\end{equation}
for $(\varphi,\psi)\in D(\mathcal{L})=\left\{(u,v)\in
H^2(\mathbb{R}^2)\times H^2(\mathbb{R}^2)\  : \
\int_{\mathbb{R}^2}^{}|x|^2(u^2+v^2)dx<\infty \right\}$. By
Friedrichs extension, the
operator $\mathcal{L}$ is self-adjoint in $L^2(\mathbb{R}^2)\times
L^2(\mathbb{R}^2)$ with domain $D(\mathcal{L})$.

%. We write the corresponding linearized operator conveniently as
%\begin{equation}\label{eqlinearization}
%\mathcal{L}(\varphi,\psi)\equiv \left(\begin{array}{l}
%                                        -\varepsilon^2 \Delta \varphi +\left[\left(g_1-\frac{g^2}{g_2} \right)(3\check{\eta}_{1,\varepsilon}^2
%                                        -a_{1,\varepsilon}) \right]
%                                        \varphi+2\frac{g^2}{g_2}\check{\eta}_{1,\varepsilon}^2\varphi+2g\check{\eta}_{1,\varepsilon}\check{\eta}_{2,\varepsilon}\psi  \\
%                                          \\
%
%         -\varepsilon^2 \Delta \psi+2g_2 \check{\eta}_{2,\varepsilon}^2\psi+2g\check{\eta}_{1,\varepsilon}\check{\eta}_{2,\varepsilon}\varphi                             \end{array}
% \right),
%\end{equation}
%for $\varphi,\psi\in H^2(\mathbb{R}^2)$.
\subsubsection{Energy estimates for $\mathcal{L}$} We  estimate from below the quotient
\[
\frac{\left(\mathcal{L}(\varphi,\psi),(\varphi,\psi)
\right)}{\left<(\varphi,\psi),(\varphi,\psi) \right>},\ \
(\varphi,\psi)\in D(\mathcal{L})\setminus (0,0),
\]
which turns out to be positive, where $(\cdot,\cdot)$ symbolizes the
usual inner product in $L^2(\mathbb{R}^2)\times L^2(\mathbb{R}^2)$
while $<\cdot,\cdot>$ is a suitably weighted
 one.  In turn, these lower bounds  provide
a-priori  estimates for the problem
$\mathcal{L}(\varphi,\psi)=(f_1,f_2)$.

\emph{Energy estimates in $B_{R_\varepsilon}$}. In the sequel, we
carry out this  plan in detail  in the domain
$B_{R_\varepsilon}$. Analogous results can be deduced in
$\mathbb{R}^2\setminus B_{R_\varepsilon}$ which we will describe
later.

In view of (\ref{eqeta2cut}) and (\ref{eqeta2check}), which imply
that $
\check{\eta}_{2,\varepsilon}^2-a_{2,\varepsilon}=-\frac{g}{g_2}(\check{\eta}_{1,\varepsilon}^2-
a_{1,\varepsilon})$ in $B_{R_\varepsilon}$, we can conveniently
rewrite (\ref{eqlinearization--}) in $B_{R_\varepsilon}$ as
 \begin{equation}\label{eqlinearization}
\mathcal{L}(\varphi,\psi)=\left(\begin{array}{l}
                                         -\varepsilon^2 \Delta \varphi +\left[\left(g_1-\frac{g^2}{g_2} \right)(3\check{\eta}_{1,\varepsilon}^2
                                        -a_{1,\varepsilon}) \right]
                                        \varphi+2\frac{g^2}{g_2}\check{\eta}_{1,\varepsilon}^2\varphi+2g\check{\eta}_{1,\varepsilon}\check{\eta}_{2,\varepsilon}\psi  \\
                                          \\

         -\varepsilon^2 \Delta \psi+2g_2 \check{\eta}_{2,\varepsilon}^2\psi+2g\check{\eta}_{1,\varepsilon}\check{\eta}_{2,\varepsilon}\varphi                             \end{array}
 \right),
\end{equation}
for $\varphi,\psi\in H^2(B_{R_\varepsilon})$ (the reason for not
adding the similar terms in the first row is  to keep the
linearization of (\ref{eqgroundstate1}) about
$\check{\eta}_{1,\varepsilon}$ in the beginning).

\begin{proposition}\label{proL}
The following a-priori estimates hold: Suppose that
\[\varphi,\ \psi\in H^2_{N}(B_{R_\varepsilon})\equiv\left\{v\in
H^2(B_{R_\varepsilon})\ : \ \nu\cdot \nabla v=0\ \ \textrm{on}\
\partial B_{R_\varepsilon} \right\},\]
where $\nu$ denotes the outer unit  normal vector to $\partial
B_{R_\varepsilon}$,  satisfy
\begin{equation}\label{eqLev1}
\mathcal{L}(\varphi,\psi)=\lambda(\varepsilon^\frac{2}{3}\varphi,\psi)
\ \textrm{and}\
\int_{B_{R_\varepsilon}}^{}\left(\varepsilon^\frac{2}{3}\varphi^2+\psi^2
\right)dx=1,
\end{equation}
or
\begin{equation}\label{eqLev2}
\mathcal{L}(\varphi,\psi)=\lambda(\check{\eta}_{1,\varepsilon}^2\varphi,\psi)
\ \textrm{and}\
\int_{B_{R_\varepsilon}}^{}\left(\check{\eta}_{1,\varepsilon}^2\varphi^2+\psi^2
\right)dx=1,
\end{equation}
then
\[
\lambda\geq c.
\]
\end{proposition}
\begin{proof}
We use the following estimates: \begin{equation}\label{eqLpot1}
3\check{\eta}_{1,\varepsilon}^2-a_{1,\varepsilon}\geq
\left\{\begin{array}{ll}
                                                              c |r-R_{1,\varepsilon}|+c\varepsilon^\frac{2}{3}, & \textrm{if}\ |r-R_{1,\varepsilon}|\leq
                                                               \delta, \\
                                                                 &   \\
                                                               c, &
                                                               \textrm{otherwise},
                                                             \end{array}
\right.
\end{equation}
and
\begin{equation}\label{eqLpot2}
\check{\eta}_{1,\varepsilon}^2\leq \left\{\begin{array}{ll}
                                            C\varepsilon^\frac{2}{3}+C|r-R_{1,\varepsilon}|, & \textrm{if}\ |r-R_{1,\varepsilon}|\leq \delta, \\
                                              &   \\
                                            C, &\textrm{otherwise},
                                          \end{array}
\right.
\end{equation}
which are inherited from (\ref{eqGSuUnif}), \eqref{eq:proGS3} and (\ref{eq:proGS4}).
%(see (\ref{eqGSpotlower}), (\ref{eqeta1groundinner}) and
%(\ref{eqeta1groundouter})).
In particular,  observe that
\begin{equation}\label{eqLpot3}
3\check{\eta}_{1,\varepsilon}^2-a_{1,\varepsilon}\geq c
\check{\eta}_{1,\varepsilon}^2,\ \ x\in \mathbb{R}^2.
\end{equation} Note also that
\begin{equation}\label{eqpotlowerc}
2g_2 \check{\eta}_{2,\varepsilon}^2\geq c \ \ \textrm{on}\
\bar{B}_{R_\varepsilon}.
\end{equation}

Suppose that (\ref{eqLev1}) holds. Testing  by $(\varphi,\psi)$, in
the usual sense of $L^2(B_{R_\varepsilon})\times
L^2(B_{R_\varepsilon})$, and integrating by parts the resulting
identity, we obtain that
\begin{equation}\label{eqLbilin}
\int_{B_{R_\varepsilon}}^{}\left[\varepsilon^2|\nabla
\varphi|^2+\varepsilon^2|\nabla \psi|^2 +\left(g_1-\frac{g^2}{g_2}
\right)(3\check{\eta}_{1,\varepsilon}^2-a_{1,\varepsilon})\varphi^2+2\left(\frac{g}{\sqrt{g_2}}\check{\eta}_{1,\varepsilon}\varphi
+\sqrt{g_2}\check{\eta}_{2,\varepsilon}\psi
\right)^2\right]dx=\lambda.
\end{equation}
In turn, using (\ref{eqLpot1}) and (\ref{eqLpot3}), we find that
\begin{equation}\label{eqLbasic}
\int_{B_{R_\varepsilon}}^{}\left[\varepsilon^2|\nabla
\varphi|^2+\varepsilon^2|\nabla \psi|^2
+c\varepsilon^\frac{2}{3}\varphi^2+c\check{\eta}_{1,\varepsilon}^2\varphi^2\right]dx\leq\lambda.
\end{equation}
On the other side, the second equation of the system in
(\ref{eqLev1}) can be written as
\[
-\varepsilon^2\Delta
\psi+2g_2\check{\eta}_{2,\varepsilon}^2\psi=\lambda\psi-2g\check{\eta}_{2,\varepsilon}\check{\eta}_{1,\varepsilon}\varphi.
\]
Then, testing the above relation by $\psi$, integrating by parts,
using (\ref{eqpotlowerc}) and Young's inequality, we obtain
\begin{equation}\label{eqLbasic2}
\begin{array}{rcl}
  \int_{B_{R_\varepsilon}}^{}\left(\varepsilon^2|\nabla
\psi|^2+c\psi^2\right)dx & \leq &
\int_{B_{R_\varepsilon}}^{}\left(\lambda\psi^2dx+C
\check{\eta}_{1,\varepsilon}^2\varphi^2\right)dx \\
   &  &  \\
    & \stackrel{(\ref{eqLev1}),(\ref{eqLbasic})}{\leq}  & C \lambda.
\end{array}
\end{equation}
Finally, by adding (\ref{eqLbasic}) and (\ref{eqLbasic2}), recalling
the integral constraint in (\ref{eqLev1}), we deduce that $\lambda
\geq c$, as desired. The case where (\ref{eqLev2}) holds can be
treated analogously.\end{proof}

A direct consequence of Proposition \ref{proL}, and of
(\ref{eqLbilin}), is the following
\begin{cor}\label{corLleft}
If $\varepsilon$ is sufficiently small, there exists $c>0$ such
that \begin{equation}\label{eqLcor1}
\left(\mathcal{L}(\varphi,\psi),(\varphi,\psi) \right)\geq
c\int_{B_{R_\varepsilon}}^{} \left(\varepsilon^2|\nabla
\varphi|^2+\varepsilon^2|\nabla
\psi|^2+\varepsilon^\frac{2}{3}\varphi^2+\check{\eta}_{1,\varepsilon}^2\varphi^2+\psi^2\right)
dx\ \ \ \forall \varphi,\psi \in D(\mathcal{L}),
\end{equation}
where $(\cdot,\cdot)$ denotes the usual inner product in
$L^2(\mathbb{R}^2)\times L^2(\mathbb{R}^2)$.
\end{cor}

\emph{Energy estimates in $\mathbb{R}^2\setminus
B_{R_\varepsilon}$.}
 Since $ \check{\eta}_{1,\varepsilon}=0$ in
$\mathbb{R}^2\setminus B_{R_\varepsilon}$, the operator
$\mathcal{L}$ in $\mathbb{R}^2\setminus B_{R_\varepsilon}$ has the
simple ``decoupled'' form
 \[
\mathcal{L}(\varphi,\psi)=\left(\begin{array}{l}
                                         -\varepsilon^2 \Delta \varphi +\left[g(\check{\eta}_{2,\varepsilon}^2
                                        -a_{2,\varepsilon})-g_1a_{1,\varepsilon} \right]\varphi

                                          \\
                                          \\

         -\varepsilon^2 \Delta \psi+g_2\left( 3\check{\eta}_{2,\varepsilon}^2-a_{2,\varepsilon}-\frac{g}{g_2}a_{1,\varepsilon}\right)\psi                             \end{array}
 \right),
\]
for $\varphi,\psi\in H^2(\mathbb{R}^2\setminus B_{R_\varepsilon})$.
Note that
\[
g(\check{\eta}_{2,\varepsilon}^2-a_{2,\varepsilon})-g_1a_{1,\varepsilon}=g\left(\check{\eta}_{2,\varepsilon}^2
-a_{2,\varepsilon}-\frac{g}{g_2}a_{1,\varepsilon}\right)+\left(\frac{g^2}{g_2}-g_1\right)a_{1,\varepsilon}\geq
c,
\]
in $\mathbb{R}^2\setminus B_{R_\varepsilon}$, because
$a_{1,\varepsilon}\leq -c$ and $\left|\check{\eta}_{2,\varepsilon}^2
-a_{2,\varepsilon}-\frac{g}{g_2}a_{1,\varepsilon}\right|\leq
C\varepsilon^\frac{2}{3}$ therein.

In analogy  to (\ref{eqLcor1}), for small $\varepsilon>0$, one can
show rather straightforwardly that
\begin{equation}\label{eqLcor2}
\left(\mathcal{L}(\varphi,\psi),(\varphi,\psi) \right)\geq
c\int_{\mathbb{R}^2\setminus B_{R_\varepsilon}}^{}
\left(\varepsilon^2|\nabla \varphi|^2+\varepsilon^2|\nabla
\psi|^2+\varphi^2+\varepsilon^\frac{2}{3}\psi^2+\check{\eta}_{2,\varepsilon}^2\psi^2\right)
dx,
\end{equation}
for every $(\varphi,\psi) \in D(\mathcal{L})$.

\emph{Energy estimates in $\mathbb{R}^2$.} It follows at once from
(\ref{eqLcor1}) and (\ref{eqLcor2}) that, for small $\varepsilon>0$,
we have
\begin{equation}\label{eqLcor3}\begin{array}{lll}
    \left(\mathcal{L}(\varphi,\psi), (\varphi,\psi)
\right) & \geq & c\varepsilon^2\int_{\mathbb{R}^2}^{} \left(|\nabla
\varphi|^2+|\nabla \psi|^2\right)dx+ c\int_{B_{R_\varepsilon}}^{}
\left(\varepsilon^\frac{2}{3}\varphi^2+\check{\eta}_{1,\varepsilon}^2\varphi^2+\psi^2\right)
dx+ \\
  &   &   \\
      &   & +c\int_{\mathbb{R}^2\setminus B_{R_\varepsilon}}^{}
\left(\varphi^2+\varepsilon^\frac{2}{3}\psi^2+\check{\eta}_{2,\varepsilon}^2\psi^2
\right)dx,
  \end{array}
  \end{equation}
for every $(\varphi,\psi)\in D(\mathcal{L})$.

\subsubsection{Invertibility properties of
$\mathcal{L}$}\label{InvertibilityProperties}

We are now in position to obtain estimates for the solution of the
inhomogeneous problem $\mathcal{L}(\varphi,\psi)=(f_1,f_2)$ in
$\mathbb{R}^2$, with suitable right-hand side.

\begin{proposition}\label{proL=f}
Let $f_i\in L^2(\mathbb{R}^2)$, $i=1,2$. The equation
\begin{equation}\label{eqL=f}
\mathcal{L}(\varphi,\psi)=(f_1,f_2)\ \  \textrm{in}\  \mathbb{R}^2,
\end{equation}
has a unique solution $(\varphi,\psi)\in H^2(\mathbb{R}^2)\times
H^2(\mathbb{R}^2)$, provided that $\varepsilon>0$ is sufficiently
small, independently of $f_i$. Moreover, that solution satisfies
\begin{equation}\label{eqapriori1}
\vertiii{(\varphi,\psi)}^2\leq
C\int_{B_{R_\varepsilon}}^{}\left(\varepsilon^{-\frac{2}{3}}f_1^2+f_2^2
\right)dx+C\int_{\mathbb{R}^2\setminus
B_{R_\varepsilon}}^{}\left(f_1^2+\varepsilon^{-\frac{2}{3}}f_2^2
\right)dx,
\end{equation}
with $C>0$ independent of $\varepsilon$ and $f_i$, where the norm
$\vertiii{(\cdot,\cdot)}$ in $H^1(\mathbb{R}^2)\times
H^1(\mathbb{R}^2)$ is defined by
\begin{equation}\label{eqNormtriple}
\begin{array}{rcc}
  \vertiii{(\varphi,\psi)}^2 & = & \varepsilon^2\left(\|\nabla\varphi\|_{L^2(\mathbb{R}^2)}^2+\|\nabla\psi\|_{L^2(\mathbb{R}^2)}^2
\right)+\varepsilon^\frac{2}{3}\|\varphi\|_{L^2(B_{R_\varepsilon})}^2+\|\psi\|_{L^2(B_{R_\varepsilon})}^2+ \\
    &   &   \\
    &   & +\|\varphi\|_{L^2(\mathbb{R}^2\setminus B_{R_\varepsilon})}^2+\varepsilon^\frac{2}{3}\|\psi\|_{L^2(\mathbb{R}^2\setminus B_{R_\varepsilon})}^2.
\end{array}
\end{equation}

If $\varepsilon>0$ is sufficiently small, there exist $c,C>0$ such
that, for any $f\in L^2(\mathbb{R}^2)$, the solution of
\begin{equation}\label{eqL=etaf}
\mathcal{L}(\varphi,\psi)=(\check{\eta}_{1,\varepsilon}f,0)\ \
\textrm{in}\ \mathbb{R}^2,
\end{equation}
satisfies
\begin{equation}\label{eqapriori2}
\vertiii{(\varphi,\psi)}^2\leq
C\int_{B_{R_\varepsilon}}^{}f^2dx+Ce^{-c\varepsilon^{-\frac{2}{3}}}\int_{\mathbb{R}^2}^{}f^2dx.
\end{equation}

If $\varepsilon>0$ is sufficiently small, there exists $C>0$ such
that, for any $f\in L^2(\mathbb{R}^2)$, the solution of
\begin{equation}\label{eqL=etaf2}
\mathcal{L}(\varphi,\psi)=(0,\check{\eta}_{2,\varepsilon}f)\ \
\textrm{in}\ \mathbb{R}^2,
\end{equation}
satisfies
\begin{equation}\label{eqapriori3}
\vertiii{(\varphi,\psi)}\leq C\|f\|_{L^2(\mathbb{R}^2)}.
\end{equation}

\end{proposition}
\begin{proof}
As we have already discussed, the linear operator $\mathcal{L}$ is
self-adjoint in $L^2(\mathbb{R}^2)\times L^2(\mathbb{R}^2)$ with
domain $D(\mathcal{L})$. Relation (\ref{eqLcor3}) certainly implies
that the kernel of $\mathcal{L}$ is empty for small $\varepsilon>0$.
Hence, the existence and uniqueness of a solution $(\varphi,\psi)\in
H^2(\mathbb{R}^2)\times H^2(\mathbb{R}^2)$ to (\ref{eqL=f}) are
clear. We now turn our attention to establishing estimate
(\ref{eqapriori1}). Testing (\ref{eqL=f}) by $(\varphi,\psi)$, and
using part of (\ref{eqLcor3}), we find that
\[
\varepsilon^2\int_{\mathbb{R}^2}^{} \left(|\nabla \varphi|^2+|\nabla
\psi|^2\right)dx+ \int_{B_{R_\varepsilon}}^{}
\left(\varepsilon^\frac{2}{3}\varphi^2+\psi^2\right) dx
+\int_{\mathbb{R}^2\setminus B_{R_\varepsilon}}^{}
\left(\varphi^2+\varepsilon^\frac{2}{3}\psi^2 \right)dx\leq
\]
\[
\leq C\int_{B_{R_\varepsilon}}^{}\left(|f_1\varphi|+|f_2\psi|
\right)dx+C\int_{\mathbb{R}^2\setminus
B_{R_\varepsilon}}^{}\left(|f_1\varphi|+|f_2\psi| \right)dx.
\]
Using Young's inequality, we can bound the first integral in the
right-hand side by
\[
\int_{B_{R_\varepsilon}}^{}\left(\frac{\varepsilon^\frac{2}{3}}{2}\varphi^2+C\varepsilon^{-\frac{2}{3}}f_1^2+\frac{1}{2}\psi^2+Cf_2^2
\right)dx,
\]
and  analogously we can bound the second integral. By absorbing into
the left-hand side the terms that involve $\varphi$ or $\psi$, we
get (\ref{eqapriori1}).

Suppose now that (\ref{eqL=etaf}) holds. As before, but this time
making more use of (\ref{eqLcor3}), we find that
\[
\varepsilon^2\int_{\mathbb{R}^2}^{} \left(|\nabla \varphi|^2+|\nabla
\psi|^2\right)dx+ \int_{B_{R_\varepsilon}}^{}
\left(\varepsilon^\frac{2}{3}\varphi^2+\check{\eta}_{1,\varepsilon}^2\varphi^2+\psi^2\right)
dx +\int_{\mathbb{R}^2\setminus B_{R_\varepsilon}}^{}
\left(\varphi^2+\varepsilon^\frac{2}{3}\psi^2 \right)dx\leq
\]
\[
\leq
C\int_{B_{R_\varepsilon}}^{}\check{\eta}_{1,\varepsilon}|f\varphi|dx+C\int_{\mathbb{R}^2\setminus
B_{R_\varepsilon}}^{}\check{\eta}_{1,\varepsilon}|f\varphi|dx.
\]
The desired estimate (\ref{eqapriori2}) follows readily as before,
using Young's inequality to absorb a term of the form
$\frac{1}{2}\int_{B_{R_\varepsilon}}^{}
\check{\eta}_{1,\varepsilon}^2\varphi^2 dx$ into the left-hand side,
and recalling the exponential decay \eqref{eq:proGS2} of
$\check{\eta}_{1,\varepsilon}$ for $r>R_{1,\varepsilon}$ .

Finally, suppose that (\ref{eqL=etaf2}) holds. As before, making use
of (\ref{eqLcor3}) once more, we arrive at
\[
\varepsilon^2\int_{\mathbb{R}^2}^{} \left(|\nabla \varphi|^2+|\nabla
\psi|^2\right)dx+ \int_{B_{R_\varepsilon}}^{}
\left(\varepsilon^\frac{2}{3}\varphi^2+\psi^2\right) dx
+\int_{\mathbb{R}^2\setminus B_{R_\varepsilon}}^{}
\left(\varphi^2+\varepsilon^\frac{2}{3}\psi^2+\check{\eta}_{2,\varepsilon}^2\psi^2
\right)dx\leq
\]
\[
\leq
C\int_{B_{R_\varepsilon}}^{}\check{\eta}_{2,\varepsilon}|f\psi|dx+C\int_{\mathbb{R}^2\setminus
B_{R_\varepsilon}}^{}\check{\eta}_{2,\varepsilon}|f\psi|dx.
\]
The desired estimate (\ref{eqapriori3}) follows readily as before,
using Young's inequality to absorb  terms of the form
$\frac{1}{2}\int_{\mathbb{R}^2\setminus B_{R_\varepsilon}}^{}
\check{\eta}_{2,\varepsilon}^2\psi^2 dx$ and
$\frac{1}{2}\int_{B_{R_\varepsilon}}^{} \psi^2 dx$ into the
left-hand side.\end{proof}

%\begin{rem}\label{remLradial}
%If $f_i\in L^2_{rad}(\mathbb{R}^2)$, $i=1,2$, then the solutions in
%Proposition \ref{proL=f} are also radial, as otherwise there would
%be infinitely many different solutions through rotations around the
%origin.
%\end{rem}

\subsection{Existence and properties of a positive solution of the system (\ref{eqsystem})}
\label{subsecExPert}

We seek a true solution of (\ref{eqsystem}) in the form
\begin{equation}\label{eqfluctiations}
(\eta_{1,\ep},\eta_{2,\ep})=(\check{\eta}_{1,\varepsilon},\check{\eta}_{2,\varepsilon})+(\varphi,\psi),
\end{equation}
with $\varphi,\psi\in H^2_{rad}(\mathbb{R}^2)$.

In terms of $(\varphi,\psi)$, system (\ref{eqsystem}) becomes
\begin{equation}\label{eqsystemPhiPsi}
-\mathcal{L}(\varphi,\psi)=\mathcal{N}(\varphi,\psi)+\mathcal{E}(\check{\eta}_{1,\varepsilon},\check{\eta}_{2,\varepsilon}),
\end{equation}
where $\mathcal{L}$ is the linear operator in
(\ref{eqlinearization--}), the nonlinear operator $\mathcal{N}$ is
\begin{equation}\label{eqN}
\mathcal{N}(\varphi,\psi)=\left( \begin{array}{c}
                                   N_1(\varphi,\psi) \\
                                     \\
                                   N_2(\varphi,\psi)
                                 \end{array}
\right)=\left(\begin{array}{c}
                                  g_1\varphi^3+3g_1\check{\eta}_{1,\varepsilon}\varphi^2+g\check{\eta}_{1,\varepsilon}\psi^2
                                  +2g\check{\eta}_{2,\varepsilon}\psi\varphi+g\psi^2 \varphi \\
                                   \\
                                  g_2\psi^3+3g_2\check{\eta}_{2,\varepsilon}\psi^2+g\check{\eta}_{2,\varepsilon}\varphi^2
+2g\check{\eta}_{1,\varepsilon}\varphi \psi+g\varphi^2\psi
                                \end{array}
 \right),
\end{equation}
and the remainder
$\mathcal{E}(\check{\eta}_{1,\varepsilon},\check{\eta}_{2,\varepsilon})$
is as in (\ref{eqError}).

In view of Proposition \ref{proL=f}, for small $\varepsilon>0$, we
can define a nonlinear operator
$\mathcal{T}:H^2_{rad}(\mathbb{R}^2)\times
H^2_{rad}(\mathbb{R}^2)\to H^2_{rad}(\mathbb{R}^2)\times
H^2_{rad}(\mathbb{R}^2)$ via the relation
\[
\mathcal{T}(\varphi,\psi)=(\bar{\varphi},\bar{\psi}),
\]
where $(\bar{\varphi},\bar{\psi})\in H^2_{rad}(\mathbb{R}^2)\times
H^2_{rad}(\mathbb{R}^2)$ is uniquely determined from the equation
\begin{equation}\label{eqTbar}
-\mathcal{L}(\bar{\varphi},\bar{\psi})=\mathcal{N}(\varphi,\psi)+\mathcal{E}(\check{\eta}_{1,\varepsilon},\check{\eta}_{2,\varepsilon}).
\end{equation}
Note that Sobolev's inequality implies that functions in
$H^2(\mathbb{R}^2)$ are bounded, in particular
$\mathcal{N}(\varphi,\psi)\in L^2(\mathbb{R}^2)\times
L^2(\mathbb{R}^2)$ for every $(\varphi,\psi)\in
H^2(\mathbb{R}^2)\times H^2(\mathbb{R}^2)$.

For $\varepsilon>0$, $M>1$, let
\[
\mathcal{B}_{\varepsilon,M}=\left\{(\varphi,\psi)\in
H^2_{rad}(\mathbb{R}^2)\times H^2_{rad}(\mathbb{R}^2)\ :\
\vertiii{(\varphi,\psi)}\leq M\varepsilon^\frac{5}{3} \right\}.
\]

The following proposition contains the main properties of the
operator $\mathcal{T}$.

\begin{proposition}\label{proT}
If $M>1$ is sufficiently large, the operator $\mathcal{T}$ maps
$\mathcal{B}_{\varepsilon,M}$ into itself, and its restriction to
$\mathcal{B}_{\varepsilon,M}$ is a contraction with respect to the
$\vertiii{(\cdot,\cdot)}$ norm, provided that $\varepsilon>0$ is
sufficiently small.
\end{proposition}
\begin{proof}
Let $(\varphi,\psi)\in \mathcal{B}_{\varepsilon,M}$, and
$(\bar{\varphi},\bar{\psi})=\mathcal{T}(\varphi,\psi)$. In view of
(\ref{eqN}) and (\ref{eqTbar}), we have
\begin{equation}\label{eqTmappingSeries}
(\bar{\varphi},\bar{\psi})=\sum_{i=1}^{4}(\bar{\varphi}_i,\bar{\psi}_i),
\end{equation}
where $(\bar{\varphi}_i,\bar{\psi}_i)\in
H^2_{rad}(\mathbb{R}^2)\times H^2_{rad}(\mathbb{R}^2)$,
$i=1,\cdots,4$, satisfy
\[
-\mathcal{L}(\bar{\varphi}_1,\bar{\psi}_1)=\left( \begin{array}{c}
                                                    g_1\varphi^3+2g\check{\eta}_{2,\varepsilon}\psi\varphi+g\psi^2 \varphi \\
                                                      \\
                                                    g_2 \psi^3+2g\check{\eta}_{1,\varepsilon}\varphi \psi+g\varphi^2\psi
                                                  \end{array}
\right),
\]
\[
-\mathcal{L}(\bar{\varphi}_2,\bar{\psi}_2)=\left( \begin{array}{c}
                                                    \check{\eta}_{1,\varepsilon}\left(3g_1\varphi^2+g\psi^2\right) \\
                                                      \\
                                                    0
                                                  \end{array}
\right), \ \ \ \ -\mathcal{L}(\bar{\varphi}_3,\bar{\psi}_3)=\left(
\begin{array}{c}
                                                    0 \\
                                                      \\
                                                    \check{\eta}_{2,\varepsilon}\left(3g_2\psi^2+g\varphi^2\right)
                                                  \end{array}
\right),
\]
and
\[
-\mathcal{L}(\bar{\varphi}_4,\bar{\psi}_4)=\mathcal{E}(\check{\eta}_{1,\varepsilon},\check{\eta}_{2,\varepsilon}).
\]
Using Proposition \ref{proL=f}, and the Gagliardo-Nirenberg
interpolation inequality  in order to estimate
the $L^2$-norms of the nonlinear terms, it follows readily that
\begin{equation}\label{eqTmapping123}
\vertiii{(\bar{\varphi}_1,\bar{\psi}_1)}\leq
CM^3\varepsilon^\frac{7}{3}+CM^2\varepsilon^\frac{11}{6}\ \
\textrm{and}\ \ \vertiii{(\bar{\varphi}_i,\bar{\psi}_i)}\leq
CM^2\varepsilon^2,\ i=2,3,
\end{equation}
where $C>0$ is independent of  both $\varepsilon$ and $M$, provided
that $\varepsilon>0$ is sufficiently small. In order to illustrate
the procedure, let us present in detail the proof of the second
bound ($i=2$): Estimate (\ref{eqapriori2}) implies that
\[
\vertiii{(\bar{\varphi}_2,\bar{\psi}_2)}^2\leq
C\int_{B_{R_\varepsilon}}^{}(\varphi^4+\psi^4)dx+Ce^{-c\varepsilon^{-\frac{2}{3}}}\int_{\mathbb{R}^2}^{}(\varphi^4+\psi^4)dx,
\]
with constants $c,C>0$ independent of $\varepsilon,M$, provided that
$\varepsilon>0$ is sufficiently small. Since $(\varphi,\psi)\in
\mathcal{B}_{\varepsilon,M}$, it follows that
\begin{equation}\label{eqTBproperties}
\|\varphi\|_{L^2(B_{R_\varepsilon})}\leq M\varepsilon^\frac{4}{3},\
\ \|\varphi\|_{H^1(\mathbb{R}^2)}\leq M\varepsilon^\frac{2}{3};\ \
\|\psi\|_{L^2(B_{R_\varepsilon})}\leq M\varepsilon^\frac{5}{3},\ \
\|\psi\|_{H^1(\mathbb{R}^2)}\leq M\varepsilon^\frac{2}{3}.
\end{equation}
Now, the desired bound follows via the Gagliardo-Nirenberg
inequality
\begin{equation}\label{eqGagliardo}
\|u\|_{L^p(\Omega)}\leq C_p
\|u\|_{H^1(\Omega)}^{1-\frac{2}{p}}\|u\|_{L^2(\Omega)}^\frac{2}{p},\
\ p\geq 2,
\end{equation}
for $\Omega \subseteq \R^2$ regular, which implies that
\begin{equation}\label{eqL4}
\|\varphi\|_{L^4(B_{R_\varepsilon})}\leq C
\|\varphi\|_{H^1(B_{R_\varepsilon})}^{\frac{1}{2}}\|\varphi\|_{L^2(B_{R_\varepsilon})}^\frac{1}{2}\leq
CM\varepsilon,\end{equation} with constant $C>0$ independent of
$\varepsilon,M$, provided that $\varepsilon>0$ is sufficiently
small, and analogous estimates can be derived for
$\|\psi\|_{L^4(B_{R_\varepsilon})}$,
$\|\varphi\|_{L^4(\mathbb{R}^2)}$ and
$\|\psi\|_{L^4(\mathbb{R}^2)}$. The remaining bounds in
(\ref{eqTmapping123}) can be proven analogously. On the other side,
by (\ref{eqError}), (\ref{eqerrorestims}), and Proposition
\ref{proL=f}, we obtain that
\begin{equation}\label{eqTmapping4}
\vertiii{(\bar{\varphi}_4,\bar{\psi}_4)}\leq
C\varepsilon^\frac{5}{3},
\end{equation}
for small $\varepsilon>0$ (here $C$ is clearly independent of $M$ as
well). Hence, by (\ref{eqTmappingSeries}), (\ref{eqTmapping123}) and
(\ref{eqTmapping4}), we deduce that
\[
\vertiii{(\bar{\varphi},\bar{\psi})}\leq
C\varepsilon^\frac{5}{3}\left(M^3\varepsilon^\frac{2}{3}+M^2\varepsilon^\frac{1}{6}+1\right),
\]
with $C>0$ independent of $\varepsilon,M$, provided that
$\varepsilon>0$ is sufficiently small. Consequently, if we choose
$M=2C$, and fix it from now on, decreasing $\varepsilon>0$ further
if necessary so that
$M^3\varepsilon^\frac{2}{3}+M^2\varepsilon^\frac{1}{6}\leq 1$, it
follows that
\[
\vertiii{(\bar{\varphi},\bar{\psi})}=\vertiii{\mathcal{T}(\varphi,\psi)}\leq
M \varepsilon^\frac{5}{3}.
\]
We conclude that, if $\varepsilon>0$ is sufficiently small, the
operator $\mathcal{T}$ maps $\mathcal{B}_{\varepsilon,M}$ into
itself, as asserted.

It remains to show that the restriction of $\mathcal{T}$ to
$\mathcal{B}_{\varepsilon,M}$ is a contraction with respect to the
$\vertiii{(\cdot,\cdot)}$ norm, provided that $\varepsilon>0$ is
sufficiently small. To this end, let
\begin{equation}\label{eqTcombo0}
(\varphi_i,\psi_i)\in \mathcal{B}_{\varepsilon,M},\ i=1,2,\ \
\textrm{and}\ \
(\bar{\varphi}_i,\bar{\psi}_i)=\mathcal{T}(\varphi_i,\psi_i),\
i=1,2.
\end{equation}
Then, set
\begin{equation}\label{eqTcombo1}
(\bar{w},\bar{z})=(\bar{\varphi}_1-\bar{\varphi}_2,\bar{\psi}_1-\bar{\psi}_2).
\end{equation}
As before, it is convenient to write
\begin{equation}\label{eqTcombo2}
(\bar{w},\bar{z})=\sum_{i=1}^{5}(\bar{w}_i,\bar{z}_i),
\end{equation}
where $(\bar{w}_i,\bar{z}_i)\in H^2_{rad}(\mathbb{R}^2)\times
H^2_{rad}(\mathbb{R}^2)$, $i=1,\cdots,5$, satisfy
\[
-\mathcal{L}(\bar{w}_1,\bar{z}_1)=\left( \begin{array}{c}
                                                    g_1(\varphi_1^2+\varphi_1 \varphi_2+\varphi_2^2)(\varphi_1-\varphi_2) \\
                                                      \\
                                                    g_2 (\psi_1^2+\psi_1 \psi_2+\psi_2^2)(\psi_1-\psi_2)
                                                  \end{array}
\right),
\]
\[
-\mathcal{L}(\bar{w}_2,\bar{z}_2)=2g\left[\psi_2(\varphi_1-\varphi_2)+\varphi_1(\psi_1-\psi_2)\right]
\left(
\begin{array}{c}
                                                    \check{\eta}_{2,\varepsilon} \\
                                                      \\
                                                    \check{\eta}_{1,\varepsilon}
                                                  \end{array}
\right),
\]
\[
-\mathcal{L}(\bar{w}_3,\bar{z}_3)=g \left(
\begin{array}{c}
                                                    \psi_2^2(\varphi_1-\varphi_2)+\varphi_1(\psi_1+\psi_2)(\psi_1-\psi_2) \\
                                                      \\
                                                    \varphi_2^2(\psi_1-\psi_2)+\psi_1(\varphi_1+\varphi_2)(\varphi_1-\varphi_2)
                                                  \end{array}
\right),
\]

\[
-\mathcal{L}(\bar{w}_4,\bar{z}_4)=\left( \begin{array}{c}
                                                    \check{\eta}_{1,\varepsilon}\left[3g_1(\varphi_1+\varphi_2)(\varphi_1-\varphi_2)+
                                                    g(\psi_1+\psi_2)(\psi_1-\psi_2)\right] \\
                                                      \\
                                                    0
                                                  \end{array}
\right),
\]
and
\[
\ \ \ \ -\mathcal{L}(\bar{w}_5,\bar{z}_5)=\left(
\begin{array}{c}
                                                    0 \\
                                                      \\
                                                    \check{\eta}_{2,\varepsilon}\left[3g_2(\psi_1+\psi_2)(\psi_1-\psi_2)+
                                                    g(\varphi_1+\varphi_2)(\varphi_1-\varphi_2)\right]
                                                  \end{array}
\right).
\]
As before, using Proposition \ref{proL=f}, the Gagliardo-Nirenberg
inequality (\ref{eqGagliardo}), and the inequalities
\begin{equation}\label{eqTidentities}
  \begin{array}{ccc}
    \|\varphi\|_{L^2(B_{R_\varepsilon})}\leq
\varepsilon^{-\frac{1}{3}}\vertiii{(\varphi,\psi)}, &
\|\varphi\|_{L^2(\mathbb{R}^2\setminus B_{R_\varepsilon})}\leq
\vertiii{(\varphi,\psi)}, & \|\varphi\|_{H^1(\mathbb{R}^2)}\leq
\varepsilon^{-1}\vertiii{(\varphi,\psi)},   \\
      &   &   \\
    \|\psi\|_{L^2(B_{R_\varepsilon})}\leq
\vertiii{(\varphi,\psi)}, & \|\psi\|_{L^2(\mathbb{R}^2\setminus
B_{R_\varepsilon})}\leq \varepsilon^{-\frac{1}{3}}
\vertiii{(\varphi,\psi)}, & \|\psi\|_{H^1(\mathbb{R}^2)}\leq
\varepsilon^{-1}\vertiii{(\varphi,\psi)},
  \end{array}
\end{equation}
for every $\varphi,\psi\in H^1(\mathbb{R}^2)$, we can show that
\begin{equation}\label{eqTcombo3}
\begin{array}{cc}
  \vertiii{(\bar{w}_1,\bar{z}_1)}\leq
C\varepsilon^\frac{2}{3}\vertiii{(\varphi_1-\varphi_2,\psi_1-\psi_2)},
& \vertiii{(\bar{w}_2,\bar{z}_2)}\leq
C\varepsilon^\frac{1}{6}\vertiii{(\varphi_1-\varphi_2,\psi_1-\psi_2)}, \\
    &   \\
  \vertiii{(\bar{w}_3,\bar{z}_3)}\leq
C\varepsilon^\frac{5}{6}\vertiii{(\varphi_1-\varphi_2,\psi_1-\psi_2)},
& \vertiii{(\bar{w}_i,\bar{z}_i)}\leq
C\varepsilon^\frac{1}{3}\vertiii{(\varphi_1-\varphi_2,\psi_1-\psi_2)},\
i=4,5,
\end{array}
\end{equation}
provided that $\varepsilon>0$ is sufficiently small. In order to
illustrate the procedure, let us present in detail the proof of the
bound for $\vertiii{(\bar{w}_2,\bar{z}_2)}$: From Proposition
\ref{proL=f}, we obtain that
\[\begin{array}{rcl}
    \vertiii{(\bar{w}_2,\bar{z}_2)} & \leq & C\varepsilon^{-\frac{1}{3}}\|\psi_2(\varphi_1-\varphi_2)\|_{L^2(B_{R_\varepsilon})}
+C\varepsilon^{-\frac{1}{3}}\|\varphi_1(\psi_1-\psi_2)\|_{L^2(B_{R_\varepsilon})}+
\\
     &  &  \\
      &   & +C\varepsilon^{-\frac{1}{3}}\|\psi_2(\varphi_1-\varphi_2)\|_{L^2(\mathbb{R}^2\setminus
B_{R_\varepsilon})}
+C\varepsilon^{-\frac{1}{3}}\|\varphi_1(\psi_1-\psi_2)\|_{L^2(\mathbb{R}^2\setminus
B_{R_\varepsilon})}.
  \end{array}
\]
The second term in the right-hand side of the above relation can be
estimated as follows: By the Cauchy-Schwarz inequality, we have
\[\begin{array}{rcl}
    \varepsilon^{-\frac{1}{3}}\|\varphi_1(\psi_1-\psi_2)\|_{L^2(B_{R_\varepsilon})} &
     \leq & \varepsilon^{-\frac{1}{3}}\|\varphi_1\|_{L^4(B_{R_\varepsilon})}
\|\psi_1-\psi_2\|_{L^4(B_{R_\varepsilon})} \\
      &   &   \\
      & \stackrel{(\ref{eqL4})}{\leq} & C\varepsilon^\frac{2}{3}\|\psi_1-\psi_2\|_{L^4(B_{R_\varepsilon})} \\
      &  &  \\
      & \stackrel{(\ref{eqGagliardo}),(\ref{eqTidentities})}{\leq} & C\varepsilon^\frac{1}{6}\vertiii{(\varphi_1-\varphi_2,\psi_1-\psi_2)}.
  \end{array}
\]
The remaining terms can be estimated in a similar fashion, giving
the desired bound for $\vertiii{(\bar{w}_2,\bar{z}_2)}$. The other
bounds in (\ref{eqTcombo3}) can be verified analogously.
Consequently, combining relations (\ref{eqTcombo0}),
(\ref{eqTcombo1}), (\ref{eqTcombo2}), and (\ref{eqTcombo3}), we
infer that
\[
\vertiii{\mathcal{T}(\varphi_1,\psi_1)-\mathcal{T}(\varphi_2,\psi_2)}\leq
C\varepsilon^\frac{1}{6}\vertiii{(\varphi_1,\psi_1)-(\varphi_2,\psi_2)}\
\ \forall\ (\varphi_i,\psi_i)\in \mathcal{B}_{\varepsilon,M},\
i=1,2.
\]
We therefore conclude that, for sufficiently small $\varepsilon>0$,
the restriction of $\mathcal{T}$ to $\mathcal{B}_{\varepsilon,M}$ is
a contraction with respect to the $\vertiii{(\cdot,\cdot)}$ norm, as
asserted.\end{proof}

The  above  proposition implies the main result of this section:
\begin{proposition}\label{proExistSystem}
There exists a constant $M>0$, such that the system (\ref{eqsystem})
has a unique solution $(\eta_{1,\varepsilon},\eta_{2,\varepsilon})$
such that
\begin{equation}\label{eqestimMain}
\vertiii{(\eta_{1,\varepsilon}-\check{\eta}_{1,\varepsilon},\eta_{2,\varepsilon}-\check{\eta}_{2,\varepsilon})}\leq
M\varepsilon^\frac{5}{3},
\end{equation}
if $\varepsilon>0$ is sufficiently small, where the above norm is as
in (\ref{eqNormtriple}).
\end{proposition}
\begin{proof}
In view of Proposition \ref{proT}, for small $\varepsilon>0$, we can
define iteratively a sequence $(\varphi_n,\psi_n)\in
\mathcal{B}_{\varepsilon,M}$ such that
\begin{equation}\label{eqTiter}
(\varphi_{n+1},\psi_{n+1})=\mathcal{T}(\varphi_n,\psi_n),\ n\geq 0,\
(\varphi_0,\psi_0)=(0,0).
\end{equation}
Moreover, the same proposition implies that $(\varphi_n,\psi_n)$ is
a Cauchy sequence in $H^1_{rad}(\mathbb{R}^2)\times
H^1_{rad}(\mathbb{R}^2)$. Hence, we infer that
\[
(\varphi_n,\psi_n)\to (\varphi_\infty,\psi_\infty)\ \ \textrm{in}\
H^1_{rad}(\mathbb{R}^2)\times H^1_{rad}(\mathbb{R}^2),\ \textrm{as}\
n\to \infty,
\]
for some $(\varphi_\infty,\psi_\infty)\in
H^1_{rad}(\mathbb{R}^2)\times H^1_{rad}(\mathbb{R}^2)$ such that
\[
\vertiii{(\varphi_\infty,\psi_\infty)}\leq M
\varepsilon^\frac{5}{3}.
\]
In turn, letting $n\to \infty$ in the weak form of (\ref{eqTiter})
(recall (\ref{eqTbar})) yields that $(\varphi_\infty,\psi_\infty)$
is a weak solution of (\ref{eqsystemPhiPsi}). Then, by standard
elliptic regularity theory, we
deduce that $(\varphi_\infty,\psi_\infty)\in
H^2_{rad}(\mathbb{R}^2)\times H^2_{rad}(\mathbb{R}^2)$ (i.e.
$(\varphi_\infty,\psi_\infty)\in \mathcal{B}_{\varepsilon,M}$) and
is smooth (i.e. a classical solution). The point being  that
$\mathcal{B}_{\varepsilon,M}$ is not closed in the
$\vertiii{(\cdot,\cdot)}$ norm. In fact, equation
(\ref{eqsystemPhiPsi}) has a unique solution in
$\mathcal{B}_{\varepsilon,M}$, as the restriction of $\mathcal{T}$
to $\mathcal{B}_{\varepsilon,M}$ is a contraction with respect to
the $\vertiii{(\cdot,\cdot)}$ norm, provided that $\varepsilon>0$ is
sufficiently small. Consequently, recalling the equivalence of
(\ref{eqsystem}) to (\ref{eqsystemPhiPsi}) via
(\ref{eqfluctiations}), we conclude that the assertions of the
proposition hold.\end{proof}

A direct consequence of (\ref{eqestimMain}) is the following
\begin{cor}\label{cor1}
If $\varepsilon>0$ is sufficiently small, the solutions
$\eta_{1,\varepsilon}$ and $\eta_{2,\varepsilon}$ of the system
(\ref{eqsystem}) that are provided by Proposition
\ref{proExistSystem} satisfy
\begin{equation}\label{eqUniformOut}
\|\eta_{i,\varepsilon}-\check{\eta}_{i,\varepsilon}\|_{L^\infty(|x|\geq\delta)}\leq
C\varepsilon,\ \ \textrm{and}\ \ \eta_{i,\varepsilon}(x)\to 0\
\textrm{as}\ |x|\to \infty,
%\ \ \textrm{and}\ \
%\|\eta_{i,\varepsilon}-\check{\eta}_{i,\varepsilon}\|_{L^\infty\left(\left\{|x|\geq\delta\right\}\cap
%\left\{\left||x|-R_{i,\varepsilon}\right|\geq \delta
%\right\}\right)}\leq C\varepsilon^\frac{7}{6},
 \ \ i=1,2.
\end{equation}
\end{cor}
\begin{proof}
Consider the fluctuations
\[
\varphi=\eta_{1,\varepsilon}-\check{\eta}_{1,\varepsilon}\ \
\textrm{and} \ \
\psi=\eta_{2,\varepsilon}-\check{\eta}_{2,\varepsilon}.
\]
It follows from (\ref{eqNormtriple}) and (\ref{eqestimMain}) that
\begin{equation}\label{eqStrauss--}
\|\nabla \varphi\|_{L^2(\mathbb{R}^2)}\leq
C\varepsilon^\frac{2}{3},\ \ \| \varphi\|_{L^2(\mathbb{R}^2)}\leq
C\varepsilon^\frac{4}{3}\ \ \textrm{and}\ \ \|\nabla
\psi\|_{L^2(\mathbb{R}^2)}\leq C\varepsilon^\frac{2}{3},\ \ \|
\psi\|_{L^2(\mathbb{R}^2)}\leq C\varepsilon^\frac{4}{3},
\end{equation}
(note that we did not make full use of (\ref{eqestimMain})). In
order to transform the above into uniform estimates, we  need
the following inequality which can be traced back to
\cite{straussCMP}: There exists a constant $C>0$ such that
\begin{equation}\label{eqStrauss}
|x|^\frac{1}{2}\left|v(x) \right|\leq C\|\nabla
v\|_{L^2(\mathbb{R}^2)}^\frac{1}{2}\|v\|_{L^2(\mathbb{R}^2)}^\frac{1}{2}\
\ \textrm{for}\ \textrm{a.e}\ x\in \mathbb{R}^2,
\end{equation}
and all $v\in H^1_{rad}(\mathbb{R}^2)\equiv\{v\in H^1(\mathbb{R}^2)\
:\ v\ \textrm{is\ radial} \}$. The desired asymptotic behavior in
(\ref{eqUniformOut}) follows at once. Making use of this inequality
for $|x|\geq \delta$, we obtain that
\[
\|\varphi \|_{L^\infty(|x|\geq \delta)}\leq C \varepsilon\ \
\textrm{and} \ \ \|\psi \|_{L^\infty(|x|\geq \delta)}\leq C
\varepsilon,
\]
which are exactly the desired uniform estimates in
(\ref{eqUniformOut}).
%The
%remaining estimates follow readily by using Remark
%\ref{remweightedL2} to strengthen the second and fourth bounds in
%(\ref{eqStrauss--}) (keep in mind (\ref{eqLpot1})).
\end{proof}

We  now turn our attention to establishing uniform estimates on
$\bar{B}_\delta$. The following lemma will come in handy in the
proof of Corollary \ref{corunifInter} below.
\begin{lemma}\label{lemunifBlowUp}
There exists a constant $C>0$ such that the solutions that are
provided by Proposition \ref{proExistSystem} satisfy
\[
\|\eta_{i,\varepsilon}\|_{L^\infty(B_\delta)}\leq C,\ \ i=1,2,
\]
if $\varepsilon>0$ is sufficiently small.
\end{lemma}
\begin{proof}
Suppose that the assertion is false. We use a blow-up argument
to arrive at a contradiction (see also \cite{GidasSpruck}). Without
loss of generality, we may assume that there exist $\varepsilon_n\to
0$ and $x_n\in B_\delta$ such that
\[
\eta_{1,\varepsilon_n}(x_n)=\|\eta_{1,\varepsilon_n}\|_{L^\infty(B_\delta)}=M_n\to
\infty.
\]
We may further assume that $x_n\to x_\infty \in \bar{B}_\delta$.
Now, we re-scale $\eta_{1,\varepsilon_n}$ by setting
\[
v_n(y)=\mu_n \eta_{1,\varepsilon_n}(x_n+\varepsilon_n \mu_n y)\ \
\textrm{with}\ \mu_n=M_n^{-1}\to 0.
\]
The function $v_n$ satisfies
\[
-\Delta v_n+g_1v_n^3-g_1\mu_n^2
a_{1,\varepsilon_n}(x_n+\varepsilon_n\mu_n y)v_n+g\mu_n^2\left(
\eta_{2,\varepsilon_n}^2(x_n+\varepsilon_n \mu_n
y)-a_{2,\varepsilon_n}(x_n+\varepsilon_n \mu_n y) \right) v_n=0.
\]
By using elliptic $L^p$ estimates and standard imbeddings, exploiting the bound
\[
\|\eta_{2,\varepsilon}^2-a_{2,\varepsilon}\|_{L^p(B_\delta)}\leq C_p
\varepsilon^{\frac{2}{3}+\frac{1}{p}},\ p\geq 2,\ \
(\textrm{readily\ derivable\ from \ Proposition\
\ref{proExistSystem}}),
\]
we deduce that a subsequence of $v_n$ converges uniformly on compact
sets to a bounded nontrivial solution $v_\infty$ of the problem
\begin{equation}\label{eqBrezis}
\Delta v=g_1 v^3, \qquad v(0)=1,
\end{equation}
in the entire space $\mathbb{R}^2$ or in an open half-space $H$
containing the origin, with zero boundary conditions on $\partial
H$.
%depending on whether
%$x_\infty\in B_\delta$ or $x_\infty\in \partial B_\delta$
%respectively.
Actually, in the latter scenario, one has to perform a rotation and
stretching of coordinates in the resulting limiting equation to get
(\ref{eqBrezis}), see \cite{GidasSpruck}. In any case, by reflecting
$v_\infty$ oddly across $\partial H$ if necessary, we have been led
to a \emph{nontrivial} solution of (\ref{eqBrezis}) in the whole
space $\mathbb{R}^2$. On the other hand, this contradicts  a well
known Liouville type theorem of Brezis
\cite{brezisLiouville}.\end{proof}

The following corollary provides additional information to Corollary \ref{cor1}, but
will be considerably improved in Proposition \ref{prounifAway}.
\begin{cor}\label{corunifInter}
If $\varepsilon>0$ is sufficiently small, the solutions provided by
Proposition \ref{proExistSystem} satisfy
\begin{equation}\label{equnifInter}
\|\eta_{i,\varepsilon}-\sqrt{a_{i,\varepsilon}}\|_{L^\infty(B_\delta)}\leq
C\varepsilon^\frac{1}{3},\ \ i=1,2.
\end{equation}
\end{cor}
\begin{proof}
Let
\begin{equation}\label{eqphiDef}
\phi=\eta_{1,\varepsilon}-\sqrt{a_{1,\varepsilon}},\ \ x\in
\bar{B}_{2\delta}.
\end{equation}
Observe that estimates \eqref{eqGSuUnif} and (\ref{eqestimMain})
imply that
\begin{equation}\label{eqphiL2}
\|\phi\|_{L^2(B_{2\delta})}\leq C \varepsilon^\frac{4}{3}.
\end{equation}
From the first equation in (\ref{eqsystem}), by rearranging terms,
in $B_{2\delta}$ we obtain that
\[
    -\varepsilon^2\Delta
\phi  =
-g_1{\eta}_{1,\varepsilon}\left({\eta}_{1,\varepsilon}+\sqrt{a_{1,\varepsilon}}
\right)\phi  +\varepsilon^2\Delta
      \sqrt{a_{1,\varepsilon}}-g\eta_{1,\varepsilon}(\eta_{2,\varepsilon}^2-a_{2,\varepsilon})=:f.
\]
 By interior elliptic regularity theory, we deduce that
\begin{equation}\label{eqagmon2}
\|\phi\|_{H^2(B_\delta)}\leq
C\left(\varepsilon^{-2}\|f\|_{L^2(B_{2\delta})}+\|\phi\|_{L^2(B_{2\delta})}
\right)\leq C\varepsilon^{-\frac{2}{3}},
\end{equation}
where we also used Lemma \ref{lemunifBlowUp} and (\ref{eqphiL2}).
 Now, by the two-dimensional Agmon inequality \cite[Lem.
13.2]{Agmon}, we infer that
\[
\|\phi\|_{L^\infty(B_\delta)}\leq
C\|\phi\|_{H^2(B_\delta)}^\frac{1}{2}\|\phi\|_{L^2(B_\delta)}^\frac{1}{2}\stackrel{(\ref{eqphiL2}),(\ref{eqagmon2})}{\leq}
C\varepsilon^\frac{1}{3}.
\]
The desired bound for
$\eta_{1,\varepsilon}-\sqrt{a_{1,\varepsilon}}$ follows at once from
(\ref{eqphiDef}) and the above relation. The corresponding bound for
$\eta_{2,\varepsilon}-\sqrt{a_{2,\varepsilon}}$ can be shown
analogously.\end{proof}

We are now in position to show that the solutions in Proposition
\ref{proExistSystem} are in fact positive.

\begin{proposition}\label{proPositive}
If $\varepsilon>0$ is sufficiently small, the solutions in
Proposition \ref{proExistSystem} satisfy
\[
\eta_{i,\varepsilon}>0\ \ \textrm{in}\ \mathbb{R}^2,\ \ i=1,2.
\]
\end{proposition}
\begin{proof}
By virtue  of (\ref{eqGSuUnif}), (\ref{eq:proGS4}),
(\ref{eqUniformOut}), and Corollary \ref{corunifInter}, given $D\geq
1$, we deduce that
\begin{equation}\label{eqpositive1}
\eta_{1,\varepsilon}\geq c_D\varepsilon^\frac{1}{3}>0\ \
\textrm{if}\ |x|\leq R_{1,\varepsilon}+D\varepsilon^\frac{2}{3},
\end{equation}
provided that $\varepsilon>0$ is sufficiently small; where
throughout this proof, unless specified otherwise, the generic
constants $c,C>0$ are also independent of $D\geq 1$. From (\ref{eqsystem1}), we observe that
$\eta_{1,\varepsilon}$ satisfies a linear equation of the form
\begin{equation}\label{eqpositive2}
-\varepsilon^2\Delta
\eta_{1,\varepsilon}+Q(x)\eta_{1,\varepsilon}=0,
\end{equation}
where
\[
Q(x)=g_1
(\eta_{1,\varepsilon}^2-a_{1,\varepsilon})+g(\eta_{2,\varepsilon}^2-a_{2,\varepsilon}).
\]
If $R_{1,\varepsilon}+D\varepsilon^\frac{2}{3}\leq |x|\leq
R_{2,\varepsilon}$, recalling \eqref{eq:proGS2} and
(\ref{eqUniformOut}), we find that
\[
\begin{split}
Q(x)&=g_1
(\eta_{1,\varepsilon}^2-\check{\eta}_{1,\ep}^2+\check{\eta}_{1,\ep}^2 -a_{1,\varepsilon})+g(\eta_{2,\varepsilon}^2-\check{\eta}_{2,\ep}^2+\check{\eta}_{2,\ep}^2-a_{2,\varepsilon})\\
&\geq -g_1 a_{1,\ep}+g(\check{\eta}_{2,\ep}^2-a_{2,\ep})-C\ep \\
&\geq \left(\frac{g^2}{g_2}-g_1\right)
a_{1,\varepsilon}+g\left(\check{\eta}_{2,\varepsilon}^2-a_{2,\varepsilon}-\frac{g}{g_2}a_{1,\varepsilon}\right)-C\varepsilon.
\end{split}
\]
In particular, if $R_{1,\varepsilon}+D\varepsilon^\frac{2}{3}\leq
|x|\leq R_{\varepsilon}$, where
$\check{\eta}_{2,\varepsilon}=\tilde{\eta}_{2,\varepsilon}$, via
(\ref{eqeta2cut}), which implies that the second term in the
right-hand side is $-g^2g_2^{-1}\check{\eta}_{1,\varepsilon}^2$, and
the exponential decay of the ground state
$\hat{\eta}_{1,\varepsilon}$ for $r=|x|>R_{1,\varepsilon}$, we
obtain that
\[\begin{array}{rcl}
    Q(r) & \geq & \left(\frac{g^2}{g_2}-g_1\right) a_{1,\varepsilon}-Ce^{-cD}
\varepsilon^\frac{2}{3} \\
      &   &   \\
      &
\stackrel{(\ref{eq:condition_on_g_thomas_fermi}), (\ref{eqa12}),
(\ref{eqnondegeneracy}),(\ref{eq:convergence_lagrange_multipliers})}{\geq}
& c(r-R_{1,\varepsilon})^2+ cD\varepsilon^\frac{2}{3}-Ce^{-cD}
\varepsilon^\frac{2}{3}, \\
& &\\
& \geq &c(r-R_{1,\varepsilon})^2+c D\varepsilon^\frac{2}{3},
  \end{array}
\]
increasing the value of $D$ if needed,
 provided that $\varepsilon>0$ is sufficiently small.
Clearly, in view of (\ref{eqeta2check}), the same lower bound holds
if $R_\varepsilon \leq |x|\leq R_{\varepsilon}+\delta$. On the other
side, if $R_\varepsilon+\delta\leq |x|\leq R_{2,\varepsilon}$, where
$a_{1,\varepsilon}= -c(r^2-R_{1,\ep}^2)$,
$\check{\eta}_{2,\varepsilon}=\hat{\eta}_{2,\varepsilon}$ and
$\left|\hat{\eta}_{2,\varepsilon}^2-a_{2,\varepsilon}-\frac{g}{g_2}a_{1,\varepsilon}\right|\leq
C\varepsilon^\frac{2}{3}$, we deduce that $Q(r)\geq
c(r-R_{1,\varepsilon})^2$. The latter lower bound also holds if
$|x|\geq R_{2,\varepsilon}$. So far, we have shown that
\begin{equation}\label{eqpositive3}
Q(x)\geq c(r-R_{1,\varepsilon})^2+ cD \varepsilon^\frac{2}{3},\ \
|x|\geq R_{1,\varepsilon}+D\varepsilon^\frac{2}{3},
\end{equation}
provided that $\varepsilon>0$ is sufficiently small. By
(\ref{eqUniformOut}), (\ref{eqpositive1}), (\ref{eqpositive2}),
(\ref{eqpositive3}), and the maximum principle, we deduce that
\[
\eta_{1,\varepsilon}\geq 0\ \ \textrm{if}\ |x|\geq
R_{1,\varepsilon}+D\varepsilon^\frac{2}{3}.
\]
 The desired strict positivity of $\eta_{1,\varepsilon}$
follows immediately from  (\ref{eqpositive1}), (\ref{eqpositive2}),
the above relation, and the strong maximum principle. The
corresponding property for $\eta_{2,\varepsilon}$ can be proven
analogously.\end{proof}

The following lemma is motivated from Lemma
\ref{lemma:decay_for_large_x}, and will be used in the next section.
\begin{lemma}\label{lemsuperExp}
Given $D>1$ sufficiently large, we have
\[
\eta_{i,\varepsilon}(s)\leq
\eta_{i,\varepsilon}(r)\exp\left\{-\frac{D^\frac{1}{3}}{\varepsilon^\frac{2}{3}}(s^2-r^2)
\right\}\ \ \textrm{for}\ \ s\geq r\geq
R_{i,\varepsilon}+D\varepsilon^\frac{2}{3},\ \ i=1,2,
\]
provided that $\varepsilon>0$ is sufficiently small.
\end{lemma}
\begin{proof}
Throughout this proof, the generic constants $c,C>0$ are
independent of both small $\varepsilon>0$ and large $D>1$. Abusing
notation slightly, let
\[
u(x)=u(s)=\exp\left\{-\frac{D^\frac{1}{3}}{\varepsilon^\frac{2}{3}}s^2\right\},\
\ x\in\mathbb{R}^2,\ \ s=|x|.
\]
It is easy to see that
\begin{equation}\label{eqsuperExpLinearOp}
-\varepsilon^2 \Delta
u+c\left[(s-R_{1,\varepsilon})^2+D\varepsilon^\frac{2}{3}
\right]u\geq \left[-4D^\frac{2}{3}\varepsilon^\frac{2}{3}
s^2+c(s-R_{1,\varepsilon})^2+cD\varepsilon^\frac{2}{3} \right] u
\geq 0,
\end{equation}
if $s=|x|\geq R_{1,\varepsilon}+D\varepsilon^\frac{2}{3}$, for
sufficiently large $D>1$, provided that $\varepsilon>0$ is
sufficiently small. Here $c>0$ is as in (\ref{eqpositive3}).  For
such $D,\varepsilon$, and any $r\geq
R_{1,\varepsilon}+D\varepsilon^\frac{2}{3}$, it follows that the
function
\[
v(s)=\eta_{1,\varepsilon}(r)\exp\left\{-\frac{D^\frac{1}{3}}{\varepsilon^\frac{2}{3}}(s^2-r^2)
\right\},\ \ s=|x|\geq r,
\]
is an upper solution of the linear elliptic equation that is defined
by the left-hand side of (\ref{eqsuperExpLinearOp}). On the other
side, by virtue of (\ref{eqpositive2}) and (\ref{eqpositive3}), the
function $\eta_{1,\varepsilon}(s)$ is a lower solution of the same
equation for $|x|>r$, which  clearly coincides with   the upper
solution $v$ on $\partial B_{r}$. Hence, by the maximum principle,
and (\ref{eqUniformOut}), we deduce that
\[
\eta_{1,\varepsilon}(s)\leq v(s),\ \ \forall\ s=|x|\geq r\geq
R_{1,\varepsilon}+D\varepsilon^\frac{2}{3},
\]
for any large $D>1$, provided that $\varepsilon>0$ is sufficiently
small. The validity of the asserted estimate for
$\eta_{1,\varepsilon}$ now follows immediately, while that for
$\eta_{2,\varepsilon}$ follows analogously.
\end{proof}

\subsection{Improved uniform estimates away from $R_{1,\varepsilon}$ and $R_{2,\varepsilon}$}

In this subsection we  show that the uniform estimates in
(\ref{eqUniformOut}), for the difference
$\eta_{i,\varepsilon}-\check{\eta}_{i,\varepsilon}$, can be improved
outside of an $\mathcal{O}(\varepsilon^\frac{2}{3})$-neighborhood of
$R_{i,\varepsilon}$, $i=1,2$.

The results of this subsection, as well as those of the following
one, are not essential for the proof of Theorem
\ref{thm:nonexistence_vortices}, and, depending on the reader's
preference, can be skipped on a first reading.

\begin{proposition}\label{prounifAway}
If $\varepsilon>0$ is sufficiently small, there exist $C,\delta>0$,
with $\delta<\min\left\{\frac{R_{1,0}}{4},\frac{R_{2,0}-R_{1,0}}{4}
\right\}$, such that
\[
\left|\eta_{i,\varepsilon}(r)
-\check{\eta}_{i,\varepsilon}(r)\right|\leq C\varepsilon^2\ \
\textrm{if} \ \ |r-R_{1,\varepsilon}|\geq \delta \ \textrm{and} \
|r-R_{2,\varepsilon}|\geq \delta,\ i=1,2.
\]
\end{proposition}
\begin{proof}
We  prove the assertion in the case where $r\in
[0,R_{1,\varepsilon}-\delta]$, which reduces to show that
\begin{equation}\label{eqbootstrapDesired}
\left|\eta_{i,\varepsilon}(r)
-\sqrt{a_{i,\varepsilon}(r)}\right|\leq C\varepsilon^2\ \
\textrm{if} \ \ r\in [0,R_{1,\varepsilon}- \delta],\ i=1,2,
\end{equation}
(recall the construction of $\check{\eta}_{i,\varepsilon}$ and also
see (\ref{eqGSuUnif})). In the remaining intervals the proof carries
over analogously.

%The proof of the above relation is quite involved and it consists of
%two main steps: First we show that it holds for a point $r_*\in
%\left(R_{1,\varepsilon}-\delta,R_{1,\varepsilon}-\frac{\delta}{10}
%\right)$, and then we show that if holds for all $r\in [0,r_*]$.

We  first require a rough uniform bound for the radial
derivatives of the functions
\begin{equation}\label{equvDef}
u\equiv \eta_{1,\varepsilon}-\sqrt{a_{1,\varepsilon}}\ \
\textrm{and}\ \ v\equiv
\eta_{2,\varepsilon}-\sqrt{a_{2,\varepsilon}},
\end{equation}
say over the interval $\left[\delta,
R_{1,\varepsilon}-\frac{\delta}{50} \right]$. It follows from
(\ref{eqsystemPhiPsi}), (\ref{eqTBproperties}), and
(\ref{eqUniformOut}) (with a smaller constant $\delta>0$), that
\[
\|\Delta \varphi\|_{L^2\left(\frac{\delta}{2},
R_{1,\varepsilon}-\frac{\delta}{100} \right)} = \|\Delta
(\eta_{1,\varepsilon}-\hat{\eta}_{1,\varepsilon})\|_{L^2\left(\frac{\delta}{2},
R_{1,\varepsilon}-\frac{\delta}{100} \right)}\leq
C\varepsilon^{-\frac{2}{3}}.
\]
In turn, by interior elliptic regularity theory
 and (\ref{eqestimMain}), we deduce that
\[
\|\varphi\|_{H^2\left(\delta, R_{1,\varepsilon}-\frac{\delta}{50}
\right)}\leq C\varepsilon^{-\frac{2}{3}}.
\]
Hence, from (\ref{eqStrauss--}), via (\ref{eqStrauss}) with
$\varphi'$ in place of $v$, we get that
\[
|\varphi'(r)|\leq C,\ \ r\in \left[\delta,
R_{1,\varepsilon}-\frac{\delta}{50} \right].
\]
Thanks to (\ref{eq:convergence_lagrange_multipliers}) and
(\ref{eqGSuUnif}), we infer that
\begin{equation}\label{equ'Unif}
\|u'\|_{L^\infty\left(\delta, R_{1,\varepsilon}-\frac{\delta}{50}
\right)}\leq C.
\end{equation}
Similarly we have
\begin{equation}\label{eqv'Unif}
\|v'\|_{L^\infty\left(\delta, R_{1,\varepsilon}-\frac{\delta}{50}
\right)}\leq C.
\end{equation}

Observe that $u,v$ satisfy
\begin{equation}\label{eqStar}
\mathcal{M}(u,v)\equiv\left(
\begin{array}{c}
  -\varepsilon^2 \Delta u+g_1\eta_1(\eta_1+\sqrt{a_{1,\varepsilon}})u+g\eta_1(\eta_2+\sqrt{a_{2,\varepsilon}})v
     \\
    \\
  -\varepsilon^2 \Delta
  v+g_2\eta_2(\eta_2+\sqrt{a_{2,\varepsilon}})v+g\eta_2(\eta_1+\sqrt{a_{1,\varepsilon}})u
\end{array}
\right)= \left( \begin{array}{c}
                  \varepsilon^2 \Delta\sqrt{a_{1,\varepsilon}} \\
                    \\
                 \varepsilon^2 \Delta\sqrt{a_{2,\varepsilon}}
                \end{array}
\right),
\end{equation}
$x\in B_{\left(R_{1,\varepsilon}-\frac{\delta}{50} \right)}$, having
dropped some $\varepsilon$ subscripts for convenience. By virtue of
(\ref{eq:condition_on_g_thomas_fermi}), there exists a unique
solution $(u_0,v_0)$ to the linear algebraic system
\[\left\{
\begin{array}{ccc}
                         2g_1a_{1,\varepsilon}u+2g\sqrt{a_{1,\varepsilon}}\sqrt{a_{2,\varepsilon}}v&=&\varepsilon^2 \Delta \sqrt{a_{1,\varepsilon}},\\
                           \\
                         2g\sqrt{a_{1,\varepsilon}}\sqrt{a_{2,\varepsilon}}u+2g_2a_{2,\varepsilon}v&=&\varepsilon^2 \Delta
                         \sqrt{a_{2,\varepsilon}},
                       \end{array}\right.
\]
$x\in B_{\left(R_{1,\varepsilon}-\frac{\delta}{50} \right)}$. It
follows readily that
\begin{equation}\label{equ0v0}
\|u_0\|_{C^2\left(\bar{B}_{\left(R_{1,\varepsilon}-\frac{\delta}{50}
\right)}\right)}\leq C\varepsilon^2\ \ \textrm{and}\ \
\|v_0\|_{C^2\left(\bar{B}_{\left(R_{1,\varepsilon}-\frac{\delta}{50}
\right)}\right)}\leq C\varepsilon^2,
\end{equation}
(keep in mind that $a_{i,\varepsilon}\geq c$ in
$B_{\left(R_{1,\varepsilon}-\frac{\delta}{50} \right)}$, $i=1,2$).
We can write
\begin{equation}\label{equvtildaDef}
(u,v)=(u_0,v_0)+(\tilde{u},\tilde{v}),
\end{equation}
where $\tilde{u}, \tilde{v}$ satisfy
\begin{equation}\label{eqStarMcal}
\mathcal{M}(\tilde{u},\tilde{v})=\left( \begin{array}{c}
                  \varepsilon^2 \Delta u_0 \\
                    \\
                 \varepsilon^2 \Delta v_0
                \end{array}
\right)+\left(\begin{array}{c}
                         g_1(2a_{1,\varepsilon}-\eta_1^2-\sqrt{a}_{1,\varepsilon}\eta_1)u_0+g(2\sqrt{a_{1,\varepsilon}}\sqrt{a_{2,\varepsilon}}
                         -\eta_1\eta_2-\sqrt{a_{2,\varepsilon}}\eta_1)v_0 \\
                           g_2(2a_{2,\varepsilon}-\eta_2^2-\sqrt{a}_{2,\varepsilon}\eta_2)v_0+g(2\sqrt{a_{1,\varepsilon}}\sqrt{a_{2,\varepsilon}}
                         -\eta_1\eta_2-\sqrt{a_{1,\varepsilon}}\eta_2)u_0 \\

                       \end{array}
 \right),
\end{equation}
$x\in B_{\left(R_{1,\varepsilon}-\frac{\delta}{50} \right)}$.

Consider any $\rho\in \left(0,R_{1,\varepsilon}-\frac{\delta}{50}
\right]$. By testing the above equation by $(\tilde{u},\tilde{v})$
in the $L^2(B_{\rho})\times L^2 (B_{\rho})$ sense,  making use of
(\ref{eq:g_mathfrak_definition}), Proposition \ref{proExistSystem},
(\ref{equ0v0}), and Young's inequality,  it follows readily that
\begin{equation}\label{eqiter}\int_{B_\rho}^{}
\left(\varepsilon^2|\nabla \tilde{u}|^2+\varepsilon^2|\nabla
\tilde{v}|^2+\tilde{u}^2+\tilde{v}^2\right)dx\leq
C\varepsilon^{\frac{20}{3}}+C\varepsilon^2|\tilde{u}'(\rho)\tilde{u}(\rho)|+C\varepsilon^2|\tilde{v}'(\rho)\tilde{v}(\rho)|,
\end{equation}
provided that $\varepsilon>0$ is sufficiently small. Setting in this
relation $\rho=R_{1,\varepsilon}-\frac{\delta}{50}$, using
(\ref{eqUniformOut}), (\ref{equ'Unif}), (\ref{eqv'Unif}), and
(\ref{equ0v0}), we obtain that
\[\int_{B_{\left(R_{1,\varepsilon}-\frac{\delta}{50}\right)}}^{}\left(\varepsilon^2|\nabla \tilde{u}|^2+\varepsilon^2|\nabla
\tilde{v}|^2+\tilde{u}^2+\tilde{v}^2\right)dx\leq
C\varepsilon^\frac{20}{3}+C\varepsilon^3.\] Thus, there exists
$r_1\in
\left(R_{1,\varepsilon}-\frac{\delta}{49},R_{1,\varepsilon}-\frac{\delta}{50}\right)$
such that
\[
|\tilde{u}(r_1)|+|\tilde{v}(r_1)|\leq C
\varepsilon^\frac{10}{3}+C\varepsilon^\frac{3}{2}\ \ \textrm{and}\ \
|\tilde{u}'(r_1)|+|\tilde{v}'(r_1)|\leq C \varepsilon^\frac{1}{2}.
\]
Doing the same procedure  with $\rho=r_1$, using the above estimates
instead of (\ref{eqUniformOut}) when estimating the boundary terms
in (\ref{eqiter}), yields that
\[
\int_{B_{r_1}}^{}\left(\varepsilon^2|\nabla
\tilde{u}|^2+\varepsilon^2|\nabla
\tilde{v}|^2+\tilde{u}^2+\tilde{v}^2\right)dx\leq
C\varepsilon^\frac{20}{3}+C\varepsilon^4.
\]
Thus, there exists $r_2\in
\left(R_{1,\varepsilon}-\frac{\delta}{48},r_1\right)$ such that
\[
|\tilde{u}(r_2)|+|\tilde{v}(r_2)|\leq C\varepsilon^\frac{10}{3}+C
\varepsilon^2\ \ \textrm{and}\ \
|\tilde{u}'(r_2)|+|\tilde{v}'(r_2)|\leq C\varepsilon.
\]
Iterating this scheme a finite number of times  provides us with an
$r_*\in
\left(R_{1,\varepsilon}-\frac{\delta}{2},R_{1,\varepsilon}-\frac{\delta}{3}\right)$
such that
\begin{equation}\label{eqr*}
\int_{B_{r_*}}^{}(\tilde{u}^2+\tilde{v}^2)dx\leq
C\varepsilon^\frac{20}{3}.
\end{equation}

Now, via (\ref{equ0v0}), (\ref{eqStarMcal}), and interior elliptic
regularity theory,  we find that
\begin{equation}\label{eqH2forAgmon}
\|\tilde{u}\|_{H^2\left(B_{(R_{1,\varepsilon}-\delta)}\right)}\leq
C\varepsilon^2\ \ \textrm{and}\ \
\|\tilde{v}\|_{H^2\left(B_{(R_{1,\varepsilon}-\delta)}\right)}\leq
C\varepsilon^2.
\end{equation}
By the two-dimensional Agmon inequality \cite[Lem. 13.2]{Agmon}, we
infer that
\[
\|\tilde{u}\|_{L^\infty\left(B_{(R_{1,\varepsilon}-\delta)}\right)}\leq
C\|\tilde{u}\|_{H^2\left(B_{(R_{1,\varepsilon}-\delta)}\right)}^\frac{1}{2}
\|\tilde{u}\|_{L^2\left(B_{(R_{1,\varepsilon}-\delta)}\right)}^\frac{1}{2}\stackrel{(\ref{eqr*}),(\ref{eqH2forAgmon})}{\leq}
C\varepsilon^3.
\]
Analogously, we have
\[
\|\tilde{v}\|_{L^\infty\left(B_{(R_{1,\varepsilon}-\delta)}\right)}\leq
 C\varepsilon^3.
\]
The desired estimate (\ref{eqbootstrapDesired}) follows at once from
(\ref{equvDef}), (\ref{equ0v0}), (\ref{equvtildaDef}),  and the
above two relations.\end{proof}

In the following proposition, we  prove estimates in the
intermediate zones, bridging the estimates (\ref{eqUniformOut}) and
those provided by Proposition \ref{prounifAway}, on the left side of
$R_{i,\varepsilon}$, $i=1,2$.
\begin{proposition}\label{prounifAlg}The following estimates hold:
\[
\left|\eta_{1,\varepsilon}(r)-\check{\eta}_{1,\varepsilon}(r)
\right|\leq C\varepsilon^2 |r-R_{1,\varepsilon}|^{-\frac{3}{2}},\ \
\left|\eta_{2,\varepsilon}(r)-\check{\eta}_{2,\varepsilon}(r)
\right|\leq C\varepsilon^2 |r-R_{1,\varepsilon}|^{-1},\ \ r\in
[R_{1,\varepsilon}-\delta,R_{1,\varepsilon}-D\varepsilon^\frac{2}{3}],
\]
and
\[\left|\eta_{2,\varepsilon}(r)-\check{\eta}_{2,\varepsilon}(r)
\right|\leq C\varepsilon^2 |r-R_{2,\varepsilon}|^{-\frac{3}{2}} \ \
\textrm{if} \ \  r\in
[R_{2,\varepsilon}-\delta,R_{2,\varepsilon}-D\varepsilon^\frac{2}{3}],
\]
for some constants $C,\delta,D>0$ ($\delta$ as in Proposition
\ref{prounifAway}), provided that $\varepsilon>0$ is sufficiently
small.
\end{proposition}
\begin{proof}
We only prove the assertions of the proposition that are
related to $R_{1,\varepsilon}$, since those related to
$R_{2,\varepsilon}$ follow analogously and are in fact considerably
simpler to verify because $\eta_{1,\varepsilon}$ is small beyond all
orders for $r\geq R_{1,\varepsilon}+\delta$.

From (\ref{eqNormtriple}) and (\ref{eqestimMain}), there exists
$C>0$ and a sequence $D_j\to \infty$ such that
\begin{equation}\label{eqDj}
\left|(\eta_{2,\varepsilon}-\check{\eta}_{2,\varepsilon})(R_{1,\varepsilon}-D_j\varepsilon^\frac{2}{3})
\right|\leq C\varepsilon^\frac{4}{3},
\end{equation}
for sufficiently small $\varepsilon>0$, $j\geq 1$.

For $r\in
[R_{1,\varepsilon}-\delta,R_{1,\varepsilon}-D_j\varepsilon^\frac{2}{3}]$,
we can write
\begin{equation}\label{eqphipsiAlg}
\varphi=\eta_{1,\varepsilon}-\check{\eta}_{1,\varepsilon}=
\eta_{1,\varepsilon}-\hat{\eta}_{1,\varepsilon},\ \
\psi=\eta_{2,\varepsilon}-\check{\eta}_{2,\varepsilon}=
\eta_{2,\varepsilon}-\tilde{\eta}_{2,\varepsilon},
\end{equation}
where $\varphi,\psi$ satisfy (\ref{eqsystemPhiPsi}).

Let $\varphi_0,\ \psi_0$ be determined from the problems
\[\left\{
\begin{array}{l}
  -\varepsilon^2 \Delta \varphi_0 +\left[\left(g_1-\frac{g^2}{g_2} \right)(3\hat{\eta}_{1,\varepsilon}^2
                                        -a_{1,\varepsilon}) +2\frac{g^2}{g_2}\hat{\eta}_{1,\varepsilon}^2\right]\varphi_0=E_1,
                                        \ \ r\in (R_{1,\varepsilon}-\delta,R_{1,\varepsilon}-D_j\varepsilon^\frac{2}{3}), \\
    \\
  \varphi_0(R_{1,\varepsilon}-\delta)=\varphi(R_{1,\varepsilon}-\delta),\ \ \varphi_0(R_{1,\varepsilon}-D_j\varepsilon^\frac{2}{3})=
  \varphi(R_{1,\varepsilon}-D_j\varepsilon^\frac{2}{3}),
\end{array}\right.
\]
and
\[\left\{
\begin{array}{l}
  -\varepsilon^2 \Delta \psi_0 +2g_2\tilde{\eta}_{2,\varepsilon}^2                                     \psi_0=N_2(\varphi,\psi)+E_2,
                                        \ \ r\in (R_{1,\varepsilon}-\delta,R_{1,\varepsilon}-D_j\varepsilon^\frac{2}{3}), \\
    \\
  \psi_0(R_{1,\varepsilon}-\delta)=\psi(R_{1,\varepsilon}-\delta),\ \ \psi_0(R_{1,\varepsilon}-D_j\varepsilon^\frac{2}{3})=
  \psi(R_{1,\varepsilon}-D_j\varepsilon^\frac{2}{3}),
\end{array}\right.
\]
where $E_i,\ N_i(\cdot,\cdot)$, $i=1,2$, are as in (\ref{eqError})
and (\ref{eqN}) respectively. By virtue of (\ref{eqE1}),
(\ref{eqE2}), (\ref{eqerrorBasic}), (\ref{eqerorAlg-}),
(\ref{eqerrorAlg}), (\ref{eqLpot1}), (\ref{eqpotlowerc}),
(\ref{eqUniformOut}), Proposition \ref{prounifAway}, and
(\ref{eqDj}), via a standard barrier argument, we deduce that
\begin{equation}\label{eqphi0bound}
\left|\varphi_0(r) \right|\leq C\varepsilon^2+C\varepsilon
\exp\left\{c\frac{r-R_{1,\varepsilon}+D_j\varepsilon^\frac{2}{3}}{\varepsilon^\frac{2}{3}}
\right\},
\end{equation}
\begin{equation}\label{eqpsi0bound}
\left|\psi_0(r) \right|\leq C\varepsilon^2+C\varepsilon^4
|r-R_{1,\varepsilon}|^{-4}+C\varepsilon^\frac{4}{3}
\exp\left\{c\frac{r-R_{1,\varepsilon}+D_j\varepsilon^\frac{2}{3}}{\varepsilon}
\right\},
\end{equation}
if
$r\in[R_{1,\varepsilon}-\delta,R_{1,\varepsilon}-D_j\varepsilon^\frac{2}{3}]$.

We can write
\[
\varphi=\varphi_0+\tilde{\varphi},\ \ \psi=\psi_0+\tilde{\psi},
\]
where $\tilde{\varphi},\ \tilde{\psi}$ satisfy
\begin{equation}\label{eqLtilda}
\left\{\begin{array}{c}
         \mathcal{L}(\tilde{\varphi},
\tilde{\psi})=\left(\begin{array}{c}
                                                   \tilde{E_1} \\
                                                     \\
                                                   \tilde{E_2}
                                                 \end{array}
 \right),\ \
 r\in(R_{1,\varepsilon}-\delta,R_{1,\varepsilon}-D_j\varepsilon^\frac{2}{3}), \\
           \\
         \tilde{\varphi}(R_{1,\varepsilon}-\delta)=\tilde{\psi}(R_{1,\varepsilon}-\delta)=0,\
\
\tilde{\varphi}(R_{1,\varepsilon}-D_j\varepsilon^\frac{2}{3})=\tilde{\psi}(R_{1,\varepsilon}-D_j\varepsilon^\frac{2}{3})=0,
\end{array}
 \right.
\end{equation}
with $\mathcal{L}$ as in (\ref{eqlinearization}), for some functions
$\tilde{E}_i$, $i=1,2$, satisfying the following pointwise
estimates:
\begin{equation}\label{eqE1tilda}
\begin{array}{rcl}
  |\tilde{E}_1| & \leq &  C |\hat{\eta}_{1,\varepsilon} \psi_0|+ |N_1(\varphi,\psi)| \\
    &   &   \\
  \textrm{via}\  (\ref{eqLpot2}), (\ref{eqN}), (\ref{eqUniformOut}),
(\ref{eqpsi0bound}) & \leq & C\varepsilon^2+C\varepsilon^4
|r-R_{1,\varepsilon}|^{-\frac{7}{2}}+C_j\varepsilon^\frac{5}{3}
\exp\left\{c\frac{r-R_{1,\varepsilon}+D_j\varepsilon^\frac{2}{3}}{\varepsilon}
\right\},
\end{array}
\end{equation}
\begin{equation}\label{eqE2tilda}
   |\tilde{E}_2|   \leq    C |\hat{\eta}_{1,\varepsilon} \varphi_0| \stackrel{(\ref{eqLpot2}), (\ref{eqphi0bound})}{\leq} C\varepsilon^2+C_j\varepsilon^\frac{4}{3}
\exp\left\{c\frac{r-R_{1,\varepsilon}+D_j\varepsilon^\frac{2}{3}}{\varepsilon^\frac{2}{3}}
\right\},
\end{equation}
for $
r\in[R_{1,\varepsilon}-\delta,R_{1,\varepsilon}-D_j\varepsilon^\frac{2}{3}]$.

Our plan is to solve the second equation in (\ref{eqLtilda}) for
$\tilde{\psi}$ and substitute into the first, thus reducing the
system to one scalar equation for $\tilde{\varphi}$. Then, we
derive estimates for $\tilde{\varphi}$ in some carefully chosen
weighted norms that we  define afterwards. Let
$I=(R_{1,\varepsilon}-\delta,R_{1,\varepsilon}-D_j\varepsilon^\frac{2}{3})$
and
\[
\rho(r)=R_{1,\varepsilon}-r,\ \ r\in I,
\]
for $\phi\in C^2(I)$, we define
\begin{equation}\label{eqnormStar}
\|\phi\|_*=\varepsilon^2\|\rho^\frac{1}{2}\phi_{rr}\|_{L^\infty(I)}+\varepsilon^2\|\rho^{-\frac{1}{2}}\phi_{r}\|_{L^\infty(I)}
+\|\rho^\frac{3}{2}\phi\|_{L^\infty(I)},
\end{equation}
(for related weighted norms we refer the interested reader to the
monograph \cite{pacard-riviere} and the references therein). In
particular, we  rely on the following a-priori estimate: there
exist $\varepsilon_0,j_0,K>0$ such that if $h\in C(\bar{I})$ and
$\phi\in C^2_r\left(|x|\in \bar{I}\right)$  satisfy
\begin{equation}\label{eqAlgLph=h}
-\varepsilon^2\Delta \phi+\left(g_1-\frac{g^2}{g_2}
\right)(3\hat{\eta}_{1,\varepsilon}^2-a_{1,\varepsilon})\phi=h\ \
\textrm{in}\ I;\ \phi=0\ \ \textrm{on}\ \partial I,
\end{equation}
with $0<\varepsilon<\varepsilon_0$, $j\geq j_0$, then
\begin{equation}\label{eqaprioriWeight}
\|\phi\|_*\leq K\|\rho^\frac{1}{2}h\|_{L^\infty(I)}.
\end{equation}
We stress that the above constant $K$ is also independent of
$\delta$ (i.e. $\varepsilon_0=\varepsilon_0(\delta)$). In the
remainder of this proof, we  denote by $k/K$ a small/large
generic constant that is independent of large $j$ and small
$\delta,\varepsilon$.
 The proof of this
estimate proceeds in two steps. Firstly, similarly to \cite[Prop.
3.5]{KaraliSourdisGround}, using the following consequence of
\eqref{eqGSuAlg}
\begin{equation}\label{eqAlgLowerk} k|r-R_{1,\varepsilon}|\leq 3\hat{\eta}_{1,\varepsilon}^2-a_{1,\varepsilon}\leq
K |r-R_{1,\varepsilon}|, \ \ r\in I,
\end{equation}
 and the maximum principle (in the
equation for $\rho^\frac{3}{2}\phi$), one obtains the partial
estimate
\begin{equation}\label{eqAlgUnifK}\|\rho^\frac{3}{2}\phi\|_{L^\infty(I)}\leq
K\|\rho^\frac{1}{2}h\|_{L^\infty(I)}.\end{equation} Then, the full
a-priori estimate follows by going back to the equation for $\phi$
and using the upper bound in (\ref{eqAlgLowerk}). The details are
 given in Appendix \ref{ApAlg}. From now on, we fix such a large
$j$ and drop the subscript from $D_j$.

In view of the second row in (\ref{eqlinearization}) and
(\ref{eqLtilda}), we can write
\begin{equation}\label{eqpsitilda}
\tilde{\psi}=-\frac{g}{g_2}\frac{\hat{\eta}_{1,\varepsilon}}{\tilde{\eta}_{2,\varepsilon}}
\tilde{\varphi}+w+z,\ \ r\in I,
\end{equation}
where
\begin{equation}\label{eqwEq}
-\varepsilon^2\Delta
w+2g_2\tilde{\eta}_{2,\varepsilon}^2w=-\varepsilon^2
\frac{g}{g_2}\Delta\left(\frac{\hat{\eta}_{1,\varepsilon}}{\tilde{\eta}_{2,\varepsilon}}
\tilde{\varphi} \right)\ \ \textrm{in}\ I;\ \ w=0\ \ \textrm{on}\ \
\partial I,
\end{equation}
and
\begin{equation}\label{eqzEq}
-\varepsilon^2\Delta
z+2g_2\tilde{\eta}_{2,\varepsilon}^2z=\tilde{E}_2\ \ \textrm{in}\
I;\ \ z=0\ \ \textrm{on}\ \ \partial I.
\end{equation}
Using the pointwise estimates
\[\left\{\begin{array}{lll}
    0<\hat{\eta}_{1,\varepsilon}\leq K\rho^\frac{1}{2}, & |\nabla\hat{\eta}_{1,\varepsilon}|\leq K\rho^{-\frac{1}{2}}, & |\Delta\hat{\eta}_{1,\varepsilon}|\leq K\rho^{-\frac{3}{2}}, \\
      &   &   \\
    k\leq \tilde{\eta}_{2,\varepsilon}\leq K, & |\nabla\tilde{\eta}_{2,\varepsilon}|\leq K, &
|\Delta\tilde{\eta}_{2,\varepsilon}|\leq K+ K\varepsilon^2\rho^{-4},
  \end{array}\right.
\]
for $r\in I$, which follow readily  from (\ref{eqgroundstate1}),
(\ref{eqGSuAlg}) (having increased $j$ if needed),
(\ref{eqeta2cut}), (\ref{eqerrorBasic}), (\ref{eqerorAlg-}),
(\ref{eqerrorAlg}),  and (\ref{eqLpot2}), we can bound pointwise the
right-hand side of (\ref{eqwEq}) as
\[
\begin{array}{rcl}
  \varepsilon^2 \frac{g}{g_2} \left|
\Delta\left(\frac{\hat{\eta}_{1,\varepsilon}}{\tilde{\eta}_{2,\varepsilon}}
\tilde{\varphi} \right)\right| & \leq & K\varepsilon^2
\left(\rho^\frac{1}{2}|\Delta \tilde{\varphi}|+
\rho^{-\frac{1}{2}}|\nabla \tilde{\varphi}|+\rho^{-\frac{3}{2}}|\tilde{\varphi}| \right) \\
    &   &   \\
    & \leq & K \left(\varepsilon^2\rho^\frac{1}{2}| \tilde{\varphi}_{rr}|+
\varepsilon^2\rho^{-\frac{1}{2}}|
\tilde{\varphi}_r|+\rho^{\frac{3}{2}}|\tilde{\varphi}| \right)\\
&   &   \\
& =&K\|\tilde{\varphi}\|_*.
\end{array}
\]
Hence, by the maximum principle, we deduce that
\begin{equation}\label{eqwbound}
\|w\|_{L^\infty(I)}\leq K\|\tilde{\varphi}\|_*.
\end{equation}
On the other side, from (\ref{eqE2tilda}), (\ref{eqzEq}) and a
standard comparison argument, it follows that
\begin{equation}\label{eqzbound}
\left|z(r) \right|\leq C\varepsilon^2+C\varepsilon^\frac{4}{3}
\exp\left\{c\frac{r-R_{1,\varepsilon}+D\varepsilon^\frac{2}{3}}{\varepsilon^\frac{2}{3}}
\right\},\ \ r\in I.
\end{equation}
Substituting (\ref{eqpsitilda}) into the first equation of
(\ref{eqLtilda}), recalling (\ref{eqlinearization}), we arrive at
\[
-\varepsilon^2\Delta \tilde{\varphi}+\left(g_1-\frac{g^2}{g_2}
\right)(3\hat{\eta}_{1,\varepsilon}^2-a_{1,\varepsilon})\tilde{\varphi}=\tilde{E}_1
-2g\hat{\eta}_{1,\varepsilon}\tilde{\eta}_{2,\varepsilon}(w+z)\ \
\textrm{in}\ I;\ \ \tilde{\varphi}=0\ \textrm{on}\ \partial I.
\]
Making use of the a-priori estimate (\ref{eqaprioriWeight}), bound
(\ref{eqwbound}), and the easy estimates
\[
\|\rho^\frac{1}{2}\tilde{E}_1\|_{L^\infty(I)}\leq C\varepsilon^2 \ \
(\textrm{recall}\ (\ref{eqE1tilda})),\ \ \
\|\rho^\frac{1}{2}\hat{\eta}_{1,\varepsilon}z\|_{L^\infty(I)}\leq
C\varepsilon^2\ \ (\textrm{recall}\ (\ref{eqLpot2}),\
(\ref{eqzbound})),
\]
we obtain that
\[
\|\tilde{\varphi}\|_*\leq C\varepsilon^2+ \delta
K\|\tilde{\varphi}\|_*,
\]
where we also exploited that $0<\hat{\eta}_{1,\varepsilon}\leq K
\rho^\frac{1}{2}\leq K \delta^\frac{1}{2}$ in $I$. Consequently,
choosing a sufficiently small $\delta$,  and fixing it from now on,
we infer that
\[
\|\tilde{\varphi}\|_*\leq C\varepsilon^2.
\]

 In particular,
for small $\varepsilon$, we have that
\[
\left|\tilde{\varphi}(r) \right|\leq C\varepsilon^2
|r-R_{1,\varepsilon}|^{-\frac{3}{2}},\ \ r\in
[R_{1,\varepsilon}-\delta,R_{1,\varepsilon}-D\varepsilon^\frac{2}{3}].
\]
In turn, from the second equation in (\ref{eqLtilda}), recalling
(\ref{eqLpot2}) and (\ref{eqE2tilda}), via a standard barrier
argument, we find that
\[
\left|\tilde{\psi}(r) \right|\leq C\varepsilon^2
|r-R_{1,\varepsilon}|^{-1},\ \ r\in
[R_{1,\varepsilon}-\delta,R_{1,\varepsilon}-D\varepsilon^\frac{2}{3}].
\]
The desired assertion  of the proposition now follows at once from
(\ref{eqphipsiAlg}), (\ref{eqphi0bound}), (\ref{eqpsi0bound}), and
the above two relations.\end{proof}

\begin{lemma}\label{lemUnifeta2nearR1}
Given $D>0$, we have that
\[
\left|\eta_{2,\varepsilon}(r)-\check{\eta}_{2,\varepsilon}(r)
\right| \leq C\varepsilon^\frac{4}{3},\ \ |r-R_{1,\varepsilon}|\leq
D\varepsilon^\frac{2}{3},
\]
provided that $\varepsilon>0$ is sufficiently small.
\end{lemma}
\begin{proof}
As in the beginning of the proof of Proposition \ref{prounifAlg},
given $D>0$, if $\varepsilon>0$ is sufficiently small, there exist
$r_-\in \left(R_{1,\varepsilon}
-(D+1)\varepsilon^\frac{2}{3},R_{1,\varepsilon}
-D\varepsilon^\frac{2}{3} \right)$ and $r_+\in
\left(R_{1,\varepsilon} + D
\varepsilon^\frac{2}{3},R_{1,\varepsilon}
+(D+1)\varepsilon^\frac{2}{3} \right)$ such that
\[
\left|\psi(r_\pm) \right|\leq C\varepsilon^\frac{4}{3},
\]
where $\psi=\eta_{2,\varepsilon}-\check{\eta}_{2,\varepsilon}$.
Keeping in mind the proof of Proposition \ref{proerror},
(\ref{eqlinearization}), (\ref{eqLpot2}) and (\ref{eqUniformOut})),
it follows readily from the second equation of
(\ref{eqsystemPhiPsi}) that $\psi$ satisfies
\[
-\varepsilon^2\Delta
\psi+2g_2\tilde{\eta}_{2,\varepsilon}^2\psi=\mathcal{O}(\varepsilon^\frac{4}{3}),
\ \ \textrm{uniformly\ on}\  [r_-,r_+], \ \textrm{as}\
\varepsilon\to 0.
\]
  The assertion of the
lemma follows directly from the above two relations and the maximum
principle, since $\tilde{\eta}_{2,\varepsilon}\geq c$ in this
region.\end{proof}

\begin{lemma}\label{lemDecaytoetai}
If $\varepsilon>0$ is sufficiently small, we have
\[
\left|\eta_{1,\varepsilon}(r)-\check{\eta}_{1,\varepsilon}(r)
\right|\leq C\varepsilon
\exp\left\{c\frac{R_{1,\varepsilon}-r}{\varepsilon^\frac{2}{3}}
\right\},\ \ R_{1,\varepsilon} \leq r\leq R_{1,\varepsilon}+\delta,
\]
\[
\left|\eta_{2,\varepsilon}(r)-\check{\eta}_{2,\varepsilon}(r)
\right|\leq C\varepsilon^2 +C\varepsilon^\frac{4}{3}
\exp\left\{c\frac{R_{1,\varepsilon}-r}{\varepsilon^\frac{2}{3}}
\right\},\ \ R_{1,\varepsilon} \leq r\leq R_{1,\varepsilon}+\delta,
\]
and
\[\left|\eta_{2,\varepsilon}(r)-\check{\eta}_{2,\varepsilon}(r)
\right|\leq C\varepsilon
\exp\left\{c\frac{R_{2,\varepsilon}-r}{\varepsilon^\frac{2}{3}}
\right\},\ \ R_{2,\varepsilon}\leq r\leq R_{2,\varepsilon}+\delta.\]
\end{lemma}
\begin{proof}
We  only prove the estimates that concern $R_{1,\varepsilon}$
because those concerning $R_{2,\varepsilon}$ follow analogously. As
in the proof of Proposition \ref{prounifAlg}, let
$\varphi=\eta_{1,\varepsilon}-\check{\eta}_{1,\varepsilon}$ and
$\psi=\eta_{2,\varepsilon}-\check{\eta}_{2,\varepsilon}$. In view of
(\ref{eqlinearization}), (\ref{eqsystemPhiPsi}), (\ref{eqN}), the
relations
\[
E_1=0,\ \ |E_2|\leq C
\varepsilon^\frac{4}{3}\exp\left\{c\frac{R_{1,\varepsilon}-r}{\varepsilon^\frac{2}{3}}
\right\}+C\varepsilon^2,\ \ r\in
[R_{1,\varepsilon},R_{1,\varepsilon}+3\delta]\ (\textrm{from}\
(\ref{eqE1}),\ (\ref{eqE2}),\ (\ref{eqerrorBasic})),
\]
 (\ref{eqUniformOut}), and Proposition
\ref{prounifAlg}, we infer that
\begin{equation}\label{eqdecayEqs}\left\{
\begin{array}{l}
  -\varepsilon^2\Delta \varphi+p(r)\varphi=\mathcal{O}(\hat{\eta}_{1,\varepsilon}\psi), \\
    \\
  -\varepsilon^2\Delta
  \psi+q(r)\psi=\mathcal{O}(\varepsilon^2)
  +\mathcal{O}(\varepsilon^\frac{4}{3})\exp\left\{c\frac{R_{1,\varepsilon}-r}{\varepsilon^\frac{2}{3}}\right\},
\end{array}
\right.\end{equation} uniformly on
$[R_{1,\varepsilon},R_{1,\varepsilon}+3\delta]$, as $\varepsilon\to
0$, for some smooth functions $p,q$ satisfying
\begin{equation}\label{eqdecayPQ}
p\geq c\varepsilon^\frac{2}{3}\ \ \textrm{and}\ \ q\geq c\
(\textrm{recall}\ (\ref{eqLpot1})\ \textrm{and}\
(\ref{eqpotlowerc})).
\end{equation}
 Note that, from
(\ref{eqUniformOut}) and Lemma \ref{lemUnifeta2nearR1}, we have
\begin{equation}\label{eqdecayBdrys}
\left\{\begin{array}{ll}
         \varphi(R_{1,\varepsilon})=\mathcal{O}(\varepsilon), & \varphi(R_{1,\varepsilon}+2\delta)=\mathcal{O}(\varepsilon), \\
           &   \\
         \psi(R_{1,\varepsilon})=\mathcal{O}(\varepsilon^\frac{4}{3}), & \psi(R_{1,\varepsilon}+3\delta)=\mathcal{O}(\varepsilon),
       \end{array}
 \right.
\end{equation}
as $\varepsilon\to 0$. A standard barrier argument yields that
\[
\left|\psi(r) \right|\leq C
\varepsilon^\frac{4}{3}\exp\left\{c\frac{R_{1,\varepsilon}-r}{\varepsilon^\frac{2}{3}}
\right\}+C\varepsilon^2+C\varepsilon
\exp\left\{c\frac{r-R_{1,\varepsilon}-3\delta}{\varepsilon}
\right\},\ \ r\in [R_{1,\varepsilon},R_{1,\varepsilon}+3\delta],
\]
which implies that
\[
\left|\psi(r) \right|\leq C
\varepsilon^\frac{4}{3}\exp\left\{c\frac{R_{1,\varepsilon}-r}{\varepsilon^\frac{2}{3}}
\right\}+C\varepsilon^2,\ \ r\in
[R_{1,\varepsilon},R_{1,\varepsilon}+2\delta],
\]
provided that $\varepsilon>0$ is sufficiently small, as asserted.
Now, via (\ref{eqdecayEqs}) and (\ref{eq:proGS2}), we arrive at
\[
-\varepsilon^2\Delta
\varphi+p(r)\varphi=\mathcal{O}(\varepsilon^\frac{5}{3})\exp\left\{c\frac{R_{1,\varepsilon}-r}{\varepsilon^\frac{2}{3}}
\right\},
\]
uniformly on $[R_{1,\varepsilon},R_{1,\varepsilon}+2\delta]$, as
$\varepsilon\to 0$. Keeping in mind (\ref{eqdecayPQ}) and
(\ref{eqdecayBdrys}), a standard barrier argument yields that
\[
\left|\varphi(r) \right|\leq C
\varepsilon\exp\left\{c\frac{R_{1,\varepsilon}-r}{\varepsilon^\frac{2}{3}}
\right\}+C\varepsilon
\exp\left\{c\frac{r-R_{1,\varepsilon}-2\delta}{\varepsilon^\frac{2}{3}}
\right\},\ \ r\in [R_{1,\varepsilon},R_{1,\varepsilon}+2\delta],
\]
which implies that
\[
\left|\varphi(r) \right|\leq 2C
\varepsilon\exp\left\{c\frac{R_{1,\varepsilon}-r}{\varepsilon^\frac{2}{3}}
\right\},\ \ r\in [R_{1,\varepsilon},R_{1,\varepsilon}+\delta],
\]
as asserted.
\end{proof}

\subsection{Improved estimate for the Lagrange multipliers}
In the sequel, building on Proposition
\ref{prop:convergence_lagrange_multipliers}, via the results of the
previous subsection, we are able to considerably improve the
estimate for $\lambda_{i,\varepsilon}-\lambda_{i,0}$ of the
aforementioned proposition.

\begin{proposition}\label{proLagrangeImproved}
If $\varepsilon>0$ is sufficiently small, we have
\[
|\lambda_{i,\varepsilon}-\lambda_{i,0}|\leq C|\log
\varepsilon|\varepsilon^2,\ \ i=1,2.
\]
\end{proposition}
\begin{proof}
Motivated by the proof of the corresponding estimate for the
scalar equation, as given in \cite[Thm. 1.1]{KaraliSourdisGround},
we  first show that
\begin{equation}\label{eqLagr0bef}
                                 \int_{\mathbb{R}^2}^{}(\eta_{1,\varepsilon}^2-a_{1,\varepsilon}^+)dx=\mathcal{O}(|\log
\varepsilon|\varepsilon^2),
\end{equation}
\begin{equation}\label{eqLagr0}
                                 \int_{B_{R_{1,\varepsilon}}}^{}(\eta_{2,\varepsilon}^2-a_{2,\varepsilon})dx
+\int_{\mathbb{R}^2\setminus
B_{R_{1,\varepsilon}}}^{}\left[\eta_{2,\varepsilon}^2-\left(a_{2,\varepsilon}+\frac{g}{g_2}a_{1,\varepsilon}\right)^+\right]dx
=\mathcal{O}(|\log \varepsilon|\varepsilon^2),
                                                              \end{equation}
as $\varepsilon\to 0$, and then exploit that
\begin{equation}\label{eqLagr-}
\int_{\mathbb{R}^2}^{}\eta_{i,\varepsilon}^2dx=\int_{\mathbb{R}^2}^{}a_idx=1,\
\ i=1,2.
\end{equation}

It suffices to establish only the validity of estimate
(\ref{eqLagr0}) because that of (\ref{eqLagr0bef}) follows verbatim.
By (\ref{eqeta2cut}), (\ref{eqeta2check}), Proposition
\ref{prounifAlg} and Lemma \ref{lemUnifeta2nearR1}, we obtain that
\[
\begin{array}{rcl}
  \eta_{2,\varepsilon}^2-a_{2,\varepsilon} & = & \tilde{\eta}_{2,\varepsilon}^2-a_{2,\varepsilon}+
  \mathcal{O}\left(\varepsilon^2|r-R_{1,\varepsilon}|^{-1}\right) \\
   &  &  \\
    & =  & \frac{g}{g_2}(a_{1,\varepsilon}-\hat{\eta}_{1,\varepsilon}^2)+
  \mathcal{O}\left(\varepsilon^2|r-R_{1,\varepsilon}|^{-1}\right),
\end{array}
\]
uniformly on
$[R_{1,\varepsilon}-\delta,R_{1,\varepsilon}-D\varepsilon^\frac{2}{3}]$,
as $\varepsilon\to 0$. Analogously, making use of Lemma
\ref{lemUnifeta2nearR1}, we see that
\[
\eta_{2,\varepsilon}^2-a_{2,\varepsilon}=\frac{g}{g_2}(a_{1,\varepsilon}-\hat{\eta}_{1,\varepsilon}^2)+
  \mathcal{O}(\varepsilon^\frac{4}{3}),
\]
uniformly on
$[R_{1,\varepsilon}-D\varepsilon^\frac{2}{3},R_{1,\varepsilon}+D\varepsilon^\frac{2}{3}]$,
as $\varepsilon\to 0$. Hence, via Proposition \ref{prounifAway}, we
find that
\begin{equation}\label{eqLagr1}
\begin{array}{rcl}
                                \int_{B_{R_{1,\varepsilon}}}^{}(\eta_{2,\varepsilon}^2-a_{2,\varepsilon})dx &
                                = & \frac{g}{g_2}\int_{B_{R_{1,\varepsilon}}}^{}(a_{1,\varepsilon}-\hat{\eta}_{1,\varepsilon}^2)dx+
  \mathcal{O}(|\log \varepsilon|\varepsilon^2) \\
                                  &  &  \\
                                   & \stackrel{(\ref{eqGSuUnif})}{=} &
                                   \frac{g}{g_2}\int_{R_{1,\varepsilon}-\delta<|x|<R_{1,\varepsilon}}^{}(a_{1,\varepsilon}-\hat{\eta}_{1,\varepsilon}^2)dx+
  \mathcal{O}(|\log \varepsilon|\varepsilon^2),
                               \end{array}
\end{equation} as $\varepsilon\to 0$.
Similarly, keeping in mind (\ref{eq:proGS2}), we have
\begin{equation}\label{EqLagr2}
 \int_{R_{1,\varepsilon}<|x|<R_{2,\varepsilon}-\delta}^{}\left(\eta_{2,\varepsilon}^2
 -a_{2,\varepsilon}-\frac{g}{g_2}a_{1,\varepsilon} \right)dx=
 -\frac{g}{g_2}\int_{R_{1,\varepsilon}<|x|<R_{1,\varepsilon}+\delta}^{}
 \hat{\eta}_{1,\varepsilon}^2dx+\mathcal{O}(\varepsilon^2),
\end{equation}
as $\varepsilon\to 0$, where we use Lemmas
\ref{lemUnifeta2nearR1}-\ref{lemDecaytoetai} instead of Proposition
\ref{prounifAlg}. On the other side, thanks to (\ref{eqGSuUnif}) and
Proposition \ref{prounifAlg}, for $r\in
[R_{2,\varepsilon}-\delta,R_{2,\varepsilon}-D\varepsilon^\frac{2}{3}]$,
we find that
\[\begin{array}{rcl}
    \eta_{2,\varepsilon}^2
 -\left(a_{2,\varepsilon}+\frac{g}{g_2}a_{1,\varepsilon}\right) & = &
  \hat{\eta}_{2,\varepsilon}^2
 -\left(a_{2,\varepsilon}+\frac{g}{g_2}a_{1,\varepsilon}\right)+2\hat{\eta}_{2,\varepsilon}
 (\eta_{2,\varepsilon}-\hat{\eta}_{2,\varepsilon})
 +(\eta_{2,\varepsilon}-\hat{\eta}_{2,\varepsilon})^2 \\
      &   &   \\
 & = & \hat{\eta}_{2,\varepsilon}^2
 -\left(a_{2,\varepsilon}+\frac{g}{g_2}a_{1,\varepsilon}\right)
+\mathcal{O}\left(\varepsilon^2|r-R_{2,\varepsilon}|^{-1}\right),
  \end{array}
\]
uniformly, as $\varepsilon\to 0$. Analogously, making use of
(\ref{eqGSuUnif}) and (\ref{eqUniformOut}), we see that
\[
\eta_{2,\varepsilon}^2
 -\left(a_{2,\varepsilon}+\frac{g}{g_2}a_{1,\varepsilon}\right)=
 \hat{\eta}_{2,\varepsilon}^2
 -\left(a_{2,\varepsilon}+\frac{g}{g_2}a_{1,\varepsilon}\right)
+\mathcal{O}(\varepsilon^\frac{4}{3}),
\]
uniformly on
$[R_{2,\varepsilon}-D\varepsilon^\frac{2}{3},R_{2,\varepsilon}+D\varepsilon^\frac{2}{3}]$,
as $\varepsilon\to 0$. Thus,  we get that
\begin{equation}\label{eqLagr3}
\int_{R_{2,\varepsilon}-\delta<|x|<R_{2,\varepsilon}}^{}
\left(\eta_{2,\varepsilon}^2
 -a_{2,\varepsilon}-\frac{g}{g_2}a_{1,\varepsilon}\right)dx=
 \int_{R_{2,\varepsilon}-\delta<|x|<R_{2,\varepsilon}}^{}
\left(\hat{\eta}_{2,\varepsilon}^2
 -a_{2,\varepsilon}-\frac{g}{g_2}a_{1,\varepsilon}\right)dx
+\mathcal{O}(|\log \varepsilon|\varepsilon^2),
\end{equation}
as $\varepsilon\to 0$. Similarly, using Lemma \ref{lemDecaytoetai}
instead of Proposition \ref{prounifAlg}, keeping in mind
(\ref{eq:proGS2}), we obtain that
\begin{equation}\label{eqLagr4}
\int_{|x|>R_{2,\varepsilon}}^{}\eta_{2,\varepsilon}^2dx=
\int_{R_{2,\varepsilon}<|x|<R_{2,\varepsilon}+\delta}^{}\hat{\eta}_{2,\varepsilon}^2dx+\mathcal{O}(\varepsilon^2)\
\ \textrm{as}\ \varepsilon\to 0.
\end{equation}
Now, estimate  (\ref{eqLagr0}) follows readily by adding relations
(\ref{eqLagr1}), (\ref{EqLagr2}), (\ref{eqLagr3}), (\ref{eqLagr4}),
and using the estimates
\[
\int_{\left||x|-R_{1,\varepsilon} \right|<\delta}^{}
(\hat{\eta}_{1,\varepsilon}^2-a_{1,\varepsilon}^+)dx=\mathcal{O}(|\log
\varepsilon|\varepsilon^2),\ \ \int_{\left||x|-R_{2,\varepsilon}
\right|<\delta}^{}\left[\hat{\eta}_{2,\varepsilon}^2-\left(a_{2,\varepsilon}
+\frac{g}{g_2}a_{1,\varepsilon}\right)^+ \right]dx=\mathcal{O}(|\log
\varepsilon|\varepsilon^2),
\]
as $\varepsilon\to 0$, which follow from the proof of Theorem 1.1 in
\cite{KaraliSourdisGround}. The proof of relation (\ref{eqLagr0}) is
complete.

By virtue of (\ref{eqRi-R0}), increasing the value  of $D$, if
needed, we may assume that $|R_{i,\varepsilon}-R_{i,0}|\leq D |\log
\varepsilon|^\frac{1}{2}\varepsilon$, $i=1,2$, for small
$\varepsilon>0$. It follows from  (\ref{eqLagr0bef}),
(\ref{eqLagr0}) and (\ref{eqLagr-}), recalling
 (\ref{eqa12}) and (\ref{eq:convergence_lagrange_multipliers}),   that
\[
\int_{|x|<R_{1,0}-D|\log
\varepsilon|^\frac{1}{2}\varepsilon}^{}(a_{1,\varepsilon}-a_{1,0})dx=\mathcal{O}
(|\log \varepsilon|\varepsilon^2),
\]
\[\begin{array}{c}
    \int_{|x|<R_{1,0}-D|\log
\varepsilon|^\frac{1}{2}\varepsilon}^{}(a_{2,\varepsilon}-a_{2,0})dx+     \\
\\
  +
\int_{R_{1,0}+D|\log
\varepsilon|^\frac{1}{2}\varepsilon<|x|<R_{2,0}-D|\log
\varepsilon|^\frac{1}{2}\varepsilon}^{}\left[(a_{2,\varepsilon}-a_{2,0})
+\frac{g}{g_2}(a_{1,\varepsilon}-a_{1,0})\right]dx      =
\mathcal{O} (|\log \varepsilon|\varepsilon^2),
  \end{array}
 \]
as $\varepsilon\to 0$. In view of (\ref{eqa12}), (\ref{eqR12}), and
(\ref{eq:convergence_lagrange_multipliers}), this leads to the
following system:
\[
(\lambda_{1,\varepsilon}-\lambda_{1,0})-\frac{g}{g_2}(\lambda_{2,\varepsilon}-\lambda_{2,0})=\mathcal{O}
(|\log \varepsilon|\varepsilon^2),
\]
\[
\frac{1}{\Gamma}\left[(\lambda_{2,\varepsilon}-\lambda_{2,0})-\frac{g}{g_1}(\lambda_{1,\varepsilon}-\lambda_{1,0})
\right]\left(\lambda_{1,\varepsilon}-\frac{g}{g_2}\lambda_{2,\varepsilon}\right)
+(\lambda_{2,\varepsilon}-\lambda_{1,\varepsilon})(\lambda_{2,\varepsilon}-\lambda_{2,0})=\mathcal{O}
(|\log \varepsilon|\varepsilon^2),
\]
as $\varepsilon\to 0$. Now, recalling that $g<g_2$, the assertion of
the proposition follows straightforwardly.
\end{proof}

\subsection{Proof of Theorem \ref{thmLong}}\label{secProofs}
 \begin{proof}
 Let $(\eta_1,\eta_2)$ be the unique positive minimizer of
$E_\varepsilon^0$ in $\mathcal{H}$ provided by Theorem
\ref{thm:uniqueness} (2). We saw in Proposition
\ref{prop:convergence_lagrange_multipliers} that the associated
Lagrange multipliers $\la_{i,\ep}$ satisfy
$|\la_{i,\ep}-\la_{i,0}|\leq \ep |\log\ep|^{1/2}$, $i=1,2$. In view
of (\ref{eqUniformOut}) and Proposition \ref{proPositive}, the
solution $(\eta_{1,\varepsilon},\eta_{2,\varepsilon})$ that is
provided by Proposition \ref{proExistSystem} also fashions a
positive radial solution of the system (\ref{eq:main_eta_1_eta_2}),
with \emph{the same} Lagrange multipliers $\la_{i,\ep}$.  Therefore,
by Theorem \ref{thm:uniqueness} (1), it coincides with
$(\eta_1,\eta_2)$.
%In particular, it holds that
%$\|\eta_{i,\varepsilon}\|_{L^2(\mathbb{R}^2)}=1$, $i=1,2$.

%Let $(\eta_1,\eta_2)$ be as in Proposition
%\ref{prop:existence_radial_minimizer}, namely a positive radial minimizer of
%$E_\varepsilon^0$ in $\mathcal{H}$. The solution
%$(\eta_{1,\varepsilon},\eta_{2,\varepsilon})$ that is provided by
%Theorem  \ref{thmTHM} also fashions a positive radial solution of
%the system (\ref{eq:main_eta_1_eta_2}), with \emph{the same}
%Lagrange multipliers $\la_{i,\ep}$ as those corresponding to
%$(\eta_1,\eta_2)$. Moreover, its components tend to zero at
%infinity. Therefore, by Proposition
%\ref{proUniqGeneral}, it coincides with
%$(\eta_1,\eta_2)$. In particular, it holds that
%$\|\eta_{i,\varepsilon}\|_{L^2(\mathbb{R}^2)}=1$, $i=1,2$.
%Consequently, by the proof of Theorem \ref{thm:uniqueness}, we infer
%that the pair $(\eta_{1,\varepsilon},\eta_{2,\varepsilon})$, as
%provided by Theorem \ref{thmTHM}, is the unique minimizer of
%$E^0_\ep$ in $\mathcal{H}^+$.

Estimate (\ref{eqthmLagrimproved}) is proven in Proposition
\ref{proLagrangeImproved}. Estimates
(\ref{eqthmUnifaway1})-(\ref{eqthmUnifaway2}) follow from
Proposition \ref{prounifAway}, the definition of
$\check{\eta}_{i,\varepsilon}$, and the second estimate in
(\ref{eqGSuUnif}). Estimates (\ref{eqthmAlg1})-(\ref{eqthmAlg2})
follow readily from Proposition \ref{prounifAlg}, the definition of
$\check{\eta}_{i,\varepsilon}$ (especially recall (\ref{eqeta2cut})
for the second estimate in (\ref{eqthmAlg1})), and (\ref{eqGSuAlg}).
Estimate (\ref{eqthmdecay1}) follows readily from Lemma
\ref{lemsuperExp}, (\ref{eqGSuUnif}) and (\ref{eqUniformOut});
estimate (\ref{eqthmdecay2}) follows from (\ref{eq:proGS2}),
(\ref{eqeta2cut}) and Lemma \ref{lemDecaytoetai}. Finally, relations
(\ref{eqthmInner1}), (\ref{eqthmInner2}) and (\ref{eqthmInner3}) are
consequences of (\ref{eq:proGS4}), (\ref{eqeta2cut}),
(\ref{eqUniformOut}) and Lemma \ref{lemDecaytoetai}.
\end{proof}

\subsection{Proof of Theorem \ref{thmMain}}
The desired  minimizer $(\eta_{1,\varepsilon},\eta_{2,\varepsilon})$
is that of Theorem \ref{thmLong}. Clearly, estimate (\ref{eq14Lagr})
is the same as (\ref{eqthmLagrimproved}). Estimate
(\ref{eq:main_theorem_relation2}) follows readily by combining
(\ref{eqa12}), (\ref{eqthmAlg1}), (\ref{eqthmAlg2}),
(\ref{eqthmdecay1}), (\ref{eqthmdecay2}), (\ref{eqthmInner2}),  and
(\ref{eqthmInner3}). Estimate (\ref{eq14complex}) follows readily
from (\ref{eqthmUnifaway1}), (\ref{eqthmUnifaway2}),
(\ref{eqthmAlg1}) and (\ref{eqthmAlg2}). In view of
(\ref{eq:a_i_def}), (\ref{eqthmUnifaway1}) and
(\ref{eqthmUnifaway2}), we infer that (\ref{eq14bootstrap}) holds.
Finally, the decay estimate (\ref{eq14decay}) follows immediately
from (\ref{eqthmdecay1}).

%%%%%%%%%%%%%%%%%%%%%%%%%%%%%%%%%%%%%%%%%%%%%%%%%%%%%%%%%%%%%%%%%%%%%%
%%%%%%%%%%%%%%%%%%%%%%%%%%%%%%%%%%%%%%%%%%%%%%%%%%%%%%%%%%%%%%%%%%%%%%
%%%%%%%%%%%%%%%%%%%%%%%%%%%%%%%%%%%%%%%%%%%%%%%%%%%%%%%%%%%%%%%%%%%%%%

\section{Estimates for the annulus case}\label{secann}
In this section, we explain how to extend the previous section to prove Theorem \ref{thmMainann}.
\subsection{Construction of an approximate solution}
\subsubsection{Outer approximations}
As before, we  work with the equivalent problem
(\ref{eqsystem}), where $a_{1,\varepsilon}$, $a_{2,\varepsilon}$ are
the same as in (\ref{eqa12}), and $\lambda_{1,\varepsilon}$,
$\lambda_{2,\varepsilon}$ are provided by Theorem
\ref{thm:uniqueness} in the case of (\ref{eqannulus}). This time,
the problem with both diffusion terms neglected  has a unique
continuous, nonnegative solution given by
\[
\begin{array}{lll}
  \eta_1=\left(a_{1,\varepsilon}+\frac{g}{g_1}a_{2,\varepsilon} \right)^\frac{1}{2}, & \eta_2=0,  & 0\leq r \leq R_{2,\varepsilon}^-, \\
    &   &   \\
  \eta_1=a_{1,\varepsilon}^\frac{1}{2}, & \eta_2=a_{2,\varepsilon}^\frac{1}{2}, & R_{2,\varepsilon}^-\leq r\leq R_{1,\varepsilon}, \\
    &   &   \\
  \eta_1=0,  & \eta_2=\left(a_{2,\varepsilon}+\frac{g}{g_2}a_{1,\varepsilon} \right)^\frac{1}{2}, &R_{1,\varepsilon} \leq r\leq R_{2,\varepsilon}^+, \\
    &   &   \\
  \eta_1=0, & \eta_2=0, & r\geq R_{2,\varepsilon}^+,
\end{array}
\]
where \[
(R_{2,\varepsilon}^-)^2=\frac{1}{\Gamma_1}\left(\lambda_{2,\varepsilon}-\frac{g}{g_1}\lambda_{1,\varepsilon}
\right),\ \ (R_{2,\varepsilon}^+)^2=\lambda_{2,\varepsilon},
\]
and
 \[
R_{1,\varepsilon}^2=\frac{1}{\Gamma_2}\left(\lambda_{1,\varepsilon}-\frac{g}{g_2}\lambda_{2,\varepsilon}
\right).
\]

In view of (\ref{eq:convergence_lagrange_multipliers}), and Remark \ref{rem3.4}, we have that
\begin{equation}\label{riest}
|R_{1,\varepsilon}-R_{1,0}|+|R_{2,\varepsilon}^\pm-R_{2,0}^\pm|\leq
C |\log \varepsilon|^\frac{1}{2}\varepsilon.
\end{equation}
\subsubsection{Inner approximations}
Here, we  define approximate solutions of the problem in
 overlapping intervals around each point
$R_{2,0}^-<R_{1,0}<R_{2,0}^+$.

On  $[0,R_{1,0}-\delta]$, where $\sqrt{a_1}$ is away from zero and
has bounded gradient,
%say when $|r-R_{2,\varepsilon}^-|\leq
%100 \delta$,
we neglect only the term $\varepsilon^2 \Delta \eta_{1,\varepsilon}$
from (\ref{eqsystem}), and get the following problem:
\[
\left\{\begin{array}{c}
  g_1\eta_1\left(\eta_1^2-a_{1,\varepsilon}(r) \right)+g\eta_1\left(\eta_2^2-a_{2,\varepsilon}(r) \right)=0, \\
   \\
-\varepsilon^2\Delta
\eta_2+g_2\eta_2\left(\eta_2^2-a_{2,\varepsilon}(r)
\right)+g\eta_2\left(\eta_1^2-a_{1,\varepsilon}(r) \right)=0.
\end{array}\right.
\]
From the first equation, we find that
\begin{equation}\label{eqeta1nearR2-}
\eta_1^2=a_{1,\varepsilon}+\frac{g}{g_1}(a_{2,\varepsilon}-\eta_2^2).
\end{equation}
Then, from the second equation, we obtain that
\[
-\varepsilon^2\Delta \eta_2+\left(g_2-\frac{g^2}{g_1}
\right)\eta_2(\eta_2^2-a_{2,\varepsilon})=0.
\]

The function $a_{2,\varepsilon}$ is  negative in
$[0,R_{2,\varepsilon}^-)$ and positive in $(R_{2,\varepsilon}^-,
\infty)$. We consider a function $A_{2,\varepsilon}$ which coincides
with $a_{2,\varepsilon}$ on $[0,R_{1,0}+\delta]$, changes sign once
in $(R_{1,0}+\delta,\infty)$, and diverges to $-\infty$ as $r\to
\infty$. We then take as an approximation for $\eta_2$ on
$[0,R_{1,0}-\delta]$ the restriction of the unique positive
 solution $\hat{\eta}_{2,\varepsilon}^-$ of the problem
\[
\varepsilon^2 \Delta \eta=\left(g_2-\frac{g^2}{g_1} \right)\eta
\left(\eta^2-A_{2,\varepsilon}(r) \right)\ \ \textrm{in}\
\mathbb{R}^2,\ \ \eta\to 0\ \textrm{as}\ r\to \infty.
\]
The properties of $\hat{\eta}_{2,\varepsilon}^-$ which we
require are contained in Appendix A. Accordingly, we take as an
approximation for $\eta_1$ on $[0,R_{1,0}-\delta]$ the one given by
(\ref{eqeta1nearR2-}) with $\hat{\eta}_{2,\varepsilon}^-$ in place
of $\eta_2$. The  approximations for
$\eta_{1,\varepsilon}$ and $\eta_{2,\varepsilon}$ on
$[R_{1,0}+\delta,R_{2,0}^+-\delta]$ are the same ones as in the
case of two disks, namely those given by (\ref{eqgroundstate1}) and
(\ref{eqeta2cut}) respectively.
 Analogously, if $r\geq R_{1,0}+\delta$, we take  as an approximation for $\eta_1$ the trivial solution, while for $\eta_2$ the unique positive
solution of the problem (\ref{eqgroundstate2}) which we now call
$\hat{\eta}_{2,\varepsilon}^+$.
\subsubsection{Gluing approximate solutions}
Let
\[
r_\varepsilon=\frac{R_{2,\varepsilon}^-+R_{1,\varepsilon}}{2},\ \
R_\varepsilon=\frac{R_{1,\varepsilon}+R^+_{2,\varepsilon}}{2}.
\]

Analogously to Subsection \ref{subApprox}, we can define a smooth
global approximate solution
$(\check{\eta}_{1,\varepsilon},\check{\eta}_{2,\varepsilon})$ such
that
\begin{equation}\label{eqeta1checkannulus}
\check{\eta}_{1,\varepsilon}=\left\{\begin{array}{ll}
                                      \left(a_{1,\varepsilon}+\frac{g}{g_1}a_{2,\varepsilon}-
                                      \frac{g}{g_1}(\hat{\eta}^-_{2,\varepsilon})^2
\right)^\frac{1}{2}, & 0\leq r\leq r_\varepsilon, \\
                                        &   \\
                                      a_{1,\varepsilon}^\frac{1}{2}+\mathcal{O}_{C^2}(\varepsilon^2), &
                                      r_\varepsilon\leq r \leq r_\varepsilon+\delta, \\
                                        &   \\
                                      \hat{\eta}_{1,\varepsilon}, & r_\varepsilon+\delta\leq
                                      r,
                                    \end{array}
 \right.
\end{equation}

\begin{equation}\label{eqeta2checkannulus}
\check{\eta}_{2,\varepsilon}=\left\{\begin{array}{ll}
                                      \hat{\eta}^-_{2,\varepsilon}, & 0\leq r\leq r_\varepsilon, \\
                                        &   \\
                                      a_{2,\varepsilon}^\frac{1}{2}+\mathcal{O}_{C^2}(\varepsilon^2), &
                                      r_\varepsilon\leq r \leq r_\varepsilon+\delta, \\
                                        &   \\
                                      \left(a_{2,\varepsilon}+\frac{g}{g_2}a_{1,\varepsilon}-
                                      \frac{g}{g_2}\hat{\eta}^2_{1,\varepsilon}
\right)^\frac{1}{2}, & r_\varepsilon+\delta\leq r\leq R_\varepsilon, \\
                                        &   \\
                                      \left(a_{2,\varepsilon}+\frac{g}{g_2}a_{1,\varepsilon}
\right)^\frac{1}{2}+\mathcal{O}_{C^2}(\varepsilon^2), & R_\varepsilon\leq r \leq R_\varepsilon+\delta, \\
                                        &   \\
                                      \hat{\eta}_{2,\varepsilon}^+, &
                                      R_\varepsilon+\delta\leq r.
                                    \end{array}
 \right.
\end{equation}
\subsection{Estimates for the error on the approximate solution}

The remainder $\mathcal{E}(\check{\eta}_{1,\varepsilon},
\check{\eta}_{2,\varepsilon})$ that is left when substituting the
approximate solution $(\check{\eta}_{1,\varepsilon},
\check{\eta}_{2,\varepsilon})$ to the system (\ref{eqsystem}) is as
in (\ref{eqError}).

For convenience, we  set
\begin{equation}\label{Aepsann}
A_\varepsilon=\{x\in \mathbb{R}^2\ :\
r_\varepsilon<|x|<R_\varepsilon \}.
\end{equation}
Analogously to Proposition \ref{proerror} for the case of two disks,
we have
\begin{proposition}\label{proerrorann}
The following estimates hold for small $\varepsilon>0$:
\[
\|E_1\|_{L^2(A_\varepsilon)}\leq C\varepsilon^2,\ \
\|E_2\|_{L^2(A_\varepsilon)}\leq C\varepsilon^\frac{5}{3},
\]
and
\[
\|E_1\|_{L^2(\mathbb{R}^2\setminus A_\varepsilon)}\leq
C\varepsilon^\frac{5}{3},\ \|E_2\|_{L^2(\mathbb{R}^2\setminus
A_\varepsilon)}\leq C\varepsilon^2.
\]
\end{proposition}
\subsection{Linear analysis}\label{subsectionLinearann}

 In  the sequel, we consider the linearization of (\ref{eqsystem}) about the approximate
solution
$(\check{\eta}_{1,\varepsilon},\check{\eta}_{2,\varepsilon})$,
namely the linear operator that is given by
(\ref{eqlinearization--}) for this choice of
$(\check{\eta}_{1,\varepsilon},\check{\eta}_{2,\varepsilon})$.

As in the case of two disks, using that
\[3\check{\eta}_{2,\varepsilon}^2-a_{2,\varepsilon}\geq \left\{\begin{array}{ll}
                                                                 c\max\{\varepsilon^\frac{2}{3},\check{\eta}_{2,\varepsilon}^2 \}, & |r
                                                                 -R_{2,\varepsilon}^-|\leq \delta, \\
                                                                   &   \\
                                                                  c,
                                                                  &r\in
                                                                  [0,R_{2,\varepsilon}^--\delta]\cup
                                                                  [R_{2,\varepsilon}^-+\delta,r_\varepsilon],
                                                                  \end{array}
 \right. \]
\[3\check{\eta}_{1,\varepsilon}^2-a_{1,\varepsilon}\geq \left\{\begin{array}{ll}
                                                                 c\max\{\varepsilon^\frac{2}{3},\check{\eta}_{1,\varepsilon}^2 \}, & |r
                                                                 -R_{1,\varepsilon}|\leq \delta, \\
                                                                   &   \\
                                                                  c, &
                                                                 r\in[r_\varepsilon,R_{1,\varepsilon}-\delta]\cup
                                                                 [R_{1,\varepsilon}+\delta,R_\varepsilon],
                                                                  \end{array}
 \right. \]
and
\[3\check{\eta}_{2,\varepsilon}^2-a_{2,\varepsilon}-\frac{g}{g_2}a_{1,\varepsilon}\geq \left\{\begin{array}{ll}
                                                                 c\max\{\varepsilon^\frac{2}{3},\check{\eta}_{2,\varepsilon}^2 \}, & |r
                                                                 -R_{2,\varepsilon}^+|\leq \delta, \\
                                                                   &   \\
                                                                  c, &
                                                                  r
                                                                  \in
                                                                  [R_\varepsilon,R_{2,\varepsilon}^+-\delta]\cup
                                                                  [R_{2,\varepsilon}^++\delta,\infty),
                                                               \end{array}
 \right. \]we can establish an analog of Proposition \ref{proL=f}.
\begin{proposition}\label{proL=fann}
The assertions of Proposition \ref{proL=f} are valid, provided that in
(\ref{eqapriori1}) and in the definition of the
$\vertiii{\cdot}$-norm in (\ref{eqNormtriple}),
$B_{R_\varepsilon}$ is replaced by $A_\varepsilon$ defined in (\ref{Aepsann}).
\end{proposition}
\subsection{Existence and properties of a positive solution to the system (\ref{eqsystem})}
As in Subsection \ref{subsecExPert}, using the properties of the
linearized operator that we discussed above, we construct a
positive, radial solution
$(\eta_{1,\varepsilon},\eta_{2,\varepsilon})$ to
(\ref{eqsystem}), near the approximate one
$(\check{\eta}_{1,\varepsilon},\check{\eta}_{2,\varepsilon})$, for
small $\varepsilon>0$. As before, the first part of the uniqueness
Theorem \ref{thm:uniqueness}  guarantees that this solution is
the desired minimizer.

Using the $\vertiii{\cdot}$-norm, as redefined in Proposition
\ref{proL=fann}, we can show that Propositions \ref{proT},
\ref{proExistSystem} and Corollary \ref{cor1} remain unchanged. We
still denote the corresponding solution by
$(\eta_{1,\varepsilon},\eta_{2,\varepsilon})$. The assertion of the
Lemma \ref{lemunifBlowUp} also remains the same. The only difference
in the proof is  that, say in the equation for $v_n$, we rearrange
the terms differently, namely write
\[
-\Delta v_n+g_1v_n^3-\left[g_1
a_{1,\varepsilon_n}(x_n+\varepsilon_n\mu_n y)+g
a_{2,\varepsilon_n}(x_n+\varepsilon_n \mu_n
y)\right]\mu_n^2v_n+g\mu_n^2
\eta_{2,\varepsilon_n}^2(x_n+\varepsilon_n \mu_n y) v_n=0,
\]
with\[ \|\eta_{2,\varepsilon}\|_{L^p(B_\delta)}\leq C_p
\varepsilon^{\frac{2}{3}+\frac{4}{3p}},\ p\geq 2.
\]
Then, the analog of Corollary \ref{corunifInter} is
\[
\|\eta_{1,\varepsilon}-\sqrt{a_{1,\varepsilon}+\frac{g}{g_1}a_{2,\varepsilon}}\|_{L^\infty(B_\delta)}+
\|\eta_{2,\varepsilon}\|_{L^\infty(B_\delta)} \leq
C\varepsilon^\frac{1}{3}.
\]
The positivity of the constructed solution, namely  the analog of
Proposition \ref{proPositive}, requires some additional
considerations, since $\eta_{2,\varepsilon}$ is also small in the
disk $|x|<R_{2,\varepsilon}^-$:
%%%%%%%%%%%%%%%%%%%%%%%%%%%%%%%%%%%%%%%%%%%%%%%%%%%%%%%%%%%%%%%%%%%%%%%%%%%%

\begin{proposition}\label{proPositiveann}
If $\varepsilon>0$ is sufficiently small, the constructed solutions  satisfy
\[
\eta_{i,\varepsilon}>0\ \ \textrm{in}\ \mathbb{R}^2,\ \ i=1,2.
\]
\end{proposition}
\begin{proof}
 The main difference with the previous case is in the domain
$|x|<R_{2,\varepsilon}^-$, which we describe below.

We know that
\begin{equation}\label{eq1annulus}
\eta_{1,\varepsilon}=\sqrt{\frac{\lambda_{1,\varepsilon}-r^2}{g_1}}+\mathcal{O}(\varepsilon^\frac{2}{3}),\
\ r\in [0,R_{2,\varepsilon}^--D\varepsilon^\frac{2}{3}],
\end{equation}
where  $\mathcal{O}(\varepsilon^\frac{2}{3})$ in dependent of $D>1$
(this follows directly from the analog of Proposition \ref{prounifAway} or from
the analogs of (\ref{eqeta2gluing}) and Corollary \ref{cor1}), and
\[
\eta_{2,\varepsilon}(R_{2,\varepsilon}^--D\varepsilon^\frac{2}{3})\geq
c \varepsilon^\frac{1}{3}>0.
\]
The function $\eta_{2,\varepsilon}$ satisfies the elliptic equation
\[
-\varepsilon^2 \Delta
\eta_{2,\varepsilon}+(r^2+g_2\eta_{2,\varepsilon}^2+g\eta_{1,\varepsilon}^2-\lambda_{2,\varepsilon})\eta_{2,\varepsilon}=0.
\]
In view of the above, the desired positivity of
$\eta_{2,\varepsilon}$  follows directly from the maximum
principle once we show that
\[
r^2+g_2\eta_{2,\varepsilon}^2+g\eta_{1,\varepsilon}^2-\lambda_{2,\varepsilon}>0,
\ \ r\in [0,R_{2,\varepsilon}^--D\varepsilon^\frac{2}{3}].
\]
Note that, thanks to (\ref{eq1annulus}), the left-hand side equals
\[
\Gamma_1
r^2+\frac{g}{g_1}\lambda_{1,\varepsilon}-\lambda_{2,\varepsilon}+g_2\eta_{2,\varepsilon}^2+\mathcal{O}(\varepsilon^\frac{2}{3}),
\]
where  $\mathcal{O}(\varepsilon^\frac{2}{3})$ in dependent of $D>1$.
In view of (\ref{riest}),  it suffices to show that
\begin{equation}\label{eq2annulus}
\Gamma_1 r^2+\frac{g}{g_1}\lambda_{1,0}-\lambda_{2,0}\geq
cD\varepsilon^\frac{2}{3},\ \ r\in
[0,R_{2,0}^--D\varepsilon^\frac{2}{3}],
\end{equation}
for some constant $c>0$ that is independent of $D,\varepsilon$,
provided that $D$ is sufficiently large and $\varepsilon$
sufficiently small. Observe that, since $r\leq
R_{2,0}^--D\varepsilon^\frac{2}{3}$, we have
\[
r^2\leq
(R_{2,\varepsilon}^-)^2+D^2\varepsilon^\frac{4}{3}-2DR_{2,0}^-\varepsilon^\frac{2}{3}\leq
(R_{2,\varepsilon}^-)^2-DR_{2,0}^-\varepsilon^\frac{2}{3},
\]
provided that $\varepsilon\leq \varepsilon(D)$. Now, recalling that
\[\Gamma_1<0,\]  we can bound the left-hand side of (\ref{eq2annulus})
from below  by
\begin{equation}\label{eq3annulus}
\Gamma_1 (R_{2,0}^-)^2+cD \varepsilon^\frac{2}{3}
+\frac{g}{g_1}\lambda_{1,0}-\lambda_{2,0}.
\end{equation} We use (\ref{lambdaiann})
 to find that the quantity (\ref{eq3annulus}) equals
$cD\varepsilon^\frac{2}{3}$ with $c=-\Gamma_1 R_{2,0}^->0$.
\end{proof}

\subsection{Proof of Theorem \ref{thmMainann}}

The proof for the case where (\ref{eqannulus}) holds, instead of
(\ref{eq:condition_on_g_two_disks}), proceeds along the same lines as the proof of Theorem \ref{thmMain}.
This time, we have to decompose $[0,\infty)$ into four intervals
with boundary points
${R}_{2,0}^-<{R}_{1,0}<{R}_{2,0}^+$. We point out
that the reduced problem near ${R}_{2,0}^-$ is a scalar
equation of the form (\ref{eqGSscalar}) where $a(r)<0$ for $r\in
[0,{R}_{2,0}^-)$, $a({R}_{2,0}^-)=0$, and $a(r_2)=0$ for
some $r_2>{R}_{2,0}^-$, which is covered in Theorem
\ref{thmGSscalar}.

%%%%%%%%%%%%%%%%%%%%%%%%%%%%%%%%%%%%%%%%%%%%%%%%%%%%%%%%%%%%%%%%%%%
%%%%%%%%%%%%%%%%%%%%%%%%%%%%%%%%%%%%%%%%%%%%%%%%%%%%%%%%%%%%%%%%%%%
%%%%%%%%%%%%%%%%%%%%%%%%%%%%%%%%%%%%%%%%%%%%%%%%%%%%%%%%%%%%%%%%%%%
%%%%%%%%%%%%%%%%%%%%%%%%%%%%%%%%%%%%%%%%%%%%%%%%%%%%%%%%%%%%%%%%%%%

\section{The auxiliary functions $F_{1,\varepsilon},\
F_{2,\varepsilon}$} \label{secAuxiliary} Assume that (\ref{eq:condition_on_g_thomas_fermi}),
 (\ref{eq:condition_on_g_two_disks}) and (\ref{eq14+}) hold. In this section, we
consider the auxiliary functions
\begin{equation}\label{eqFiAux}
F_{i,\varepsilon}(r)=\frac{\xi_{i,\varepsilon}(r)}{\eta_{i,\varepsilon}^2(r)},\
\ \textrm{with}\ \
\xi_{i,\varepsilon}(r)=\int_{r}^{\infty}s\eta_{i,\varepsilon}^2(s)ds,\
\ r\geq 0,\ \ i=1,2,
\end{equation}
  which  will play an important role when analyzing the
energy with rotation. In particular, we will link them to the
limiting functions
\begin{equation}\label{eqFiAux000}
F_{i,0}(r)=\left\{\begin{array}{ll}
              \frac{\xi_{i,0}(r)}{a_i(r)},& 0\leq r<R_{i,0}, \\
               &   \\
             0, & \textrm{otherwise},
           \end{array}
           \right.
\ \ \textrm{with}\ \ \xi_{i,0}(r)=\int_{r}^{\infty}s a_i(s)ds,
\end{equation}
where $a_i$ is as in (\ref{eq:a_i_def}). Note that $F_{i,0}$ is bounded
in $\R^2$ since $a_i>0$ for $r<R_{i,0}$, as observed in
\eqref{eq:bound_below_a_1}, \eqref{eq:bound_below_a_2}, and
$F_{i,0}(R_{i,0})=0$. Note also that $F_{i,0}$ is merely continuous
at $R_{i,0}$, as $F_{i,0}'$ has a finite jump discontinuity across
that point.

This section is devoted to proving the following.
\begin{proposition}
Assume that (\ref{eq:condition_on_g_thomas_fermi}) and
 (\ref{eq:condition_on_g_two_disks})  hold.
Let
$F_{i,\varepsilon}$ be given by  (\ref{eqFiAux})
 and $F_{i,0}$ by (\ref{eqFiAux000}).
Then
\[
F_{i,\varepsilon}(r)\leq \left\{\begin{array}{ll}
                              C(R_{i,0}-r)+C\varepsilon^\frac{2}{3}, & \textrm{if}\ \ 0\leq r \leq R_{i,0}, \\
                                &   \\
                              C\varepsilon^\frac{2}{3}, &
                              \textrm{if}\ \ r\geq R_{i,0},
                            \end{array}
 \right.
\]
and $\|F_{i,\varepsilon}-F_{i,0}\|_{L^\infty(\mathbb{R}^2)}\leq
C\varepsilon^\frac{1}{3}$, $i=1,2$, provided that $\varepsilon>0$ is
sufficiently small.
\end{proposition} This proposition follows from Corollary
\ref{lemma:estimates_f_i} and Lemma \ref{lemma:estimates_f_i-f_i0}.
The proof is made under the additional assumption (\ref{eq14+}). If $g_1 =g_2$, a simpler proof holds since $F_{1,\ep}=F_{2,\ep}$ and the property is that for a single equation \cite{AftalionJerrardLetelierJFA}. The scalar counterparts
\begin{equation}\label{eqfiAux(scalar)}
f_{i,\varepsilon}(r)=\frac{1}{\hat{\eta}_{i,\varepsilon}^2(r)}\int_{r}^{\infty}s\hat{\eta}_{i,\varepsilon}^2(s)ds,\
\ r\geq 0,\ \ i=1,2,
\end{equation}
and their convergence to the corresponding limiting functions
\begin{equation}\label{eqfiAux(scalar)000}
f_{1,0}(r)=\frac{\int_{r}^{\infty}sa_{1,0}^+(s)ds}{a_{1,0}^+(r)}\ \
\textrm{and}\ \
f_{2,0}(r)=\frac{\int_{r}^{\infty}s\left(a_{2,0}+\frac{g}{g_2}a_{1,0}\right)^+(s)ds}{\left(a_{2,0}+\frac{g}{g_2}a_{1,0}\right)^+(r)},
\end{equation}
 have
been studied in \cite[Lem. 2.2]{AftalionJerrardLetelierJFA}. We have \begin{equation}\label{eqFfequal} F_{1,0}\equiv f_{1,0},\
\ \textrm{and}\ \ F_{2,0}\equiv f_{2,0}\ \ \textrm{\emph{only}\
on}\ r\geq R_{1,0}.
\end{equation}
Actually, the ground states  in the latter lemma had unit $L^2$-norm
but its proof carries over  to the above case, yielding the
following lemma.
\begin{lemma}\label{lemGSaux}
Suppose that $u_\varepsilon$ is as in Proposition \ref{proGS} with
$|\lambda_\varepsilon-\lambda_0|\leq C|\log
\varepsilon|^\frac{1}{2}\varepsilon$.
 The auxiliary functions
\[
\xi_\varepsilon(r)=\int_{r}^{\infty}su_\varepsilon^2(s)ds\ \
\textrm{and}\ \
f_\varepsilon(r)=\frac{\xi_\varepsilon(r)}{u_\varepsilon^2(r)},\ \
r\geq 0.
\]
satisfy
\[
f_\varepsilon(r)\leq \left\{\begin{array}{ll}    C(r_0-r)+C\varepsilon^\frac{2}{3}, & \textrm{if}\ \ 0\leq r \leq r_0, \\
                                &   \\    C\varepsilon^\frac{2}{3}, &  \textrm{if}\ \ r\geq r_0,  \end{array}
 \right.\]
and $\|f_\varepsilon-f_0\|_{L^\infty(\mathbb{R}^2)}\leq
C\varepsilon^\frac{1}{3}$, where
\[
f_0(r)=\left\{\begin{array}{ll}
                \frac{1}{A_0(r)}\int_{r}^{r_0}sA_0(s)ds, & \textrm{if}\ \ r<r_0, \\
                  &   \\
                 0, & \textrm{if}\ \ r\geq r_0,
              \end{array}
 \right.
\]
provided that $\varepsilon>0$ is sufficiently small.
\end{lemma}
The main task in this section  is to show the
following proposition.
\begin{proposition}\label{proFi}
If $\varepsilon>0$ is sufficiently small, then
\begin{equation}\label{eqF1-f1}
\left|F_{1,\varepsilon}(r)-f_{1,\varepsilon}(r)\right|\leq
C\varepsilon^\frac{2}{3},\ \ r\geq 0,
\end{equation}
and
\begin{equation}\label{eqF2-f2}
\left|F_{2,\varepsilon}(r)-f_{2,\varepsilon}(r)\right|\leq
C\varepsilon^\frac{2}{3},\ \ r\geq R_{1,\varepsilon}.
\end{equation}
\end{proposition}
\begin{proof}
From now on, let us fix a large $D>1$ such that Lemma
\ref{lemsuperExp} is valid. The latter lemma, similarly to
\cite[Lem. 2.2]{AftalionJerrardLetelierJFA}, implies that
\begin{equation}\label{eqF1Out}
0<F_{1,\varepsilon}(r)\leq C\varepsilon^\frac{2}{3},\ \ r\geq
R_{1,\varepsilon}+D\varepsilon^\frac{2}{3}.
\end{equation}
Since the above estimate also holds  for $f_{1,\varepsilon}$, by
virtue of Lemma \ref{lemGSaux}, we infer that (\ref{eqF1-f1}) is
valid for $r\geq R_{1,\varepsilon}+D\varepsilon^\frac{2}{3}$.

If $r\leq R_{1,\varepsilon}+D\varepsilon^\frac{2}{3}$, via
(\ref{eqGSuUnif}), \eqref{eq:proGS4},
 and Corollaries \ref{cor1},
\ref{corunifInter}, we have
\[
\frac{\eta_{1,\varepsilon}^2
(R_{1,\varepsilon}+D\varepsilon^\frac{2}{3})}{\eta_{1,\varepsilon}^2(r)}\leq
C.
\]
 Therefore, we can write
\begin{equation}\label{eqF1basic}
\begin{array}{rcl}
  F_{1,\varepsilon}(r) & = & \frac{1}{\eta_{1,\varepsilon}^2(r)}\int_{r}^{R_{1,\varepsilon}+D\varepsilon^\frac{2}{3}}s\eta_{1,\varepsilon}^2(s)ds
+ \frac{\eta_{1,\varepsilon}^2
(R_{1,\varepsilon}+D\varepsilon^\frac{2}{3})}{\eta_{1,\varepsilon}^2(r)}F_{1,\varepsilon}(R_{1,\varepsilon}+D\varepsilon^\frac{2}{3}) \\
   &  &  \\
   & \stackrel{(\ref{eqF1Out})}{=} & \frac{1}{\eta_{1,\varepsilon}^2(r)}\int_{r}^{R_{1,\varepsilon}
+D\varepsilon^\frac{2}{3}}s\eta_{1,\varepsilon}^2(s)ds+\mathcal{O}(\varepsilon^\frac{2}{3}),
\end{array}
\end{equation}
uniformly in $r\geq 0$, as $\varepsilon\to 0$. After rearranging
terms, we find that
\[\begin{array}{rcl}
F_{1,\varepsilon}(r)-f_{1,\varepsilon}(r)&=&
  \frac{1}{\eta_{1,\varepsilon}^2(r)}\int_{r}^{R_{1,\varepsilon}
+D\varepsilon^\frac{2}{3}}s\left[\eta_{1,\varepsilon}^2(s)-\hat{\eta}_{1,\varepsilon}^2(s)\right]ds
+\frac{\hat{\eta}_{1,\varepsilon}^2(r)-\eta_{1,\varepsilon}^2(r)}{\eta_{1,\varepsilon}^2(r)\hat{\eta}_{1,\varepsilon}^2(r)}\int_{r}^{R_{1,\varepsilon}
+D\varepsilon^\frac{2}{3}}s\hat{\eta}_{1,\varepsilon}^2(s)ds+ \\
  &&  \\
 && +\mathcal{O}(\varepsilon^\frac{2}{3}),
\end{array}
\]
uniformly in $r\leq R_{1,\varepsilon} +D\varepsilon^\frac{2}{3}$, as
$\varepsilon\to 0$. Since on this interval we can set
\[
\varphi=\eta_{1,\varepsilon}-\check{\eta}_{1,\varepsilon}=\eta_{1,\varepsilon}-\hat{\eta}_{1,\varepsilon},
\]
we obtain that
\begin{equation}\label{eqF1-f1Phi}
\begin{array}{rcl}
  F_{1,\varepsilon}(r)-f_{1,\varepsilon}(r) & = & \frac{1}{\eta_{1,\varepsilon}^2(r)}\int_{r}^{R_{1,\varepsilon}
+D\varepsilon^\frac{2}{3}}s(\varphi^2+2\hat{\eta}_{1,\varepsilon}\varphi)ds- \\
   &  &  \\
    &  & -\left[
\frac{\varphi^2(r)}{\eta_{1,\varepsilon}^2(r)\hat{\eta}_{1,\varepsilon}^2(r)}
+\frac{2\varphi(r)}{\eta_{1,\varepsilon}^2(r)\hat{\eta}_{1,\varepsilon}(r)}
\right]\int_{r}^{R_{1,\varepsilon}
+D\varepsilon^\frac{2}{3}}s\hat{\eta}_{1,\varepsilon}^2(s)ds+\mathcal{O}(\varepsilon^\frac{2}{3}),
\end{array}
\end{equation}
uniformly in $r\leq R_{1,\varepsilon} +D\varepsilon^\frac{2}{3}$, as
$\varepsilon\to 0$. The above terms can be estimated by first
decomposing the interval $[0,R_{1,\varepsilon}
+D\varepsilon^\frac{2}{3}]$ as $[0,R_{1,\varepsilon}
-\delta]\cup[R_{1,\varepsilon} -\delta,R_{1,\varepsilon}
-\varepsilon^\frac{1}{3}]\cup[R_{1,\varepsilon}
-\varepsilon^\frac{1}{3},R_{1,\varepsilon}
+D\varepsilon^\frac{2}{3}]$ (with $\delta>0$ fixed small), then
making use of the uniform estimates in (\ref{eqUniformOut}) and
Proposition \ref{prounifAway} for $\varphi$, and those in
(\ref{eq:proGS4}) and (\ref{eqGSuAlg}) for
$\hat{\eta}_{1,\varepsilon}$. To illustrate the procedure, let us
estimate in detail the term
\[
\frac{1}{\eta_{1,\varepsilon}^2(r)}\int_{r}^{R_{1,\varepsilon}
+D\varepsilon^\frac{2}{3}}s\hat{\eta}_{1,\varepsilon}(s)\varphi(s)
ds.
\]
If $R_{1,\varepsilon} -\varepsilon^\frac{1}{3}\leq r\leq s \leq
R_{1,\varepsilon} +D\varepsilon^\frac{2}{3}$, since
(\ref{eq:proGS4}) and (\ref{eqUniformOut}) imply that
$\eta_{1,\varepsilon}(r)\geq c\varepsilon^\frac{1}{3}$ and
$\hat{\eta}_{1,\varepsilon}(s)\leq C\varepsilon^\frac{1}{6}$, using
(\ref{eqUniformOut}) to bound $\varphi$, we deduce that
\[
\left|\frac{1}{\eta_{1,\varepsilon}^2(r)}\int_{r}^{R_{1,\varepsilon}
+D\varepsilon^\frac{2}{3}}s\hat{\eta}_{1,\varepsilon}(s)\varphi(s)
ds\right|\leq C\varepsilon^{-\frac{2}{3}}(R_{1,\varepsilon}
+D\varepsilon^\frac{2}{3}-r)\varepsilon^\frac{1}{6}\varepsilon\leq
C\varepsilon^{-\frac{2}{3}}\varepsilon^\frac{1}{3}\varepsilon^\frac{1}{6}\varepsilon=C
\varepsilon^{\frac{2}{3}+\frac{1}{6}}.\] If $R_{1,\varepsilon}
-\delta\leq r\leq s \leq R_{1,\varepsilon}
-\varepsilon^\frac{1}{3}$, arguing similarly, this time noting that
$\eta_{1,\varepsilon}(r)\geq c\varepsilon^\frac{1}{6}$, we find that
\[\left|
\frac{1}{\eta_{1,\varepsilon}^2(r)}\int_{r}^{R_{1,\varepsilon}
+D\varepsilon^\frac{2}{3}}s\hat{\eta}_{1,\varepsilon}(s)\varphi(s)
ds\right|\leq
C\varepsilon^{-\frac{1}{3}}\varepsilon=C\varepsilon^\frac{2}{3}.
\]
Lastly, if $0\leq r\leq s \leq R_{1,\varepsilon}-\delta$, where
$\eta_{1,\varepsilon}\geq c$, via Proposition \ref{prounifAway}, we
get that
\[
\left|\frac{1}{\eta_{1,\varepsilon}^2(r)}\int_{r}^{R_{1,\varepsilon}
+D\varepsilon^\frac{2}{3}}s\hat{\eta}_{1,\varepsilon}(s)\varphi(s)
ds\right|\leq C\varepsilon^\frac{2}{3}.
\]
The remaining terms in (\ref{eqF1-f1Phi}) can be estimated
analogously to complete the proof of (\ref{eqF1-f1}). We point out
that a rather delicate term is
\[
\frac{\varphi(r)}{\eta_{1,\varepsilon}^2(r)\hat{\eta}_{1,\varepsilon}(r)}
\int_{r}^{R_{1,\varepsilon}
+D\varepsilon^\frac{2}{3}}s\hat{\eta}_{1,\varepsilon}^2(s)ds
\]
when $r\in [R_{1,\varepsilon}-\delta,
R_{1,\varepsilon}-\varepsilon^\frac{1}{3} ]$, which can be estimated
as follows: Since in this interval we have
$\hat{\eta}_{1,\varepsilon}(r)\geq c
(R_{1,\varepsilon}-r)^\frac{1}{2}\geq C \ep^{1/6}$ (from
(\ref{eqGSuAlg}) and (\ref{eqnondegeneracy})), and the same holds
for $\eta_{1,\ep}$ via (\ref{eqUniformOut}), it follows that
\[
\left|
\frac{\varphi(r)}{\eta_{1,\varepsilon}^2(r)\hat{\eta}_{1,\varepsilon}(r)}
\int_{r}^{R_{1,\varepsilon}
+D\varepsilon^\frac{2}{3}}s\hat{\eta}_{1,\varepsilon}^2(s)ds\right|
\leq
C\frac{\varepsilon\varepsilon^{-\frac{1}{6}}}{{\eta_{1,\varepsilon}^2(r)}}
\left(|r-R_{1,\varepsilon}|+\varepsilon^\frac{2}{3} \right)\leq
C\varepsilon^\frac{5}{6}.
\]

The validity of estimate (\ref{eqF2-f2}) can be verified
analogously, using (\ref{eqeta2cut}) to show that
$|\eta_{2,\varepsilon}-\hat{\eta}_{2,\varepsilon}|\leq
C\varepsilon^\frac{2}{3}$ in
$[R_{1,\varepsilon},R_{2,\varepsilon}-\delta]$.\end{proof}

The assertion of the following corollary is analogous to the first
assertion of Lemma \ref{lemGSaux} for the scalar case.

\begin{cor}\label{lemma:estimates_f_i}
If $\varepsilon>0$ is sufficiently small, we have
\begin{equation}\label{eqFiRequired}
0<F_{i,\varepsilon}(r)\leq \left\{\begin{array}{ll}
                                    C(R_{i,0}-r) & \textrm{if}\ \ 0\leq r \leq R_{i,0}-\varepsilon^\frac{2}{3}, \\
                                      &   \\
                                    C\varepsilon^\frac{2}{3} & \textrm{if}\ \ r \geq R_{i,0}-\varepsilon^\frac{2}{3},
                                  \end{array}
 \right.
\end{equation}
$i=1,2$, where $\delta>0$ is independent of $\varepsilon$ such that
$R_{2,0}-R_{1,0}>4\delta$ and $R_{1,0}>4\delta$.
\end{cor}
\begin{proof}
The desired estimate (\ref{eqFiRequired}) follows readily from the
fact that it holds with $f_{i,\varepsilon}$ in place of
$F_{i,\varepsilon}$ (see Lemma \ref{lemGSaux}), via Theorem
\ref{thmLong} and Proposition \ref{proFi}.
\end{proof}

The next lemma is a natural extension of the second assertion of
Lemma \ref{lemGSaux}.

\begin{lemma}\label{lemma:estimates_f_i-f_i0}
If $\varepsilon>0$ is sufficiently small, we have
\[
\|F_{i,\varepsilon}-F_{i,0}\|_{L^\infty(\mathbb{R}^2)}\leq
C\varepsilon^\frac{1}{3},\ \ i=1,2,
\]
where $F_{i,0}$ are as in (\ref{eqFiAux000}).
\end{lemma}
\begin{proof}
The proof is based on the fact that
\[
\|f_{i,\varepsilon}-f_{i,0}\|_{L^\infty(\mathbb{R}^2)}\leq
C\varepsilon^\frac{1}{3},\ \ i=1,2, \ \ (\textrm{see\ Lemma
\ref{lemGSaux}}),
\]
where $f_{i,0}$ are as in (\ref{eqfiAux(scalar)000}). In view of
Proposition \ref{proFi} and (\ref{eqFfequal}), we infer that the
assertion of the lemma is valid for $i=1$ and that there exists some
$C>0$ such that
\begin{equation}\label{eqF2at}
\left|F_{2,\varepsilon}(r)-F_{2,0}(r) \right|\leq
C\varepsilon^\frac{1}{3},\ \ r \geq
R_{1,\varepsilon}+\varepsilon^\frac{2}{3},
\end{equation}
(recall also (\ref{eqRi-R0})). So, for the proof to be completed, it
remains to show that there exists some $C>0$ such that
\begin{equation}\label{eqFonly}
\left|F_{2,\varepsilon}(r)-F_{2,0}(r) \right|\leq
C\varepsilon^\frac{1}{3},\ \ 0\leq r \leq
R_{1,\varepsilon}+\varepsilon^\frac{2}{3}.
\end{equation}
To this end, for $0\leq r\leq
R_{1,\varepsilon}+\varepsilon^\frac{2}{3}$, we write
\[
F_{2,\varepsilon}(r)=\frac{1}{\eta_{2,\varepsilon}^2(r)}\int_{r}^{R_{1,\varepsilon}+\varepsilon^\frac{2}{3}}
s\eta_{2,\varepsilon}^2(s)ds+\frac{\eta_{2,\varepsilon}^2(R_{1,\varepsilon}+\varepsilon^\frac{2}{3})}{\eta_{2,\varepsilon}^2(r)}
F_{2,\varepsilon}(R_{1,\varepsilon}+\varepsilon^\frac{2}{3}),
\]
and
\[
F_{2,0}(r)=\frac{1}{\eta_{2,0}^2(r)}\int_{r}^{R_{1,\varepsilon}+\varepsilon^\frac{2}{3}}
s\eta_{2,0}^2(s)ds+\frac{\eta_{2,0}^2(R_{1,\varepsilon}+\varepsilon^\frac{2}{3})}{\eta_{2,0}^2(r)}
F_{2,0}(R_{1,\varepsilon}+\varepsilon^\frac{2}{3}).
\]
Now, estimate (\ref{eqFonly}) follows readily from (\ref{eqF2at})
and the property that
\[
\left|\eta_{2,\varepsilon}(r)-\eta_{2,0}(r) \right|\leq
C\varepsilon^\frac{2}{3},\ \ 0\leq r \leq
R_{1,\varepsilon}+\varepsilon^\frac{2}{3}.
\]
The latter estimate is a consequence of
(\ref{eqeta1check})-(\ref{eqeta2cut}), the fact that
$\|\hat{\eta}_{1,\varepsilon}^2-a_{1,\varepsilon}^+\|_{L^\infty(\mathbb{R}^2)}\leq
C\varepsilon^\frac{2}{3}$ (see\ Proposition \ref{proGS}), Corollary
\ref{cor1} and Proposition \ref{prounifAway}.\end{proof}

%For completeness,  we state the following lemma  which can be
%verified by direct calculations.
%\begin{lemma}\label{lemma:xi_f_explicit}
%We have
%\[
%\xi_{1,0}(r)=\frac{\Gamma_2}{4g_1\Gamma}(R_{1,0}^2-r^2)^2, \qquad
%f_{1,0}(r)=\frac{R_{1,0}^2-r^2}{4}
%\]
%for $r\leq R_{1,0}$ and both functions are $0$ for $r\geq R_{1,0}$,
%while
%\[
%\xi_{2,0}(r)=\frac{1}{4g_2}(R_{2,0}^2-r^2)^2, \qquad
%f_{2,0}(r)=\frac{R_{2,0}^2-r^2}{4}
%\]
%for $R_{1,0}\leq r\leq R_{2,0}$ and both functions are $0$ for
%$r\geq R_{2,0}$.
%\end{lemma}

Finally, we have another estimate which will be used later.
\begin{lemma}\label{lemma:nabla_xi_estimate}
There exists $C>0$ such that $\|\nabla
\xi_{i,\varepsilon}\|_{L^\infty(\R^2)}\leq C$, $i=1,2$.
\end{lemma}
\begin{proof}
We have
\begin{equation}\label{eq:xi'}
\xi_{i,\varepsilon}'(r)=-r\eta_{i,\varepsilon}^2(r)
\end{equation}
so that the result follows from Lemma \ref{lemma:decay_for_large_x}.
\end{proof}

\section{The energy minimizer with rotation}\label{secRotation}

In this section, we study the behavior of the minimizers of the
energy functional $E^\Omega_\ep$ in the space $\mathcal{H}$, defined
in (\ref{eqEnergyOmega}) and  \eqref{H_space_definition}
respectively, as $\varepsilon\to 0$. In the following we assume
\begin{equation}\label{eq:Omega_condition1}
\Omega\leq C|\log\ep|,
\end{equation}
for some constant $C$ independent of $\ep$. Any minimizer
$(u_1,u_2)=(u_{1,\ep},u_{2,\ep})$ of  $E^\Omega_\ep$ in
$\mathcal{H}$ solves the following system
\begin{equation*}
\left\{ \begin{array}{ll}
-\ep^2 \Delta u_1+u_1(|x|^2+g_1|u_1|^2+g|u_2|^2)+2\ep^2i\Omega x^\perp\cdot\nabla u_1=\mu_{1,\ep}u_1 \quad & \text{ in } \R^2, \\
-\ep^2 \Delta u_2+u_2(|x|^2+g_2|u_2|^2+g|u_1|^2)+2\ep^2i\Omega
x^\perp\cdot\nabla u_2=\mu_{2,\ep}u_2  \quad & \text{ in } \R^2,
\end{array}\right.
\end{equation*}
for some Lagrange multipliers $\mu_{1,\ep},\mu_{2,\ep}$. The
existence of a minimizer when $\Omega$ satisfies
\eqref{eq:Omega_condition1} is a consequence of Lemma
\ref{lemma:rotation_term_bound} below and of the compactness induced
by the fact that the harmonic potential $|x|^2$ diverges as $|x|\to
\infty$.

\subsection{Energy estimates}

The following proof uses some ideas from \cite[Lem.
3.1]{IgnatMillotJFA}.

\begin{lemma}\label{lemma:rotation_term_bound}
We have
\[
\begin{split}
\Omega\sum_{j=1}^2 \left|\int_{\R^2} x^\perp \cdot (iu_j,\nabla u_j)
\,dx \right| \leq
\sum_{j=1}^2 \int_{\R^2}\frac{|\nabla u_j|^2}{4}\,dx +2\Omega^2 (R_{1,0}^2+R_{2,0}^2)+ \\
\frac{2\Omega^2g_1\Gamma}{\Gamma_2}\int_{\R^2\setminus
D_1}a_{1,0}^-|u_1|^2\,dx+ 2\Omega^2g_2\int_{\R^2\setminus D_2}
\left(a_{2,0}+\frac{g}{g_2}a_{1,0}\right)^-|u_2|^2 \,dx.
\end{split}
\]
\end{lemma}
\begin{proof}
We have
\[
\left|\Omega \int_{\R^2} x^\perp \cdot (iu_i,\nabla u_i) \,dx
\right| \leq \int_{\R^2}\left(\frac{|\nabla
u_i|^2}{4}+\Omega^2|x|^2|u_i|^2 \right)\,dx .
\]
We need to estimate the second term in the right hand side. Let us
start with $i=1$. Notice that
\[
|x|^2\leq
-\frac{2g_1\Gamma}{\Gamma_2}a_{1,0}(x)=\frac{2g_1\Gamma}{\Gamma_2}a_{1,0}^-(x)
\qquad \text{for } |x|\geq \sqrt{2}R_{1,0}.
\]
This implies
\[
\begin{split}
\int_{\R^2} |x|^2|u_1|^2 \,dx & = \int_{\{|x|\leq\sqrt{2}R_{1,0}\}}|x|^2|u_1|^2\,dx + \int_{\{|x|>\sqrt{2}R_{1,0}\}}|x|^2|u_1|^2\,dx \\
& \leq 2R_{1,0}^2 + \frac{2g_1\Gamma}{\Gamma_2}\int_{\R^2\setminus
D_1}a_{1,0}^-|u_1|^2\,dx.
\end{split}
\]
Similarly, for $i=2$, we have
\[
|x|^2\leq 2g_2\left(a_{2,0}+\frac{g}{g_2}a_{1,0}\right)^-(x) \qquad
\text{for } |x|\geq \sqrt{2}R_{2,0},
\]
so that
\[
\int_{\R^2} |x|^2|u_2|^2 \,dx \leq 2R_{2,0}^2 + 2g_2
\int_{\R^2\setminus D_2}
\left(a_{2,0}+\frac{g}{g_2}a_{1,0}\right)^-|u_2|^2\,dx,
\]
which concludes the proof of the lemma.
\end{proof}
In analogy to Proposition \ref{prop:energy_estimates}, we have the
following lemma.

\begin{lemma}\label{lemma:energy_estimates_u}
Let $(u_{1},u_{2})$ be a minimizer of $E^\Omega_\ep$ in
$\mathcal{H}$. There exists $C>0$ independent of $\ep$ such that,
for $i=1,2$, we have
\[
\int_{\R^2} |\nabla u_i|^2\,dx \leq C|\log\ep|^2,
\]
\[
\int_{\R^2} \left(|u_i|^2-a_i\right)^2\,dx \leq C \ep^2|\log\ep|^2,
\]
\[
\int_{\R^2\setminus D_1}|u_1|^2a_{1,0}^-\,dx+\int_{\R^2\setminus
D_2}(|u_1|^2+|u_2|^2)\left(a_{2,0}+\frac{g}{g_2}a_{1,0}\right)^-\,dx\leq
C\ep^2|\log\ep|^2.
\]
\end{lemma}
\begin{proof}
On the one hand,  by the definition of minimizers, by Lemma
\ref{lemma:E_tilde_def} and Proposition \ref{prop:energy_estimates}, we have
\begin{equation}\label{eq:energy_estimates_u}
E_\ep^\Omega(u_1,u_2)\leq
E_\ep^\Omega(\eta_1,\eta_2)=E_\ep^0(\eta_1,\eta_2)\leq C|\log\ep|+K,
\end{equation}
with $K$ as in Lemma \ref{lemma:E_tilde_def}. On the other hand, we have
\[
E_\ep^\Omega(u_1,u_2)= \tilde{E}_\ep^0(u_1,u_2)+K-\Omega
\sum_{j=1}^2 \int_{\R^2} x^\perp \cdot (iu_j,\nabla u_j) \,dx .
\]
The right hand side can be bounded from below by means of
\eqref{eq:g_mathfrak_definition} and of Lemma
\ref{lemma:rotation_term_bound} as follows:
\[
\begin{split}
E_\ep^\Omega(u_1,u_2)\geq \sum_{i=1}^2 \int_{\R^2}\left\{ \frac{|\nabla u_i|^2}{4}+\frac{\gamma}{4\ep^2}(|u_i|^2-a_i)^2 \right\}\,dx \\
+g_1\Gamma\left(\frac{1}{2\ep^2}-\frac{2\Omega^2}{\Gamma_2}\right)\int_{\R^2\setminus
D_1} a_{1,0}^-|u_1|^2\,dx
+\frac{g}{2\ep^2}\int_{\R^2\setminus D_2}|u_1|^2\left(a_{2,0}+\frac{g}{g_2}a_{1,0}\right)^-\,dx \\
+g_2\left(\frac{1}{2\ep^2}-2\Omega^2\right)\int_{\R^2\setminus D_2}
\left(a_{2,0}+\frac{g}{g_2}a_{1,0}\right)^-|u_2|^2\,dx
-2\Omega^2(R_{1,0}^2+R_{2,0}^2)+K.
\end{split}
\]
The result follows by combining the last inequality with
\eqref{eq:energy_estimates_u} and by using
\eqref{eq:Omega_condition1}.
\end{proof}
In analogy to Lemma \ref{lemma:L_infty_bounds}, we have the
following

\begin{lemma}\label{lemma:L_infty_bounds_u}
Let $(u_1,u_2)$ be a minimizer of $E^\Omega_\ep$ in $\mathcal{H}$.
For $\ep$ sufficiently small, we have
\[
|u_i|^2\leq \mu_{i,\ep}/g_i, \qquad \|\nabla
u_i\|_{L^\infty(\R^2)}\leq
\frac{C\sqrt{\mu_{i,\ep}}(\mu_{i,\ep}+\mu_{j,\ep}+1)+C}{\ep}
\]
for some $C>0$, $i=1,2$, $j\neq i$.
\end{lemma}
\begin{proof}
For $\ep$ sufficiently small, the following holds
\[
\Delta(|u_j|^2)=2|\nabla u_j|^2+2(u_j,\Delta u_j) \geq
\frac{2}{\ep^2} |u_j|^2(g_j|u_j|^2-|\mu_{j,\ep}|),
\]
where we use $-2\Omega x^\perp\cdot(iu_j,\nabla u_j)\geq
-\Omega^2|x|^2|u_j|^2-|\nabla u_j|^2$ and condition
\eqref{eq:Omega_condition1}. We can proceed very similarly to Lemma
\ref{lemma:L_infty_bounds}: let
\[
w_i=\frac{g_i|u_i|^2-|\mu_{i,\ep}|}{\ep^2} \quad \text{ we have }
\quad \Delta w_i^+\geq 2(w_i^+)^2
\]
so that we conclude again with the non-existence result by Brezis
\cite{brezisLiouville}. Note that by testing the equation of $u_i$
by $u_i$ itself, and working as above, yields that
$\mu_{i,\varepsilon}>0$.

To prove the second part, fix $x\in\R^2$, $L>0$ and for $y\in
B_{2L}(x)$, let $z_i(y)=u_i(\ep(y-x))$. Then
\[
-\Delta z_i=-z_i(\ep^2|y-x|^2+g_i |z_i|^2+g
|z_j|^{2}-\mu_{i,\ep})-2\ep^2i\Omega (y-x)^\perp\cdot \nabla  z_i
=:h_{i,\ep}(y).
\]
We have, by Lemma \ref{lemma:energy_estimates_u} and by
\eqref{eq:Omega_condition1},
\[
\ep^2\Omega \|(y-x)^\perp\cdot \nabla  z_i\|_{L^2(B_{2L}(x))} =
\ep\Omega \|x^\perp\cdot\nabla u_i\|_{L^2(B_{2\ep L}(0))} \leq C
\]
for a constant $C$ independent of $x$. Therefore, using also the
$L^\infty$-bound above, we have $\|h_{i,\ep}\|_{L^2(B_{2L}(x))}\leq
C\sqrt{\mu_{i,\ep}}(\mu_{i,\ep}+\mu_{j,\ep}+1)+C$. We deduce that
$\|z_i\|_{H^2(B_{L}(x))}\leq
C\sqrt{\mu_{i,\ep}}(\mu_{i,\ep}+\mu_{j,\ep}+1)+C$ and we conclude by
a bootstrap argument.
\end{proof}

\begin{lemma}\label{lemma:lagrange_multipliers_mu}
Let $(u_{1},u_{2})$ be a minimizer of $E^\Omega_\ep$ in
$\mathcal{H}$ and denote by $\mu_{i,\ep}$ the associated Lagrange
multipliers. There exists $C>0$ independent of $\ep$ such that, for
$i=1,2$,
\[
|\mu_{i,\ep}|\leq C.
\]
\end{lemma}
\begin{proof}
 We test the equation for $u_i$ by $u_i$ itself and integrate by parts, which is possible since $u_i\in H^1(\mathbb{R}^2,\mathbb{C})$.
%The boundary term vanishes because
%\[
%\left|\int_{\partial B_R} u_i\nabla u_i\cdot \nu \,d\sigma\right|
%\leq C \frac{1+\mu_{i,\ep}}{\ep} \int_{\partial B_R} |u_i| \,d\sigma
%\to 0 \quad \text{as } R\to\infty.
%\]
The term containing $\Omega$ can be bounded by means of Lemma
\ref{lemma:rotation_term_bound}, whereas the other terms can be
rewritten as in Proposition
\ref{prop:convergence_lagrange_multipliers}. Finally, the desired
bound follows from the energy estimates of Lemma
\ref{lemma:energy_estimates_u}.
\end{proof}

\subsection{Non-existence of vortices}

The proof presented here is an adaptation of the proof of the main
theorem in \cite{AftalionJerrardLetelierJFA}. Let us start with the
following splitting of the energy, which is introduced in
\cite{AftalionMasonWei}.

\begin{lemma}\label{lemma:splitting_energy}
Let $(u_1,u_2)$ be any minimizer of  $E^\Omega_\ep$ in $\mathcal{H}$
and let $(\eta_1,\eta_2)$ be the unique positive minimizer of $E_\ep^0$ in
$\mathcal{H}$ provided by Theorem \ref{thm:uniqueness}. Let
\[
v_i=\frac{u_i}{\eta_i}, \quad \text{for } i=1,2.
\]
Then
\[
E^\Omega_\ep(u_1,u_2)=E^0_\ep(\eta_1,\eta_2)+F_\ep^\Omega(v_1,v_2),
\qquad \text{where}
\]
\[
\begin{split}
F_\ep^\Omega(v_1,v_2)=\sum_{j=1}^2\int_{\R^2}\left\{
\frac{\eta_j^2}{2}|\nabla v_j|^2+\frac{g_j}{4\ep^2}\eta_j^4(|v_j|^2-1)^2-\eta_j^2\Omega x^\perp\cdot(iv_j,\nabla v_j) \right\}\, dx \\
+\frac{g}{2\ep^2}
\int_{\R^2}\eta_1^2\eta_2^2(1-|v_1|^2)(1-|v_2|^2)\, dx.
\end{split}
\]
\end{lemma}
We skip the proof since it is similar to the one of Proposition
\ref{lemma:splitting_energy_no_rotation}.
 An integration by parts and  assumption
(\ref{eq:condition_on_g_thomas_fermi}) yield

\begin{lemma}\label{lemma:tilde_F_negative}
Let $F_{i,\ep}$ be the auxiliary functions introduced in
\eqref{eqFiAux} and let $\gamma$ be as in
\eqref{eq:g_mathfrak_definition}. Then we have
\[
\begin{split}
\tilde F_\ep^\Omega(v_1,v_2)=\sum_{i=1}^2 \int_{\R^2}\left\{
\frac{\eta_i^2}{2}\left(|\nabla v_i|^2-4\Omega F_{i,\ep}
Jv_i\right)+ \frac{\gamma}{4\ep^2}\eta_i^4(|v_i|^2-1)^2 \right\} \,
dx\leq 0,
\end{split}
\]
where $Jv_j=(i\partial_{x_1}v_j,\partial_{x_2}v_j)$ stands for the
Jacobian of $v_j$.
\end{lemma}
\begin{proof}
First we prove that we can rewrite $F_\ep^\Omega$  in terms of $F_{i,\ep}$ as follows
\begin{equation}\label{eq:F_ep_Omega}
\begin{split}
F_\ep^\Omega(v_1,v_2)=\sum_{i=1}^2 \int_{\R^2}\left\{
\frac{\eta_i^2}{2}\left(|\nabla v_i|^2-4\Omega F_{i,\ep} Jv_i\right)+ \frac{g_i}{4\ep^2}\eta_i^4(|v_i|^2-1)^2 \right\}\,dx \\
+\frac{g}{2\ep^2} \int_{\R^2}
\eta_1^2\eta_2^2(1-|v_1|^2)(1-|v_2|^2)\, dx.
\end{split}
\end{equation}
Indeed, by \eqref{eq:xi'}, the following holds
\[
\nabla^\perp\xi_i=(-\partial_{x_2}\xi_i,\partial_{x_1}\xi_i) =
-\eta_i^2 x^\perp,
\]
%(note that our definition for $\nabla^\perp$ has the opposite sign
%from that in  some other references such as \cite{pacard-riviere}),
and Stokes theorem yields
\[
\int_{\partial B_R} \xi_j (iv_j,\nabla v_j)^\perp\cdot\nu \, d\sigma
= \int_{B_R} \{ -\xi_j \nabla\times(iv_j,\nabla v_j) + \eta_j^2
x^\perp \cdot (iv_j,\nabla v_j) \}\,dx,
\]
where $\nabla\times(iv_j,\nabla
v_j)=\partial_{x_1}(iv_j,\partial_{x_2}
v_j)-\partial_{x_2}(iv_j,\partial_{x_1} v_j) = 2 Jv_j$. The boundary
term vanishes because, by Corollary \ref{lemma:estimates_f_i}, for $R$ large, we have
\[
\left|\int_{\partial B_R} \xi_j (iv_j,\nabla v_j)^\perp\cdot\nu \,
d\sigma\right|= \left|\int_{\partial B_R} F_{j,\ep}(iu_j,\nabla
u_j)\,d\sigma\right| \leq C\ep^{2/3} \int_{\partial B_R} (|\nabla
u_j|^2+|u_j|^2)\,d\sigma
\]
which vanishes along a sequence $R_k\to\infty$. Hence we have
obtained
\[
\int_{\R^2} \eta_j^2 x^\perp\cdot(i v_j,\nabla
v_j)\,dx=2\int_{\R^2}\eta_j^2F_{j,\ep} J v_j \,dx,
\]
and \eqref{eq:F_ep_Omega} is proved. Then, reasoning as in
\eqref{eq:coercivity}, we deduce that $F_\ep^\Omega(v_1,v_2)\geq
\tilde F_\ep^\Omega(v_1,v_2)$. On the other hand, since $(u_1,u_2)$
is a minimizer and $(\eta_1,\eta_2)$ is real valued, we have
\[
E_\ep^\Omega(u_1,u_2)\leq
E_\ep^\Omega(\eta_1,\eta_2)=E_\ep^0(\eta_1,\eta_2),
\]
which, by Lemma \ref{lemma:splitting_energy}, implies that
$F_\ep^\Omega(v_1,v_2)\leq 0$.
\end{proof} The rest of the section is devoted to proving that $v_i=1$.

Let $0\leq \chi_i \leq 1$ be regular cut-off functions with the
property that
\[
\chi_i(r)=1 \text{ for } r\leq R_{i,0}-2|\log\ep|^{-3/2}
\quad\text{and} \quad \chi_i(r)=0 \text{ for } r\geq
R_{i,0}-|\log\ep|^{-3/2},
\]
and moreover $\|\nabla\chi_i\|_{L^\infty(\R^2)}\leq
2|\log\ep|^{3/2}$. We  estimate $\tilde F_\ep^\Omega(v_1,v_2)$
according to the following splitting
\[
\tilde F_\ep^\Omega(v_1,v_2)=A_1+B_1-C_1+A_2+B_2-C_2,
\]
where
\[
A_i=\int_{\R^2} \chi_i \left\{  \frac{\eta_i^2}{2}|\nabla v_i|^2+
\frac{\gamma}{4\ep^2}\eta_i^4(|v_i|^2-1)^2\right\}\,dx
\]
\[
B_i= \int_{\R^2}(1-\chi_i)\left\{ \frac{\eta_i^2}{2}\left(|\nabla
v_i|^2-4\Omega F_{i,\ep} Jv_i\right)+
\frac{\gamma}{4\ep^2}\eta_i^4(|v_i|^2-1)^2 \right\} \, dx
\]
\[
C_i=2\Omega \int_{\R^2} \chi_i \xi_i Jv_i \,dx.
\]
Lemma \ref{lemma:tilde_F_negative} immediately provides
\begin{equation}\label{eq:A+B<C}
A_1+B_1+A_2+B_2\leq C_1+C_2.
\end{equation}

\begin{proposition}\label{prop:B_i>0}
With the notation above, for $\varepsilon$ small, we have $B_i\geq 0$ for $i=1,2$ so that
$A_1+A_2\leq C_1+C_2$.
\end{proposition}
\begin{proof}
Due to the definition of $B_i$, we can restrict our attention to the
set
\[
\text{supp}(1-\chi_i)=\{ x:\ |x|> R_{i,0}-2|\log\ep|^{-3/2} \}.
\]
Corollary \ref{lemma:estimates_f_i} implies that in such a set we
have $F_{i,\ep}\leq C|\log\ep|^{-3/2}$. Hence assumption
\eqref{eq:Omega_condition1} implies that  $\Omega F_{i,\ep}\leq 1/4$,
for $\ep$ sufficiently small. Recalling that $|J v_i|\leq|\nabla
v_i|^2/2$, we deduce that
\[
|\nabla v_i|^2-4\Omega F_{i,\ep} Jv_i \geq \frac{1}{2} |\nabla
v_i|^2,
\]
and as a consequence,
\begin{equation}\label{eq:B_i_bound_below}
B_i\geq \int_{\R^2}(1-\chi_i)\left\{ \frac{\eta_i^2}{4}|\nabla
v_i|^2+ \frac{\gamma}{4\ep^2}\eta_i^4(|v_i|^2-1)^2 \right\} \, dx.
\end{equation}
The second part of the statement is obtained by combining with
\eqref{eq:A+B<C}.
\end{proof}

\begin{lemma}\label{lemma:tilde_ep_definition}
Let
\[
\tilde\ep_i=\ep\gamma^{-1/2}
\left(\inf_{\{\text{supp}\chi_i\}}\eta_i\right)^{-1}.
\]
There exists $C>0$  such that $\tilde\ep_i \leq C \ep
|\log\ep|^{3/4}$.
\end{lemma}
\begin{proof}
Clearly $a_i\geq c|\log\ep|^{-3/2}$ in $\{\text{supp}\chi_i\}$.
 Hence property \eqref{eq:main_theorem_relation2} implies that
\begin{equation}\label{eq:eta_bound_from_below}
\eta_i\geq\sqrt{a_i}-C\ep^{1/3}\geq c|\log\ep|^{-3/4} \quad \text{in
} \{\text{supp}\chi_i\},
\end{equation}
which provides the statement.
\end{proof}

\begin{lemma}\label{lemma:estimate_on_supp_chi}
There exists $C$ independent of $\ep$ such that, for small $\varepsilon$,
\[
\sum_{i=1}^2 \int_{\{\text{supp}\chi_i\}} \left\{ \frac{|\nabla
v_i|^2}{2}+\frac{1}{4\tilde\ep_i^2}(|v_i|^2-1)^2 \right\}\,dx \leq C
|\log\ep|^{3/2}(C_1+C_2).
\]
\end{lemma}
\begin{proof}
Recalling that $\tilde\ep_i^2\geq\ep^2/(\gamma\eta_i^2)$ in
$\{\text{supp}\chi_i\}$ and relation
\eqref{eq:eta_bound_from_below}, we deduce
\[
\begin{split}
\int_{\{\text{supp}\chi_i\}} & \left\{ \frac{|\nabla v_i|^2}{2}+\frac{1}{4\tilde\ep_i^2}(|v_i|^2-1)^2 \right\} \,dx \\
& \leq C |\log\ep|^{3/2} \int_{\{\text{supp}\chi_i\}} \left\{
\frac{\eta_i^2}{2}|\nabla
v_i|^2+\frac{\gamma}{4\ep^2}\eta_i^4(|v_i|^2-1)^2 \right\}\,dx.
\end{split}
\]
On the other side, estimate \eqref{eq:B_i_bound_below} implies that
\[
\int_{\mathbb{R}^2} \left\{ \frac{\eta_i^2}{2}|\nabla
v_i|^2+\frac{\gamma}{4\ep^2}\eta_i^4(|v_i|^2-1)^2 \right\}\,dx  \leq
A_i + 2 B_i.
\]
The result follows by summing the above for $i=1,2$ and combining
with \eqref{eq:A+B<C}.
\end{proof}

\begin{proposition}\label{prop:estimate_log_ep-11}
Suppose that
\begin{equation}\label{eq:Omega_condition2}
2\Omega\max_{i=1,2} \{\|F_{i,0}\|_{L^\infty(\R^2)}\} \leq
|\log\ep|-(\alpha+1)\log|\log\ep|,
\end{equation}
for a suitable $\alpha>0$, where $F_{i,0}$ are as in
(\ref{eqFiAux000}). There exists $C>0$ independent of $\ep$ such
that
\[
A_i+B_i+|C_i|\leq C|\log\ep|^{-11}, \quad \text{for } i=1,2.
\]
\end{proposition}
\begin{proof}
We use a result by Jerrard \cite{jerard}, as it is stated in
\cite{AftalionJerrardLetelierJFA}. Following the last mentioned
paper, we let
\begin{equation}\label{eq:alpha_k_beta_def}
\alpha=1300, \quad k=1+\alpha\frac{\log|\log\ep|}{|\log\ep|}, \quad
\beta=\frac{k-1}{100}.
\end{equation}
Notice that
\begin{equation}\label{eq:eps_power_beta}
\ep^\beta=|\log\ep|^{-\alpha/100}=|\log\ep|^{-13}.
\end{equation}
As in \cite{AftalionJerrardLetelierJFA}, we can write \cite[Lemma
8]{jerard} as
\[
\begin{split}
\sum_{i=1}^2 |C_i|\leq 2\Omega k \sum_{i=1}^2 \int_{\R^2}
\frac{\chi_i\xi_i}{|\log\tilde\ep_i|}\left\{ \frac{|\nabla
v_i|^2}{2}+ \frac{1}{4\tilde\ep_i^2}(|v_i|^2-1)^2\right\}\,dx
+C\ep^\beta(1+\sum_{i=1}^2 |C_i|),
\end{split}
\]
where $\tilde\ep_i$ is defined in Lemma
\ref{lemma:tilde_ep_definition}. This formulation only makes use of
the estimates in Lemmas \ref{lemma:nabla_xi_estimate} and
\ref{lemma:estimate_on_supp_chi}, so that it holds also in our case.
Now, recalling that $\xi_i=F_{i,\ep}\eta_i^2$ and that
$\tilde\ep_i^2\geq\ep^2/(\gamma\eta_i^2)$ in
$\{\text{supp}\chi_i\}$, we deduce
\[
\sum_{i=1}^2 |C_i| \leq 2\Omega k \sum_{i=1}^2
\frac{\|F_{i,\ep}\|_{L^\infty(\R^2)}}{|\log\tilde\ep_i|} A_i +
C\ep^\beta(1+\sum_{i=1}^2|C_i|),
\]
so that
\[
(1-C\ep^\beta)\sum_{i=1}^2|C_i| \leq 2\Omega k \sum_{i=1}^2
\frac{\|F_{i,\ep}\|_{L^\infty(\R^2)}}{|\log\tilde\ep_i|} A_i
+C\ep^\beta.
\]
We estimated $F_{i,\ep}$ in Lemma \ref{lemma:estimates_f_i-f_i0},
which provides
\[
\|F_{i,\ep}\|_{L^\infty(\R^2)}\leq (1+C\ep^{1/3})
\|F_{i,0}\|_{L^\infty(\R^2)}\leq (1+C\ep^\beta)
\|F_{i,0}\|_{L^\infty(\R^2)},
\]
where the last inequality holds for $\ep$ sufficiently small by
virtue of \eqref{eq:eps_power_beta}. Also, Lemma
\ref{lemma:tilde_ep_definition} implies that, for every $K>0$, we
have
\[
|\log\tilde\ep_i|\geq (|\log\ep|-\log|\log\ep|)(1+K\ep^\beta)
\]
for $\ep$ sufficiently small with respect to $K$. By combining these
facts with assumption \eqref{eq:Omega_condition2} and with our
choice of $k$ in \eqref{eq:alpha_k_beta_def} we obtain
\begin{equation}\label{eq:bound_C_i}
\begin{split}
|C_1|+|C_2| & \leq \left(1-\alpha \frac{\log|\log\ep|}{|\log\ep|-\log|\log\ep|}\right) k (A_1+A_2) + C\ep^\beta \\
& \leq \left( 1-\alpha^2\frac{\log^2|\log\ep|}{|\log\ep|^2} \right)
(A_1+A_2)+C\ep^\beta.
\end{split}
\end{equation}
Recalling Proposition \ref{prop:B_i>0} we deduce
\[
A_1+A_2\leq |C_1|+|C_2|\leq \left(
1-\alpha^2\frac{\log^2|\log\ep|}{|\log\ep|^2} \right) (A_1+A_2)+
C\ep^\beta,
\]
so that
\[
A_1+A_2\leq
\frac{C\ep^\beta}{\alpha^2}\frac{|\log\ep|^2}{\log^2|\log\ep|} \leq
C |\log\ep|^{-11},
\]
where in the last step we replaced relation
\eqref{eq:eps_power_beta}. Being $A_i$ non-negative quantities, the
last estimate holds for both terms. In turn we deduce from
\eqref{eq:bound_C_i} that $|C_i|\leq C |\log\ep|^{-11}$, and from
\eqref{eq:A+B<C} that
\[
B_1+B_2\leq A_1+B_1+A_2+B_2\leq C_1+C_2\leq C |\log\ep|^{-11}.
\]
Being $B_i$ non-negative by Proposition \ref{prop:B_i>0}, the
estimate holds for both terms.
\end{proof}

We can now derive a ``clearing-out'' property (see also
\cite{bethuel}).

\begin{proposition}\label{prop:v>1/2}
Suppose that \eqref{eq:Omega_condition2} holds. For $\ep$
sufficiently small, we have
\[
|v_i|\geq \frac{1}{2} \quad\text{in } \{\text{supp}\chi_i\}.
\]
\end{proposition}
\begin{proof}
We shall prove that
\begin{equation}\label{eq:v>1/2}
|v_i|>1-|\log\ep|^{-1} \quad\text{in } \{\text{supp}\chi_i\},
\end{equation}
for $i=1,2$, which implies the statement. By combining Lemma
\ref{lemma:estimate_on_supp_chi} with Proposition
\ref{prop:estimate_log_ep-11} we obtain
\[
\sum_{i=1}^2 \int_{\{\text{supp}\chi_i\}} \left\{ \frac{|\nabla
v_i|^2}{2}+\frac{1}{4\tilde\ep_i^2}(|v_i|^2-1)^2 \right\}\,dx \leq C
|\log\ep|^{3/2}|\log\ep|^{-11}.
\]
Then Lemma \ref{lemma:tilde_ep_definition} provides
\begin{equation}\label{eq:estimate_on_supp_chi}
\frac{1}{\ep^2} \int_{\{\text{supp}\chi_i\}} (|v_i|^2-1)^2 \,dx \leq
C|\log\ep|^{-8}
\end{equation}
for $i=1,2$. Next we observe that
\begin{equation}\label{eq:estimate_nabla_v_i}
\|\nabla v_i\|_{L^\infty(\{\text{supp}\chi_i\})} \leq
C\frac{|\log\ep|^{3/2}}{\ep}.
\end{equation}
This comes from the fact that $\|\nabla u_i\|_{L^\infty(\R^2)}\leq
C/\ep$, as can be seen by combining Lemmas
\ref{lemma:L_infty_bounds_u} and
\ref{lemma:lagrange_multipliers_mu}, and that $\nabla v_i=\nabla
u_i/\eta_i-u_i\nabla \eta_i/\eta_i^2$, together with estimate
\eqref{eq:eta_bound_from_below}. Suppose by contradiction that
\eqref{eq:v>1/2} does not hold, i.e. there exists $x_0\in
\{\text{supp}\chi_i\}$ such that
\[
|v_i(x_0)|\leq 1-|\log\ep|^{-1} \quad \text{as } \ep\to0.
\]
Then \eqref{eq:estimate_nabla_v_i} implies
\[
|v_i(x)|\leq 1-C|\log\ep|^{-1} \quad \text{in } B_{r_0}(x_0) \quad
\text{with } r_0=\ep|\log\ep|^{-5/2},
\]
so that
\[
\frac{1}{\ep^2} \int_{B_{r_0}(x_0)\cap\{\text{supp}\chi_i\}}
(|v_i|^2-1)^2 \,dx \geq C |\log\ep|^{-7},
\]
which contradicts \eqref{eq:estimate_on_supp_chi} for $\ep$
sufficiently small. Therefore \eqref{eq:v>1/2} is proved.
\end{proof}

\subsection{Proof of Theorem \ref{thm:nonexistence_vortices}}
We are now in position to give the proof of Theorem
\ref{thm:nonexistence_vortices}.

\begin{proof}[Proof of Theorem \ref{thm:nonexistence_vortices}]
We take $\Omega\leq \omega_0|\log\ep|-\omega_1 \log|\log\ep|$ with
$\omega_0,\omega_1$ such that \eqref{eq:Omega_condition2} holds
(recall that $F_{i,0}$ is bounded in $\R^2$). Thanks to the previous
proposition the quantity $w_i=v_i/|v_i|$ is well defined in
$\{\text{supp}\chi_i\}$ and satisfies $Jw_i=0$ (see
\cite{AftalionJerrardLetelierJFA}). Hence, we find that
\[
\begin{split}
C_j&=2\Omega \int_{\{\text{supp}\chi_j\}} \chi_j\xi_j(Jv_j-Jw_j )\,dx \\
&= 2\Omega \int_{\{\text{supp}\chi_j\}} \nabla^\perp(\chi_j\xi_j)
[(iv_j,\nabla v_j)-(iw_j,\nabla w_j)]\,dx.
\end{split}
\]
Writing $v_j=\rho_je^{i\phi_j}$ in $\{\text{supp}\chi_j\}$ we see
that $(iv_j,\nabla v_j)=\rho_j^2\nabla\phi_j$ and $(iw_j,\nabla
w_j)=\nabla\phi_j$, so that Proposition \ref{prop:v>1/2} implies
\[
|(iv_j,\nabla v_j)-(iw_j,\nabla
w_j)|=\frac{|\rho_j^2-1|}{\rho_j}|\rho_j\nabla\phi_j| \leq 2
|\rho_j^2-1||\rho_j\nabla\phi_j|  \leq
2\left||v_j|^2-1\right||\nabla v_j|.
\]
We insert it in the previous estimate to obtain
\[
\begin{split}
C_j&\leq 2\Omega \|\nabla(\chi_j\xi_j)\|_{L^\infty(\{\text{supp}\chi_j\})} \int_{\{\text{supp}\chi_j\}} 2\left||v_j|^2-1\right|\cdot|\nabla v_j|\,dx \\
&\leq 4\sqrt{2} \Omega \|\nabla(\chi_j\xi_j)\|_{L^\infty(\R^2)} \int_{\{\text{supp}\chi_j\}} \left\{ \frac{\tilde\ep_j}{2}|\nabla v_j|^2+\frac{1}{4\tilde\ep_j}\left(|v_j|^2-1\right)^2\right\}\,dx \\
&\leq C \Omega \ep|\log\ep|^{9/4} \int_{\{\text{supp}\chi_j\}}
\left\{ \frac{|\nabla
v_j|^2}{2}+\frac{1}{4\tilde\ep_j^2}\left(|v_j|^2-1\right)^2\right\}\,dx,
\end{split}
\]
where we used Lemma \ref{lemma:tilde_ep_definition} and the estimate
$\|\nabla(\chi_i\xi_i)\|_{L^\infty(\mathbb{R}^2)}\leq
C|\log\ep|^{3/2}$. We sum for $i=1,2$ and then we use the assumption
\eqref{eq:Omega_condition1}, and Lemma
\ref{lemma:estimate_on_supp_chi}, to obtain $C_1+C_2\leq
C\ep|\log\ep|^{19/4}(C_1+C_2)$, so that $C_1+C_2\leq (C_1+C_2)/2$
for $\ep$ sufficiently small. Since $C_1+C_2$ is non-negative by
Proposition \ref{prop:B_i>0}, we conclude that $C_1+C_2=0$. In turn,
relation \eqref{eq:A+B<C} and Proposition \ref{prop:B_i>0} imply
also $A_i=B_i=0$ for $i=1,2$, that is
\[
A_i=\int_{\R^2} \chi_i \left\{  \frac{\eta_i^2}{2}|\nabla v_i|^2+
\frac{\gamma}{4\ep^2}\eta_i^4(|v_i|^2-1)^2\right\}\,dx=0
\]
and (see \eqref{eq:B_i_bound_below})
\[
\int_{\R^2}(1-\chi_i)\left\{ \frac{\eta_i^2}{4}|\nabla v_i|^2+
\frac{\gamma}{4\ep^2}\eta_i^4(|v_i|^2-1)^2 \right\} \, dx=0.
\]
Therefore, we infer that $v_i$ are both constants of modulus 1 as we wanted to
prove.
\end{proof}

\appendix
\section{The scalar ground state}\label{secAppenGS}
Throughout Subsections \ref{subApprox}-\ref{subsecError}, we
have referred to the following
\begin{theorem}\label{thmGSscalar}
Assume that $a\in C^1[0,\infty)$ satisfies $a'(0)=0$, there exist
positive numbers $r_1<r_2<\cdots<r_n$ such that $a(r_i)=0$,
$a(r)\neq 0$ if $r\neq r_i$, and $(-1)^ia'(r_i)>0$, $i=1,\cdots,n$,
and $a(r)\to -\infty$ as $r\to \infty$. Assume also that
$\mu_\varepsilon \in \mathbb{R}$ satisfy $\mu_\varepsilon \to 0$ as
$\varepsilon \to 0$.

Let $A_\varepsilon=a+\mu_\varepsilon$. For sufficiently small
$\varepsilon>0$, by the implicit function theorem, there exist
$0<r_{1,\varepsilon}<r_{2,\varepsilon}<\cdots<r_{n,\varepsilon}$
such that $r_{i,\varepsilon}\to r_i$ as $\varepsilon\to 0$,
satisfying $A_\varepsilon(r_{i,\varepsilon})=0$,
$A_\varepsilon(r)\neq 0$ if $r\neq r_{i,\varepsilon}$, and
$(-1)^iA_\varepsilon'(r_{i,\varepsilon})>0$, $i=1,\cdots,n$.

If $\varepsilon>0$ is sufficiently small, there exists a positive
radially symmetric solution $\eta_\varepsilon\in C^2(\mathbb{R}^2)$
to the problem
\begin{equation}\label{eqGSscalar}
\varepsilon^2 \Delta \eta=\eta\left(\eta^2-A_\varepsilon(x)\right),\
\ x\in \mathbb{R}^2,\ \ \eta(x)\to 0\ \textrm{as}\ |x|\to \infty,
\end{equation}
such that
\begin{equation}\label{eqGSunifglob}
\|\eta_\varepsilon-\sqrt{A_\varepsilon^+}\|_{L^\infty(\mathbb{R}^2)}\leq
C\varepsilon^\frac{1}{3},
\end{equation}
and
\begin{equation}\label{eqGSpotlower}
3\eta_\varepsilon^2-A_\varepsilon\geq \left\{\begin{array}{ll}
                                                              c |r-r_{i,\varepsilon}|+c\varepsilon^\frac{2}{3}, &
                                                              \textrm{if}\ |r-r_{i,\varepsilon}|\leq
                                                               \delta, \\
                                                                 &   \\
                                                               c, &
                                                               \textrm{otherwise},
                                                             \end{array}
\right.
\end{equation}
for some $\delta \in
\left(0,\frac{1}{4}\min_{i=1,\cdots,n-1}\{r_{i+1}-r_i\} \right)$.
 More precisely,  we
have
\begin{equation}\label{eqeta1groundinner}
\eta_\varepsilon(r)=\varepsilon^\frac{1}{3}(-1)^{i+1}\beta_{i,\varepsilon}
V\left(\beta_{i,\varepsilon}
\frac{r-r_{i,\varepsilon}}{\varepsilon^\frac{2}{3}}
\right)+\left\{\begin{array}{ll}
                 \mathcal{O}\left(\varepsilon+|r-r_{i,\varepsilon}|^\frac{3}{2}\right) & \textrm{if}\ 0\leq (-1)^i(r-r_{i,\varepsilon})\leq \delta, \\
                   &   \\
                 \mathcal{O}(\varepsilon) \exp\left\{-c\frac{|r-r_{i,\varepsilon}|}{\varepsilon^\frac{2}{3}} \right\} &
                 \textrm{if} \ -\delta \leq (-1)^i(r-r_{i,\varepsilon})\leq
                 0,
               \end{array}
\right.
\end{equation}
where
\[
\beta_{i,\varepsilon}^3=-a'(r_{i,\varepsilon}),\ \ i=1,\cdots,n,
\]
and $V$ is the Hastings-McLeod solution, as described in
(\ref{eqpainleve}). Estimate (\ref{eqeta1groundinner}) can be
differentiated once to give
\begin{equation}\label{eqeta1groundinner2bis}
\eta_\varepsilon'(r)=\varepsilon^{-\frac{1}{3}}(-1)^{i+1}\beta_{i,\varepsilon}^2
V'\left(\beta_{i,\varepsilon}
\frac{r-r_{i,\varepsilon}}{\varepsilon^\frac{2}{3}}
\right)+\mathcal{O}\left(\varepsilon^\frac{1}{3}+|r-r_{i,\varepsilon}|^\frac{1}{2}\right)\
\ \textrm{if}\ |r-r_{i,\varepsilon}|\leq \delta,
\end{equation}
uniformly, as $\varepsilon\to 0$. On the other side, we have
\begin{equation}\label{eqeta1groundouter}
\eta_\varepsilon
(r)-\sqrt{A_\varepsilon(r)}=\varepsilon^2\mathcal{O}(|r-r_{i,\varepsilon}|^{-\frac{5}{2}})\
\ \textrm{if}\ \ C\varepsilon^\frac{2}{3}\leq
(-1)^i(r-r_{i,\varepsilon})\leq \delta,
\end{equation}
uniformly, as $\varepsilon\to 0$. Furthermore,
\begin{equation}\label{eqGSunifAway}
\left| \eta_\varepsilon-\sqrt{A_\varepsilon^+}\right|\leq
C\varepsilon^2\ \ \textrm{in}\ I_\delta\equiv
[0,r_{1,\varepsilon}-\delta]\cup
[r_{1,\varepsilon}+\delta,r_{2,\varepsilon}-\delta]\cup \cdots \cup
[r_{n,\varepsilon}+\delta,\infty),
\end{equation}
and
\begin{equation}\label{eqGSscalarDecay}
\eta_\varepsilon(r)\leq C
\varepsilon^\frac{1}{3}\exp\left\{-c\varepsilon^{-\frac{2}{3}}\min_{i=1,\cdots,n}|r-r_{i,\varepsilon}|\right\}\
\ \textrm{if}\  \ A_\varepsilon^+(r)=0.
\end{equation}
Moreover, if $a(x)=a(|x|)\in C^4(\mathbb{R}^2)$,
\begin{equation}\label{eqeta1groundouter'}
\eta'_\varepsilon-\left(\sqrt{A_\varepsilon}\right)'=\varepsilon^2\mathcal{O}(|r-r_{i,\varepsilon}|^{-\frac{7}{2}}),
\end{equation}
and
\begin{equation}\label{eqeta1groundouter''}
\Delta
\eta_\varepsilon-\Delta\left(\sqrt{A_\varepsilon}\right)=\varepsilon^2\mathcal{O}(|r-r_{i,\varepsilon}|^{-\frac{9}{2}}),\
\ \textrm{if}\ \ C\varepsilon^\frac{2}{3}\leq
(-1)^i(r-r_{i,\varepsilon})\leq \delta,
\end{equation}
uniformly, as $\varepsilon\to 0$, and
\begin{equation}\label{eqGSscalarBootstrap}
\|\eta_\varepsilon-\sqrt{A_\varepsilon^+}\|_{C^2(I_\delta)}\leq
C\varepsilon^2.
\end{equation}
\end{theorem}
\begin{proof}
All the assertions up to (\ref{eqGSscalarDecay}) are essentially
contained in \cite[Thm. 1.1]{KaraliSourdisGround}, where in fact no
radial symmetry is imposed on $a(\cdot)$. Actually, relation
(\ref{eqeta1groundinner2bis})  can be proven by combining the proof
of Corollary 4.1 in \cite{KaraliSourdisGround} with relation (3.40)
therein. In passing, we note that $V'(s)<0$, $s\in \mathbb{R}$.

Let us further assume that $a(x)=a(|x|)\in C^4(\mathbb{R}^2)$. In
order to establish relations
(\ref{eqeta1groundouter'})-(\ref{eqeta1groundouter''}), we need
a refinement of (\ref{eqeta1groundouter}). Motivated from the
identity
\[
\varepsilon^2 \Delta
\left(\eta_\varepsilon-\sqrt{A_\varepsilon}\right)-
\eta_\varepsilon\left(\eta_\varepsilon+\sqrt{A_\varepsilon}
\right)\left(\eta_\varepsilon-\sqrt{A_\varepsilon}
\right)=-\varepsilon^2 \Delta \left(\sqrt{A_\varepsilon}\right)
\]
if $C\varepsilon^\frac{2}{3}\leq (-1)^i(r-r_{i,\varepsilon})\leq
\delta$ (see also \cite[Prop. 2.1]{IgnatMillotJFA}), we let
\begin{equation}\label{eqerroreta1fluctuation}
\eta_\varepsilon-\sqrt{A_\varepsilon}=\varepsilon^2 \frac{\Delta
\left(\sqrt{A_\varepsilon}\right)}{2A_\varepsilon}+\phi\ \
\textrm{if}\ \ C\varepsilon^\frac{2}{3}\leq
(-1)^i(r-r_{i,\varepsilon})\leq \delta,
\end{equation}
for some fluctuation function $\phi$. Pushing further the analysis
in \cite[Thm. 2.1]{dancer-lazer} or \cite[Thm. 1.1]{liyanedinburg},
it can be shown that
\begin{equation}\label{eqerrorDanceYan}
|\phi(r)|\leq C\varepsilon^4\ \ \textrm{if}\ \delta \leq
(-1)^i(r-r_{i,\varepsilon})\leq 2\delta.
\end{equation}
 Making use of
(\ref{eqeta1groundouter}), and recalling that
\begin{equation}\label{eqa1eps'}A_\varepsilon'(r_{i,\varepsilon})=a'(r_{i,\varepsilon})\to -c_i<0\ \
\textrm{as}\ \ \varepsilon\to 0,\end{equation} it follows readily
that
\[
\varepsilon^2 \Delta \phi-\eta_\varepsilon\left(\eta_\varepsilon
+\sqrt{A_\varepsilon}\right)\phi=\varepsilon^4
\mathcal{O}\left(|r-r_{i,\varepsilon}|^{-\frac{9}{2}} \right) \
\textrm{if}\ \  C\varepsilon^\frac{2}{3}\leq
(-1)^i(r-r_{i,\varepsilon})\leq \delta,
\]
uniformly, as $\varepsilon\to 0$. Since
\[\eta_{\varepsilon}\left(\eta_\varepsilon+\sqrt{A_\varepsilon}
\right)\geq c |r-r_{i,\varepsilon}|\ \ \textrm{if}\ \
C\varepsilon^\frac{2}{3}\leq (-1)^i(r-r_{i,\varepsilon})\leq
\delta,\] (from (\ref{eqeta1groundouter}) and (\ref{eqa1eps'})), a
standard comparison argument yields that
\[
|\phi(r)|\leq C\varepsilon^4 |r-r_{i,\varepsilon}|^{-\frac{11}{2}},\
\ C\varepsilon^\frac{2}{3}\leq (-1)^i(r-r_{i,\varepsilon})\leq
\delta,
\]
where we have also used that
\[
\left|\phi\left(r_{i,\varepsilon}+(-1)^i\delta\right)\right|\leq
C\varepsilon^4\ \ \textrm{and}\ \
\left|\phi\left(r_{i,\varepsilon}+(-1)^iC\varepsilon^\frac{2}{3}\right)\right|\leq
C\varepsilon^\frac{1}{3},\] which follow from
(\ref{eqerrorDanceYan}) and (\ref{eqeta1groundouter}) respectively;
one plainly uses barriers of the form $\pm M
\varepsilon^4|r-r_{i,\varepsilon}|^{-\frac{11}{2}}$ with $M$ chosen
sufficiently large, see also \cite[Lem. 2.1]{GalloPelinovskyAA} or
\cite[Lem. 3.10]{karaliSourdisPoincare} for related arguments when
the problem is independent of $\varepsilon$ (for the present
argument to work it is crucial that $|r-r_{i,\varepsilon}|\geq
\varepsilon^\frac{2}{3}$). Consequently, recalling
(\ref{eqerroreta1fluctuation}), we have shown the following
refinement of (\ref{eqeta1groundouter}):
\[
\eta_\varepsilon-\sqrt{A_{\varepsilon}}=\varepsilon^2 \frac{\Delta
\left(\sqrt{A_\varepsilon}\right)}{2A_\varepsilon}+\varepsilon^4\mathcal{O}\left(|r-r_{i,\varepsilon}|^{-\frac{11}{2}}\right),
\]
uniformly if $C\varepsilon^\frac{2}{3}\leq
(-1)^i(r-r_{i,\varepsilon})\leq \delta$, as $\varepsilon\to 0$,
which complements (\ref{eqerrorDanceYan}). In turn, via
(\ref{eqGSscalar}) and some straightforward calculations, this can
be shown to imply (\ref{eqeta1groundouter''}). Equivalently, we have
that
\[
\left(r(\eta_\varepsilon-\sqrt{A_\varepsilon})'
\right)'=\varepsilon^2\mathcal{O}\left(|r-r_{i,\varepsilon}|^{-\frac{9}{2}}\right),
\]
if $C\varepsilon^\frac{2}{3}\leq (-1)^i(r-r_{i,\varepsilon})\leq
\delta$. Integrating the above identity from
$r_{i,\varepsilon}+(-1)^i\delta$ to
$r_{i,\varepsilon}+(-1)^iC\varepsilon^\frac{2}{3}$, and using that
$\left(\eta_\varepsilon-\sqrt{A_\varepsilon}\right)'\left(r_{i,\varepsilon}+(-1)^i\delta\right)=\mathcal{O}(\varepsilon^2)$
as $\varepsilon\to 0$ (from \cite[Prop. 2.1]{IgnatMillotJFA}), we
arrive at (\ref{eqeta1groundouter'}). Finally, relation
(\ref{eqGSscalarBootstrap}) is shown in \cite[Prop.
2.1]{IgnatMillotJFA} to hold in  the $C^1$-topology but their proof
carries over to yield the same estimate in $C^m$, $m\geq 2$, via a
standard bootstrap argument (as in \cite[Thm. 1]{bethuelCVPDE}),
provided that the coefficients in the equation are sufficiently
smooth.\end{proof}

\section{Proof of the technical estimate (\ref{eqaprioriWeight}) in Proposition
\ref{prounifAlg}}\label{ApAlg} Here we present the
\begin{proof}[Proof of (\ref{eqaprioriWeight})] Suppose that $\phi$ satisfies
(\ref{eqAlgLph=h}) for some $h\in C(\bar{I})$.

Firstly, we  establish (\ref{eqAlgUnifK}). Let
\[
\Phi=\rho^\frac{3}{2}\phi.
\]
It is easy to see that $\Phi$ satisfies
\begin{equation}\label{eqAlgApPhiEq}
-\varepsilon^2\Phi_{rr}-\varepsilon^2\left(3\rho^{-1}+\frac{1}{r}
\right)\Phi_r+Q(r)\Phi=\rho^\frac{3}{2}h,\ \ r\in I;\ \  \Phi=0\
\textrm{on} \ \partial I,
\end{equation}
 where
\[
Q(r)=\left(g_1-\frac{g^2}{g_2}
\right)(3\hat{\eta}_{1,\varepsilon}^2-a_{1,\varepsilon})-\frac{15}{4}\varepsilon^2\rho^{-2}-\frac{3}{2r}\varepsilon^2\rho^{-1}.
\]
Observe that, thanks to the lower bound in (\ref{eqAlgLowerk}), we
have
\begin{equation}\label{eqAlgApQlo}
Q(r)\geq \rho
\left(k-\frac{15}{4}\varepsilon^2\rho^{-3}-K\varepsilon^2\rho^{-2}
\right)\geq \rho(k-KD_j^{-3})\geq k \rho,
\end{equation}
provided that $D_j$ is sufficiently large. We may assume, without
loss of generality, that $\Phi,h\geq 0$ (by writing $h=h^+-h^-$ if
necessary). If $\Phi$ attains its maximum value at a point $r_0\in
I$, then $\Phi_{rr}(r_0)\leq 0$ and $\Phi_r(r_0)=0$. So, letting
$\rho_0=R_{1,\varepsilon}-r_0$, via (\ref{eqAlgApPhiEq}) and
(\ref{eqAlgApQlo}), we obtain that
\[
k\rho_0 \Phi(r_0)\leq \rho_0^\frac{3}{2} h(r_0),
\]
i.e., $\Phi(r_0)\leq K\|\rho^\frac{1}{2} h\|_{L^\infty(I)}$ which
clearly implies the validity of (\ref{eqAlgUnifK}).

By (\ref{eqAlgLph=h}), the upper bound in (\ref{eqAlgLowerk}), and
(\ref{eqAlgUnifK}),  we find that
\begin{equation}\label{eqAlgAp0}
\varepsilon^2\|\rho^\frac{1}{2}\Delta \phi\|_{L^\infty(I)}\leq K
\|\rho^\frac{1}{2}h\|_{L^\infty(I)}.
\end{equation}
From this, we  derive a pointwise estimate for $\phi_r$ by
making use of the identity
\begin{equation}\label{eqAlgAp1}
r\phi_r(r)-r_0\phi_r(r_0)=\int_{r_0}^{r}s\Delta \phi ds,\ \ \forall\
r_0,r\in I.
\end{equation}
We can choose $r_0\in
(R_{1,\varepsilon}-2D_j\varepsilon^\frac{2}{3},R_{1,\varepsilon}-D_j\varepsilon^\frac{2}{3})$
such that
\[
\phi_r(r_0)=\frac{\phi(R_{1,\varepsilon}-D_j\varepsilon^\frac{2}{3})-\phi(R_{1,\varepsilon}-2D_j\varepsilon^\frac{2}{3})}{D_j\varepsilon^\frac{2}{3}}.
\]
It follows from (\ref{eqAlgUnifK}) that \[\left|
\phi_r(r_0)\right|\leq
K\varepsilon^{-\frac{5}{3}}\|\rho^\frac{1}{2}h\|_{L^\infty(I)}.\] In
turn, via (\ref{eqAlgAp0}) and (\ref{eqAlgAp1}), we get that
\[
\begin{array}{rcl}
  \left|\phi_r(r)\right| & \leq  & K\varepsilon^{-\frac{5}{3}}\|\rho^\frac{1}{2}h\|_{L^\infty(I)}
  +\|\rho^\frac{1}{2}\Delta
  \phi\|_{L^\infty(I)}\left|\int_{r_0}^{r}(R_{1,\varepsilon}-s)^{-\frac{1}{2}}ds\right|
   \\
   &  &  \\
    &  \leq &
    K\varepsilon^{-\frac{5}{3}}\|\rho^\frac{1}{2}h\|_{L^\infty(I)}+K\varepsilon^{-2}\|\rho^\frac{1}{2}h\|_{L^\infty(I)}
    \left|\rho_0^\frac{1}{2}-\rho^\frac{1}{2} \right|,
\end{array}
\]
$r\in I$. Hence, since $\rho\geq D_j \varepsilon^\frac{2}{3}$ and
$D_j \varepsilon^\frac{2}{3} \leq \rho_0\leq 2D_j
\varepsilon^\frac{2}{3}$, we infer that
\[
\rho^{-\frac{1}{2}}\left|\phi_r(r) \right|\leq
K\varepsilon^{-2}\|\rho^\frac{1}{2}h\|_{L^\infty(I)},\ \ r\in I.
\]
Now, the desired estimate (\ref{eqaprioriWeight}) follows by
combining (\ref{eqAlgUnifK}), (\ref{eqAlgAp0}) and the above
relation.
\end{proof}

\end{document}